\documentclass[12pt]{amsart}
\usepackage{cancel}
\usepackage{amssymb}
\usepackage{hyperref}
\usepackage{enumitem}
\usepackage{amsmath}
\usepackage{latexsym}
\usepackage{extarrows}
\usepackage{enumerate}
\usepackage{txfonts}
\usepackage{mathtools}
\usepackage{bm}
\usepackage{tikz-cd}
\usepackage{xcolor}
\usepackage{quiver}
\usepackage{comment}
\usepackage{graphics}
\usepackage{float}
\usepackage{relsize}
\usepackage{soul}
\makeatletter
\@namedef{subjclassname@2020}{%
  \textup{2020} Mathematics Subject Classification}
\makeatother
\usepackage[T1]{fontenc}

\hypersetup{
    hidelinks,
    linktoc=page,
    }

\counterwithin{figure}{section}

\theoremstyle{plain}

\newtheorem{theorem}{Theorem}[section]
\newtheorem{proposition}[theorem]{Proposition}

\newtheorem{lemma}[theorem]{Lemma}
\newtheorem{corollary}[theorem]{Corollary}

\theoremstyle{definition}
\newtheorem{definition}[theorem]{Definition}

\newtheorem{remark}[theorem]{Remark}

\newtheorem{example}[theorem]{Example}

\newtheorem{question}[subsection]{Question}

\newcommand\E{\mathbb{E}}
\newcommand\Z{\mathbb{Z}}
\newcommand\R{\mathbb{R}}

\newcommand\T{\mathbb{T}}
\newcommand\C{\mathbb{C}}

\newcommand\N{\mathbb{N}}
\newcommand\F{\mathbb{F}}

\newcommand\X{\mathcal{X}}
\newcommand\Y{\mathcal{Y}}

\newcommand\M{\mathcal{M}}

\newcommand\ZZ{\mathrm{Z}}
\newcommand\YY{\mathrm{Y}}
\newcommand\XX{\mathrm{X}}
\newcommand\WW{\mathrm{W}}

\newcommand\blue{\color{blue}}

\newcommand\red{\color{red}}

\definecolor{revcrimson}{RGB}{220, 20, 60}
\definecolor{revcerulean}{RGB}{0, 130, 220}

\newcommand\Hom{\operatorname{Hom}}
\newcommand\Aut{\operatorname{Aut}}

\newcommand\id{\operatorname{id}}
\newcommand\eps{\varepsilon}

\newcommand\ML{\mathrm{ML}}
\newcommand\SML{\mathrm{SML}}
\newcommand\Spec{\mathrm{Spec}}
\newcommand\Poly{\mathrm{Poly}}
\newcommand\genprod{{\ast}}
\DeclareMathOperator*{\biggenprod}{{\mathlarger{\mathlarger{\mathlarger{\ast}}}}}
\newcommand{\mG}{{\mathcal{G}}}

\newcommand{\HK}{\mathrm{HK}}

\begin{document}

\baselineskip=17pt

\title[Totally disconnected Host--Kra and inverse Gowers]{The structure of totally disconnected Host--Kra factors, and the inverse theorem for the $U^k$ Gowers uniformity norms on finite abelian groups of bounded torsion}

\author[A. Jamneshan]{Asgar Jamneshan}
\address{Institute of Mathematics \\ University of Bonn \\ 53113 Bonn, Germany}
\email{ajamnesh@math.uni-bonn.de}

\author[O. Shalom]{Or Shalom}
\address{Department of Mathematics\\ Bar Ilan University \\ 
Ramat Gan \\
5290002, Israel}
\email{Or.Shalom@math.biu.ac.il}

\author[T. Tao]{Terence Tao}
\address{Department of Mathematics\\ University of California \\ 
Los Angeles \\
CA 90095-1555, USA}
\email{tao@math.ucla.edu}

\date{\today}

\begin{abstract} 
Let $\Gamma$ be a countable abelian group, let $k\geq 1$, and let $\XX=(X,\X,\mu,T)$ be an ergodic $\Gamma$-system of order $k$ in the sense of Host--Kra. The $\Gamma$-system $\XX$ is said to be totally disconnected if all its structure groups are totally disconnected. We show that any totally disconnected $\Gamma$-system of order $k$ is a generalized factor of a $\Z^\omega$-system with the structure of a Weyl system. 
As a consequence of this structure theorem, we show that totally disconnected $\Gamma$-systems of order $k$ are represented by translations on double cosets of nilpotent Polish groups. By a correspondence principle of two of us, we can use this representation to establish a (weak) inverse theorem for the $U^k$ Gowers uniformity norms on finite abelian groups of bounded torsion.   
\end{abstract}



\maketitle

\section{Introduction}

\subsection{Types of measure-preserving systems}

In this paper we investigate the relationships between different types of measure-preserving systems $\XX$, focusing in particular on ``systems of order $k$'' for some integer $k \geq 1$, where the acting group $\Gamma$ is either bounded torsion or torsion-free.  Our results are particularly strong when the system $\XX$ is ``totally disconnected'' and enjoys some additional ``divisibility'' conditions.

\begin{figure}
\centering
 \[\begin{tikzcd}
	{\substack{\text{totally disc.} \\ \text{order } k}} && {\substack{\text{totally disc.}\\ \Gamma \text{ torsion-free} \\ \text{order }k}} &&{\substack{\text{totally disc.} \\ k-\text{divisible} \\ \Gamma \text{ torsion-free} \\ \text{order }k}} && {\substack{\text{totally disc.}\\\infty-\text{divisible} \\\Gamma \text{ torsion-free}\\\mathrm{Weyl}_k}} \\
 \\ 	&& {\substack{\text{double coset}\\\text{system}\\ \text{degree } k}} && {\substack{\text{translational}\\\text{system} \\ \text{degree } k}} \\
	\\
	{\substack{U_1,\dots,U_k\ m-\text{torsion} \\ \Gamma\ m-\text{torsion} \\ \text{order }k}} && {\text{order }k} && {\mathrm{Weyl}_k} \\
	\\
	{\substack{U_1,\dots,U_k\ p^r-\text{torsion} \\ \Gamma\ p^r-\text{torsion} \\ \text{order }k}} && {\mathrm{Abramov}_{C(k,p)} } && {\mathrm{Abramov}_k}
	\arrow["{\substack{{\blue \Gamma = {\mathbb F}_p^\omega,\ k \leq p-1} \\ {\blue \text{OR } k=1}}}"', swap, from=5-3, to=5-5]
	\arrow["{\substack{\blue \Gamma = {\mathbb F}_p^\omega}}", from=5-3, to=7-3]
	\arrow["{\substack{{\blue \Gamma = {\mathbb F}_p^\omega} \\ {\blue k \leq p+1}}}"', shift right=1, from=5-3, to=7-5]
	\arrow[shift right=2, from=5-5, to=7-5]
	\arrow[from=7-5, to=7-3]
	\arrow["{\substack{{\red \text{extend } X}\\ \text{Thm. } \ref{divis}}}", shift left=1, Rightarrow, from=1-3, to=1-5]
	\arrow["{\substack{{\red \text{extend }\Gamma \text{ to } \Z^\omega}}}", shift left=1, from=1-1, to=1-3]
	\arrow[shift right=2, from=7-5, to=5-3]
	\arrow[from=1-7, to=5-5]
	\arrow["{\substack{\text{Thm. \ref{infty-weyl}} }}",shift left=1, Rightarrow, from=1-5, to=1-7]
	\arrow["{\substack{{\blue k-\text{divisible}}\\{\blue Z^{k-1}(X)\ \mathrm{Weyl}_{k-1}} \\ \text{Thm. \ref{abramov-weyl}} }}"', shift right=2, Rightarrow, from=7-5, to=5-5]
	\arrow["{\substack{{\blue \Gamma\ m-\text{ torsion }}\\ \text{Thm. \ref{bounded-tor}} }}", swap, Rightarrow, from=5-3, to=5-1]
	\arrow[from=5-1, to=1-1]
	\arrow["{\substack{\text{Thm. \ref{structure-thms}(ii)} }}", Rightarrow, from=1-3, to=3-3]
	\arrow["{\substack{\text{Thm. \ref{structure-thms}(i)} }}", Rightarrow, from=5-5, to=3-5]
	\arrow[shift left=1,from=3-5, to=3-3] 
    \arrow["{\substack{{\red \text{extend } X}}}", shift left=1, from=3-3, to=3-5]
    \arrow[shift left=1,from=1-7, to=1-5] 
    \arrow[shift left=1,from=1-5, to=1-3] 
    \arrow[shift left=1,from=1-3, to=1-1] 
    \arrow[from=3-3, to=5-3]
	\arrow["{\substack{{\red \text{split } \Gamma,\XX \text{ into } p-\text{Sylow factors}}\\ \text{Thm. \ref{pSylow}} }}", Rightarrow, shift left=1,from=5-1, to=7-1]
    \arrow[shift left=2,from=7-1, to=5-1]
\end{tikzcd}\]
    \caption{A schematic depiction of the key implications between various properties of ergodic $\Gamma$-systems $\XX$.  Here $k,m \geq 1$ are integers and $p$ is prime.  Thin arrows are implications which are either easy or established in previous literature; thick arrows are non-trivial and are established here. Some implications require additional hypotheses on $\Gamma, \XX, k$ as indicated (in {\blue blue}) next to the arrows; others (indicated in {\red red}) modify the system by replacing either $\Gamma$ or $\XX$ with an extension, or by decomposing $\Gamma$ and $\XX$ into their $p$-Sylow components.}
  \label{fig:main}
\end{figure} 

These implications are somewhat difficult to summarize succinctly: see Figure \ref{fig:main}.  In order to state these results more precisely, we need to review a certain amount of notation.  We shall give informal descriptions of the key concepts here, and defer precise definitions to Appendix \ref{host-kra-theory}.

Throughout this paper, $\Gamma = (\Gamma,+)$ will denote a countable discrete abelian group.  We distinguish some special types of such groups:
\begin{itemize}
\item \emph{torsion-free groups} $\Gamma$, in which $n \gamma \neq 0$ whenever $\gamma \in \Gamma$ and $n \in \Z$ are non-zero; and
\item \emph{$m$-torsion groups} $\Gamma$ for some $m \geq 1$, in which $m \gamma = 0$ for all $\gamma \in \Gamma$.
\end{itemize}

A key example of an $m$-torsion group is the infinite-dimensional vector space $\F_p^\omega = \bigoplus_{n=1}^\infty \F_p$ over a finite field of prime order $p$, in which case we can take $m=p$.  A key example of a torsion-free group is the free abelian group $\Z^\omega = \bigoplus_{n=1}^\infty \Z$ on countably many generators.  We observe that every countable discrete abelian group $\Gamma$ is isomorphic to a quotient of $\Z^\omega$, since $\Gamma$ is countably generated and thus the image of $\Z^\omega$ under a suitable homomorphism. In other words, $\Z^\omega$ serves as a ``universal'' group for the class of countable abelian groups $\Gamma$. 

Given $\Gamma$, we can introduce the notion of a \emph{$\Gamma$-system} $\XX = (X, \X, \mu, T)$ which roughly speaking is a probability space $(X,\X,\mu)$ equipped with a measure-preserving action\footnote{Strictly speaking, for minor technical reasons one should work with near-actions rather than actions: see Definition \ref{system-def}.  However, in this ``countable'' setting the distinction between actions and near-actions will not be of critical importance.} $T$ of $\Gamma$.  To avoid having to deal with the subtleties of ``uncountable'' ergodic theory \cite{jamneshan2019fz,jt19, jt20, jt-foundational}, we will insist that our probability spaces are standard Lebesgue, in order to make all probability algebras separable (and all structure groups metrizable, and all groups of transformations Polish).  There is an ordering $\leq$ on $\Gamma$-systems (up to a certain type of isomorphism), with the relation $\YY \leq \XX$ if $\YY$ is a \emph{factor} of $\XX$, or equivalently $\XX$ is an \emph{extension} of $\YY$; we give the formal definitions in Definition \ref{system-def}, but roughly speaking this means that there is a factor map $\pi \colon \XX \to \YY$ that pushes forward the measure on $\XX$ to the measure on $\YY$ and which is equivariant with respect to the $\Gamma$-actions.  We will also work with \emph{generalized} extensions (resp. factors) where we extend (resp. quotient) the group $\Gamma$ as well as the system $\XX$.

We isolate several notable classes of $\Gamma$-systems, namely the \emph{translational systems} $G/\Lambda$ (with $\Lambda$ a closed cocompact subgroup of $G$ and $\Gamma$ acting via a homomorphism to $G$), which contain in particular the \emph{rotational systems} $Z$  (where $G$ is abelian) and the \emph{nilsystems} (where $G$ is a nilpotent Lie group and $\Lambda$ is a sublattice) as special cases, and the more general \emph{double coset systems} $K \backslash G / \Lambda$ (where $K$ is a closed subgroup of $G$ normalized by the action of $\Gamma$).  We refer to Definition \ref{translational-def} for the formal definitions.

Every $\Gamma$-system $\XX$ gives rise to a canonical sequence of \emph{Host--Kra factors}\footnote{The factor $\ZZ^k(\XX)$ is sometimes denoted $\ZZ^{<k+1}(\XX)$ in the literature.}
$$ \ZZ^0(\XX) \leq \ZZ^1(\XX) \leq \ZZ^2(\XX) \leq \dots \leq \XX;$$
we review the formal definitions in Appendix \ref{host-kra-theory}. The lowest factor $\ZZ^0(\XX)$ is the invariant factor of $\XX$, which is trivial precisely when the $\Gamma$-system $\XX$ is ergodic.  In fact we will restrict attention almost exclusively in this paper to ergodic systems; in principle one can use the ergodic decomposition to then extend our results to non-ergodic settings, but we will not attempt to do so here.  The factor $\ZZ^1(\XX)$ is known as the \emph{Kronecker factor}, and is generated by the eigenfunctions of $\XX$; the factor $\ZZ^2(\XX)$ is also known as the \emph{Conze--Lesigne factor}. Roughly speaking, the $k^\mathrm{th}$ Host--Kra factor $\ZZ^k(\XX)$ controls the distribution of $k+1$-dimensional cubes in $\XX$, and plays an important role in understanding other patterns in $\XX$, such as $k$-term arithmetic progressions.

An ergodic $\Gamma$-system $\XX$ is said to be \emph{of order $k$} if $\ZZ^k(\XX)$ is equal to $\XX$. One can think of $\ZZ^k(\XX)$ as the maximal order $k$ factor of $\XX$, and any $\Gamma$-system of order $k$ is automatically of order $k'$ for any $k' \geq k$; see Proposition \ref{prop-functoriality}.  Ergodic $\Gamma$-systems of order $k$ have significant algebraic structure; for instance, ergodic $\Gamma$-systems of order $1$ (also known as \emph{Kronecker systems}) are (up to isomorphism) precisely the ergodic rotational $\Gamma$-systems on a compact abelian metrizable group $U$ (see e.g., \cite[Theorem 1.3]{jst}).  More generally, the Host--Kra factors are connected to each other by the relation
$$ \ZZ^k(\XX) = \ZZ^{k-1}(\XX) \rtimes_{\rho_{k-1}} U_k$$
(up to isomorphism) for any $k \geq 1$, where $U_k$ is a compact metrizable abelian group, $\rho_{k-1} \colon \Gamma \to \M(\ZZ^{k-1}(\XX),U_k)$ is a $U_k$-valued cocycle on $\ZZ^{k-1}(\XX)$ of type $k$, and $\rtimes$ denotes the (dynamical) skew-product construction; see Definitions \ref{coh:def}, \ref{type:def} and Proposition \ref{abelext} for the precise statements.  Thus we can express an ergodic $\Gamma$-system $\XX$ of order $k$ (up to isomorphism) as an iterated tower of skew-products
\begin{align}
    \label{structuregroups}
\XX &= U_1 \rtimes_{\rho_1}U_2 \ldots \rtimes_{\rho_{k-1}} U_{k}
\end{align}
for some compact abelian \emph{structure groups} $U_1,\dots,U_k$ and $U_i$-valued cocycles $\rho_i$ of type $i+1$, where the skew-product operation $\rtimes$ is performed from left to right, and $U_1$ has the structure of a rotational $\Gamma$-system.

By adapting the arguments in \cite[Chapter 12]{hk-book} (see also \cite[\S 2]{gmv}), one can show that ergodic translational $\Gamma$-systems $\mG/\Lambda$ of nilpotency class $k$ are of order $k$; taking quotients using Proposition \ref{prop-functoriality}(i), we conclude in particular that ergodic double coset systems $K \backslash \mG/\Lambda$ of nilpotency class $k$ are also of order $k$. The major theme in Host--Kra structure theory is the difficult converse direction, establishing that an ergodic $\Gamma$-system of order $k$ is isomorphic to (an inverse limit of) translational $\Gamma$-systems or a double coset $\Gamma$-system of nilpotency class $k$. This ergodic-theoretic inverse question is currently open for arbitrary $\Gamma$ and $k\geq 3$; see the end of this subsection for a discussion of known results, and see the next subsection for the contributions of this paper.   

We say that an ergodic $\Gamma$-system $\XX$ of order $k$ is \emph{totally disconnected} if each of the structure groups $U_1,\dots,U_k$ are totally disconnected.  Our results will be strongest in the totally disconnected case, which as we shall see in Theorem \ref{bounded-tor} below arises when $\Gamma$ is bounded torsion (i.e., $m$-torsion for some $m$).  In these cases, the structure of order $k$ systems $\XX$ is related to the polynomials on $\XX$, which we now pause to define.

\begin{definition}[Polynomials]
Let $k\geq 0$, let $\XX=(X,\X,\mu,T)$ be a $\Gamma$-system, and let $U=(U,+)$ be a compact metrizable abelian group. 
\begin{itemize}
\item We use $\M(\XX,U)$ to denote the abelian group of measurable functions from $X$ to $U$, up to almost everywhere equivalence, under the pointwise addition.  When $U=\T$ is the unit circle $\T \coloneqq \R/\Z$, we can give this space the $L^2$ metric $d_{L^2}(f,g) \coloneqq \| e(f)-e(g)\|_{L^2(\XX)} = (\int_X |e(f)-e(g)|^2\ d\mu)^{1/2}$, where $e(\theta) \coloneqq e^{2\pi i \theta}$; more generally, by the Pontryagin duality theorem, the following topological groups are isomorphic
$$\M(\XX,U) \equiv \M(\XX,\Hom(\hat U,\T)) \equiv \Hom(\hat U, \M(\XX,\T))\footnote{The isomorphism $\M(\XX,U)\cong \Hom(\widehat{U},\M(\XX,\T))$ maps $f\in \M(\XX,U)$ to the homomorphism $\chi\mapsto \chi\circ f$.}$$ 
(where $\Hom$ is in the category of abelian groups, and the countable discrete abelian group $\hat U$ is the Pontryagin dual of the compact abelian metrizable group $U$). This gives $\M(\XX,U)$ the structure of an abelian Polish group. Indeed, we can assign a complete metric on $\M(X,U)$ by setting $d(f,g) \coloneqq \sum_{n=1}^\infty \frac{1}{2^n} d_{L^2}(\chi_n\circ f,\chi_n\circ g)$, where $\{\chi_n:n\in \mathbb{N}\}$ is an enumeration of $\widehat{U}$.
\item For any $\gamma \in \Gamma$ and $f \in \M(\XX,U)$, we define the derivative $\partial_\gamma f \in \M(\XX,U)$ by the formula
$$ \partial_\gamma f \coloneqq f \circ T^\gamma - f.$$
This is well defined in $\M(\XX,U)$ since $T^\gamma$ preserves almost everywhere equivalence.  Observe that these operators $\partial_\gamma$ commute with each other, and we have the cocycle identity
\begin{equation}\label{cocycle-ident}
\partial_{\gamma_1+\gamma_2} = \partial_{\gamma_1} + \partial_{\gamma_2} \circ T^{\gamma_1} = \partial_{\gamma_1} + \partial_{\gamma_2} + \partial_{\gamma_1} \partial_{\gamma_2}
\end{equation}
for any $\gamma_1,\gamma_2 \in \Gamma$.
\item An element $p \in \M(\XX,U)$ is said to be a \emph{polynomial of degree (at most) $k$} if for every $\gamma_1,\ldots,\gamma_{k+1}\in \Gamma$ we have that
$$\partial_{\gamma_1}\ldots \partial_{\gamma_{k+1}}p = 0.$$
We denote by $\Poly_{\leq k}(X,U)$ the set of polynomials of degree at most $k$; this is a closed subgroup of $\M(\XX,U)$.  By convention, we set $\Poly_{\leq k}(X,U)=\{0\}$ for all integers $k<0$, thus $0$ is a polynomial of degree $k$ for any $k \in \Z$.  A function $\rho \colon \Gamma \to \M(\XX,U)$ is said to be a  \emph{polynomial of degree (at most) $k$} if it lies in $\Poly_{\leq k}(X,U)^\Gamma$, that is to say that $\rho_\gamma\in \Poly_{\leq k}(X,U)$ for all $\gamma \in \Gamma$.
\item When $U$ is the torus $\T = \R/\Z$, we abbreviate $\Poly_{\leq k}(\XX,U)$ as $\Poly_{\leq k}(X)$.  We also write $\Poly_{<k}$ for $\Poly_{\leq k-1}$.
\end{itemize}
\end{definition}

Note that we have the inclusions
$$ U \leq \Poly_{\leq 0}(\XX,U) \leq \Poly_{\leq 1}(\XX,U) \leq \Poly_{\leq 2}(\XX,U) \dots \leq \M(\XX,U)$$
of abelian Polish groups, with equality $U = \Poly_{\leq 0}(\XX,U)$ in the first inclusion when $\XX$ is ergodic.  We also have the equivalence
\begin{equation}\label{fourier-equiv}
\Poly_{\leq k}(\XX,U) \equiv \Poly_{\leq k}(\XX,\Hom(\hat U,\T)) \equiv \Hom(\hat U, \Poly_{\leq k}(\XX)).
\end{equation}
The following definitions will play a key role in our paper:

\begin{definition}[Abramov, Weyl, and divisible systems]  Let $\Gamma$ be a discrete countable abelian group, let $\XX$ be an ergodic $\Gamma$-system, and let $k \geq 1$.
\begin{itemize}
    \item $\XX$ is an \emph{Abramov system of order $k$} if the $\sigma$-algebra of $\XX$ is generated (up to null sets) by $\Poly_{\leq k}(\XX)$, or equivalently if $\{ e(P): P \in \Poly_{\leq k}(\XX)\}$ span a dense subspace of $L^2(\XX)$.
    \item $\XX$ is a \emph{Weyl system of order $k$} if it is of order $k$, and for each $1 \leq i < k$, the cocycle $\rho_i$ appearing in the description \eqref{structuregroups} is a polynomial of degree $i$, i.e., $\rho_i \in \Poly_{\leq i}(\ZZ^i(\XX), U_{i+1})^\Gamma$.
    \item $\XX$ is \emph{$k$-divisible}\footnote{This concept was denoted \emph{$k+1$-divisibility} in \cite{shalom2}.} if the abelian groups $\Poly_{\leq d}(\XX)$ are divisible for every $0 \leq d \leq k$ (i.e., if $P \in \Poly_{\leq d}(\XX)$ and $n \geq 1$ then there exists $Q \in \Poly_{\leq d}(\XX)$ such that $nQ=P$).  $\XX$ is \emph{$\infty$-divisible} if it is $k$-divisible for every $k$, i.e., $\Poly_{\leq d}(\XX)$ is divisible for every $d \geq 0$.
\end{itemize}
\end{definition}

In \cite[Theorem 3.8]{btz} it is shown that every Weyl system of order $k$ is Abramov of order $k$, and in \cite[Lemma A.35]{btz} it is shown\footnote{The paper \cite{btz} focuses on the case $\Gamma = \F_p^\omega$, but the arguments work in general.} that every Abramov system of order $k$ is of order $k$.  Conversely, from the description \eqref{structuregroups} it is clear that any system of order $1$ is Weyl of order $1$.  In general, it is known that systems of order $k$ are not necessarily Abramov or Weyl of order $k$; see Section \ref{counterexamples-sec} below.  Also, many systems of order $k$ are not divisible of any order (e.g., any ergodic $\mathbb{F}_p^\omega$-system) or totally disconnected(e.g., an irrational rotation), though as we shall see later one can often recover divisibility properties by passing to an extension.  

The case $\Gamma = \F_p^\omega$ has received particular interest (being closely related \cite{taozieglerhigh} to the inverse theorem for the Gowers norms in finite field vector spaces $\F_p^n$), and the following statements are known.

\begin{theorem}[$\F_p^\omega$-systems]\label{fp-omega}  Let $\XX$ be an ergodic $\F_p^\omega$-system of order $k$ for some prime $p$ and some $k \geq 1$.
\begin{itemize}
\item[(i)]  The structure groups $U_1,\dots,U_k$ are $p^r$-torsion for some $r \geq 1$.  In particular\footnote{Note that if $U$ is $p^r$-torsion, then so is the Pontraygin dual of $U$, or any closed subgroup of $U$.  In particular, no non-trivial closed subgroup of $U$ can be connected (since it has a non-trivial torsion character), and thus $U$ is totally disconnected.}, $\XX$ is totally disconnected.
\item[(ii)]  $\XX$ is Abramov of order $C$ for some $C = C(p,k)$ depending only on $p,k$.  If $p \geq k-1$, one can take $C=k$.
\item[(iii)]  If $p \geq k+1$, then $\XX$ is Weyl of order $k$.
\end{itemize}
\end{theorem}

These results are largely summarized in the bottom half of Figure \ref{fig:main}.

\begin{proof}  For $(i)$, the case $k=1$ follows from  Pontryagin duality (see e.g., the discussion in \cite[\S 1.5]{jst}), while the general case is established in \cite[Theorem 4.8]{btz}.  As we shall see in Theorem \ref{bounded-tor} below, we can in fact take $r=1$; this latter claim was essentially already established (in the framework of nilspaces) in \cite[Proposition 3.5]{CGSS}.  

The first part of $(ii)$ was established in \cite[Theorem 1.20]{btz}.  The second part was established for $p \geq k+1$ in \cite[Corollary 8.7]{btz}, and the $p \geq k-1$ case was established by nilspace methods in \cite[Theorem 1.12]{CGSS}.  As we shall recall in Section \ref{counterexamples-sec} below, the condition $p \geq k-1$ cannot be completely omitted.

Finally, $(iii)$ was established in \cite[Corollary 8.7]{btz}.
\end{proof}
Comparing parts (ii) and (iii) of Theorem \ref{fp-omega}, it is natural to ask whether every system of order $k$ is necessarily Weyl of order $k$ provided $p \geq k-1$. In Proposition \ref{second-counter}, we provide a counterexample showing that this is not the case in general.

We also briefly list some other structural results for $\Gamma$-systems of finite order:

\begin{itemize}
    \item In the landmark paper \cite{host2005nonconventional}, it was shown that ergodic $\Z$-systems of order $k$ were inverse limits of nilsystems of order $k$.  This result has since been extended to ergodic $\Gamma$-systems when $\Gamma$ is finitely generated \cite{griesmer}, \cite{candela-szegedy-cubic}.
    \item If $\Gamma$ is an arbitrary countable abelian group, it was shown in \cite[Theorem 1.21]{shalom2} that every ergodic $\Gamma$-system $\XX$ of order $2$ was a double coset system, and in \cite[Theorem 1.18]{shalom2} it was shown that they admit (generalized) extensions which are translational.  In \cite[Theorem 1.7]{jst} it was shown that ergodic $\Gamma$-systems of order $2$ were also inverse limits of translational systems; see \cite{shalom1, shalom2} for some previous results in this direction. 
    \item Additional structure theorems for $\mathbb{F}_p^\omega$-systems were established in \cite{CGSS} using the theory of nilspaces.
\end{itemize}

\subsection{New positive results}

We now state some additional implications between various types of systems of order $k$.  We begin with a generalization and strengthening of Theorem \ref{fp-omega}(i).

\begin{theorem}[Bounded torsion systems]\label{bounded-tor}  Let $\Gamma$ be $m$-torsion for some $m \geq 1$, and let $\XX$ be an ergodic $\Gamma$-system of order $k$ for some $k \geq 1$.  Then the structure groups $U_1,\dots,U_k$ are $m$-torsion. In particular, $\XX$ is totally disconnected.
\end{theorem}

This result is a straightforward consequence of existing theory and is established in Section \ref{sec-mtorsion}.  At the opposite extreme, it was shown in \cite[Theorem 9.5]{host2005nonconventional} that if $\Gamma$ is torsion-free and finitely generated\footnote{The arguments in \cite{host2005nonconventional} were stated for $\Z$-systems, but the arguments extend without difficulty to other torsion-free finitely generated abelian groups \cite[Theorem 4.9.5]{griesmer}.} and $\XX$ is an ergodic $\Gamma$-system of order $k$, then all the structure groups $U_2,\dots,U_k$ (except possibly for the first group $U_1$) are connected. 

In Theorem \ref{pSylow} we establish a convenient companion result to Theorem \ref{bounded-tor}, related to the Schur--Zassenhaus theorem, showing that an ergodic $\Gamma$-system $\XX$ of order $k$ of an $m$-torsion abelian group $\Gamma$ factors into a direct product $\XX = \biggenprod_{p|m} \XX_p$ of ``$p$-Sylow factors'', $\XX_p$, which are ergodic $\Gamma_p$-systems of order $k$ with $\Gamma_p$ the $p$-Sylow group of $\Gamma$ (the elements of $\Gamma$ whose order is a power of $p$).  In effect, this ``Sylow decomposition'' reduces the study of the action of $m$-torsion groups to the action of $p^r$-torsion groups for various prime powers $p^r$.

Theorem \ref{bounded-tor} motivates the study of totally disconnected systems of order $k$.  Such systems turn out to be particularly tractable after lifting $\Gamma$ to a torsion-free group.  This is because of the following two results:

\begin{theorem}[Divisible extension]\label{divis}  Let $\Gamma$ be torsion-free, and let $\XX$ be an ergodic $\Gamma$-system of order $k$ for some $k \geq 1$.  Then there exists an extension $\YY$ of $\XX$ which is still ergodic of order $k$, and is also $k$-divisible.  Furthermore, if $\XX$ is totally disconnected, then $\YY$ can also be chosen to be totally disconnected.
\end{theorem}

\begin{theorem}[From $k$-divisibility to $\infty$-divisibility and Weyl]\label{infty-weyl}  Let $\Gamma$ be torsion-free, let $k \geq 1$, and let $\XX$ be an ergodic totally disconnected $k$-divisible $\Gamma$-system of order $k$.  Then $\XX$ is in fact $\infty$-divisible and Weyl of order $k$.
\end{theorem}

Theorem \ref{divis} strengthens a previous result \cite[Theorems 3.16, 3.17]{shalom2} of the second author (which only established $1$-divisibility rather than $k$-divisibility, and did not attempt to preserve the property of being totally disconnected), and is proven by a similar method: we do so in Section \ref{divis-ext-sec}. We note that the hypothesis that $\Gamma$ is torsion-free is necessary. For example, no non-trivial $\mathbb{F}_p^\omega$-system is $1$-divisible, because it is impossible to divide a polynomial of degree $1$ by $p$ without increasing its degree.

Theorem \ref{infty-weyl} is one of the main new results of this paper, and is proven in Section \ref{divis-weyl-sec} by adapting several arguments from \cite{btz}.  However, extra effort is needed due to the fact that the structure groups are now merely totally disconnected instead of being elementary $p$-groups. The main new difficulty in this step arises from the fact that open subgroups of an arbitrary totally disconnected group do not necessarily split; more precisely, \cite[Lemma D.2]{btz} fails in this setting.

Theorems \ref{bounded-tor}, \ref{divis}, \ref{infty-weyl} are depicted in the upper half of Figure \ref{fig:main}.  By combining these theorems together, we will be able to easily obtain the following corollary:

\begin{corollary}[Weyl extensions]\label{Weyl extensions} Let $k \geq 1$.
\begin{itemize}
    \item[(i)]  If $\Gamma$ is torsion-free, then every totally disconnected ergodic $\Gamma$-system of order $k$ has an extension which is a totally disconnected $\infty$-divisible ergodic Weyl $\Gamma$-system of order $k$.
    \item[(ii)]  If $\Gamma$ is $m$-torsion for some $m \geq 1$, then every ergodic $\Gamma$-system of order $k$ has a generalized extension (as defined in Section \ref{system-sec}) to a totally disconnected $\infty$-divisible ergodic Weyl $\Z^\omega$-system of order $k$.
\end{itemize}
\end{corollary}

We will establish this corollary in Section \ref{weyl-sec}.  Using this corollary, we will be able to obtain some further structure theorems:

\begin{theorem}[Structure theorems]\label{structure-thms}  Let $k \geq 1$.
\begin{itemize}
    \item[(i)]  Every Weyl system of order $k$ is a translational system $G/\Lambda$ with $G$ a filtered\footnote{See Definition \ref{def-filtration} for the definition of a filtered group.} nilpotent Polish group of degree $k$. 
    \item[(ii)]  If $\Gamma$ is torsion-free, then every totally disconnected ergodic $\Gamma$-system of order $k$ is isomorphic to a double coset system $K \backslash G/\Lambda$ with $G$ a filtered nilpotent Polish group of degree $k$.
\end{itemize}
\end{theorem}

These results (which are also depicted in Figure \ref{fig:main}) can be compared with the celebrated result of Host and Kra \cite{host2005nonconventional} that ergodic $\Z$-systems of order $k$ are inverse limits of degree $k$ nilsystems.  We establish Theorem \ref{structure-thms} in Section \ref{sec-doublecoset}.  Part $(i)$ of the theorem is a straightforward algebraic construction; it is part $(ii)$ that requires Corollary \ref{Weyl extensions}.  Verifying that the construction actually makes $K \backslash G / \Lambda$ a Polish space and that the $\Gamma$-dynamics on this space is compatible with the quotient compact nilspace structure inherited from $G/\Lambda$ is surprisingly delicate, but can be accomplished with some effort.  Combining Theorem \ref{structure-thms}(i) with Theorem \ref{fp-omega}(iii), we conclude, in particular, that ergodic $\F_p^\omega$-systems are translational of degree $k$ (see Definition \ref{translational-def}) whenever $p\geq k+1$; this was previously observed in \cite[Theorem 2.3]{shalom1}.  

We can also obtain a further relation between Abramov and Weyl systems:

\begin{theorem}[From Abramov to Weyl assuming divisibility]\label{abramov-weyl}  Let $k \geq 1$, let $\Gamma$ be a countable abelian group, and let $\XX$ be an ergodic $k$-divisible Abramov $\Gamma$-system of order $k$.  If $\ZZ^{k-1}(\XX)$ is Weyl of order $k-1$, then $\XX$ is Weyl of order $k$.
\end{theorem}

This result (also depicted in Figure \ref{fig:main}) rests on a higher-order (or, polynomial) Moore--Schmidt type theorem (classifying cocycles that are cohomologous to polynomials) previously obtained by the second author in \cite{shalom2} which requires that the underlying system satisfies a higher-order divisibility\footnote{The r\^ole of divisibility in the context of the classical Moore--Schmidt theorem was also observed by the first and the third author in \cite{jt19}. In this case, the required (linear) amount of divisibility is automatically given for arbitrary systems.}, and is proved in Section \ref{sec-abweyl}.

\subsection{Application: an inverse theorem for the Gowers norm in bounded torsion groups}

Recall the definition of the Gowers uniformity norms: 
\begin{definition}
Let $G=(G,+)$ be a finite abelian group, let $k\geq 1$, and let $f \colon G\rightarrow \mathbb{C}$ be a function. The \emph{$k$-th Gowers norm}\footnote{$U^k(G)$ is a seminorm when $k=1$ and a norm for all $k>2$, see e.g., \cite[\S 11.1]{tao-vu}.} of $f$ is defined by the formula
$$\|f\|_{U^k(G)} \coloneqq \left(\mathbb{E}_{x,h_1,\ldots,h_k\in G} \Delta_{h_1}\ldots\Delta_{h_k} f(x)\right)^{1/2^k}$$ where $\Delta_h f(x) \coloneqq f(x+h)\cdot\overline{f(x)}$, and $\E_{x\in A} f(x)\coloneqq \frac{1}{|A|}\sum_{x\in A}f(x)$ is the average. 
\end{definition}

An \emph{inverse theorem} characterizes when a (bounded) function $f$ has large Gowers norm.  In the case when $G$ is a vector space $\F_p^n$ over a finite field $\F_p$, we have a (qualitatively) satisfactory inverse theorem:
 
\begin{theorem}[Inverse theorem for $U^{k+1}(\F_p^n)$]\label{TZinverse}
Let $p$ be a prime number, let $k\geq 1$, and let $\delta>0$. Then there exists $\eps = \eps(\delta,k,p)>0$ such that for every finite-dimensional vector space $G=\F_p^n$ and any $1$-bounded\footnote{A function $f \colon G \to \C$ is \emph{$1$-bounded} if $|f(x)| \leq 1$ for all $x \in G$.} function $f \colon G \to \C$ with $\|f\|_{U^{k+1}(G)}>\delta$, there exists a polynomial $P\in\Poly_{\leq k}(G)$ (viewing $G$ as a translational $G$-system in the obvious fashion) such that $$\left|\E_{x\in G} f(x) e(-P(x))\right|>\varepsilon,$$
where $e(\theta) \coloneqq e^{2\pi i \theta}$.
\end{theorem}

The high characteristic case of this theorem, that is when $p>k$, was established by the third author and Ziegler in \cite{taozieglerhigh} using a correspondence principle and a structure theorem from ergodic theory for $\F_p^\omega$-actions which they established jointly with Bergelson in \cite{btz}. The proof of the low characteristic case $p\leq k$ was established in \cite{tz-lowchar} by studying various notions of rank and equidistribution of polynomials. See also  \cite{szegedy-inv2}, \cite{milicevic}, \cite{milicevic-2}, \cite{milicevic-u56}, \cite{BSST}, \cite{CGSS} for other proofs of and related results to Theorem  \ref{TZinverse}, including several proofs that give quantitative control on the quantity $\eps$. 

When $\F_p^n$ is replaced by a more general finite abelian group, the above statement fails. For example, when $G = \Z/N\Z$ is a cyclic group (thinking of $N$ as being arbitrarily large), it was shown that a larger class of bracket polynomials or nilsequences form a higher-order Fourier basis; see the inverse theorem obtained by Green, Ziegler, and the third author in \cite{gtz-4} and \cite{gtz} (and \cite{green-tao-inverseu3} for the $k=2$ case).  Alternate proofs, generalizations, strengthenings, and other variants of this result have subsequently been established by several authors \cite{manners-periodic}, \cite{szegedy-inv2}, \cite{szegedy-inv}, \cite{manners-inv}, \cite{candela-szegedy-inverse}, \cite{jt21-1}.

However, by combining Corollary \ref{Weyl extensions} and Theorem \ref{structure-thms} with the correspondence principle from \cite{jt21-1} (see also \cite{towsner}), as well as some results on nilspace fibrations  (see Proposition \ref{fibration}),  we will be able to extend Theorem \ref{TZinverse} to the more general class of $m$-torsion groups, at the cost of worsening the degree:

\begin{theorem}\label{gowers-main}
Let $k,m\geq 1$ and let $\delta>0$. Then there exists $\varepsilon = \varepsilon(\delta,k,m)$ and a constant $C=C(k,m)$ such that for every finite abelian $m$-torsion group $G$ and any $1$-bounded function $f \colon G\rightarrow\mathbb{C}$ with $\|f\|_{U^{k+1}(G)}>\delta$, there exists a polynomial $P\in\Poly_{\leq C}(G)$ such that $$\left|\mathbb{E}_{x\in G} f(x) e(-P(x))\right|>\varepsilon.$$ 
\end{theorem}

This inverse theorem is already new for the groups $G=(\Z/4\Z)^n, n\geq 1$.  We give the details of the proof of this theorem in Section \ref{sec-gowers}.  The quantity $C(k,m)$ can be explicitly computed, but we do not attempt to optimize this quantity here. 

Our proof provides a stronger structural result for the polynomial $P$ from Theorem \ref{gowers-main}.
\begin{definition}
Let $G$ be a finite abelian group and write $G=\prod_{p \text{ is a prime}} G_p$ where $G_p$ is the $p$-sylow subgroup of $G$ and $l\in \mathbb{N}$ be an integer. Letting $H=(H,+)$ be another finite abelian group, we say that an extension $q:H\rightarrow G$ is an $l$-extension if the $p$-sylow subgroup of $\ker q$ is a $p^l$-torsion group.
\end{definition}
\begin{example}
   For every $n\geq 1$, the quotient map $q:\mathbb{Z}/2^{n+1}\mathbb{Z}\rightarrow \mathbb{Z}/2\mathbb{Z}$, $q(x)=x\mod 2$ is an $n$-extension.
\end{example}

We show that the function $f$ from theorem \ref{gowers-main} in fact correlates with a polynomial phase of degree $k$ defined on a bounded extension of the finite group $G$.
\begin{theorem}\label{gowersextension}
    Let $k,m\geq 1$ and let $\delta>0$. Then there exists $\epsilon=\epsilon(\delta,k,m)$ and a constant $C_1=C_1(k,m)$ such that for every finite abelian group $G$ and any $1$-bounded function $f:G\rightarrow \mathbb{C}$ with $\|f\|_{U^{k+1}(G)}>\delta$, there exist a $C_1$-extension $q:H\rightarrow G$ and a polynomial $P\in \Poly_{\leq k}(H)$ such that
    $$\left|\mathbb{E}_{x\in H} f(q(x))e(-P(x))\right|>\epsilon.$$
    \noindent Furthermore, there is a constant $C_2=C_2(k,m)$ and a cross-section $|\cdot|:G\rightarrow H$ of $q$ so that
    $$\left|\mathbb{E}_{x\in G} f(x) e(-P(|x|))\right|>\epsilon$$ and $P(|\cdot|)$ is a polynomial of degree $\leq C_2$.
    
\end{theorem}
\subsection{Counterexamples and open questions}\label{counterexamples-sec}

We now turn to some negative results that show that at least some of the hypotheses in the above results are necessary.

The classical example of the Heisenberg $\Z$-system is well known to be an ergodic $\Z$-system of order $2$ that is not Abramov of any order; see \cite{fw} and \cite[Remark 1.17]{btz}.  In fact, a similar example can be constructed for $\Z$ replaced by an unbounded torsion abelian group such as $\Gamma = \bigoplus_{p\in\mathbb{P}}\mathbb{F}_p$; see \cite[Theorem 1.21]{shalom3}.  Among other things, this shows that the bounded torsion hypothesis in Theorem \ref{bounded-tor} is necessary.  In a similar vein, from Pontragyin duality it is known that for abelian groups of unbounded torsion such as $\Gamma = \bigoplus_{p\in\mathbb{P}}\mathbb{F}_p$, there are ergodic $\Gamma$-systems of order $1$ which are not totally disconnected; see the discussion in \cite[\S 1.5]{jst}, as well as \cite[Example 1.28]{shalom3} for a concrete example.  In particular, being Weyl does not imply being totally disconnected even when the acting group $\Gamma$ is torsion (but not bounded torsion).  Finally, a well known example of Rudolph \cite{rudolph} (discussed further in \cite[Remark 5.6]{jst}) that produces an ergodic $\Z$-system of order $2$ that is not a translational system of degree $2$.

It was conjectured in \cite{btz} that the condition $k \leq p+1$ in Theorem \ref{fp-omega}(ii) could be removed, that is to say that every ergodic $\F_p^\omega$-system of order $k$ was Abramov of order $k$.  Recently, we were able to disprove this in \cite{jstcounterexample}, in which we constructed an ergodic $\F_2^\omega$-system of order $5$ which was not Abramov of order $5$.  We also have the following variant of this example:

\begin{proposition}\label{second-counter}  There exists an ergodic $\mathbb{F}_2^\omega$-system of order $3$ which is Abramov of order $3$, but not Weyl of order $3$. In particular, totally disconnected and Abramov systems are not necessarily Weyl. 
\end{proposition}

We establish Proposition \ref{second-counter} in Appendix \ref{appendix-d}; the argument is based off of a construction in \cite[Appendix E]{tz-lowchar} of an ergodic $\mathbb{F}_2^\omega$-system that does not have the ``exact root property''.

\begin{figure}
\centering
    \[\begin{tikzcd}
	{\substack{\Gamma=\Z \\ \text{order } 2}} && {\substack{\text{translational}\\ \text{system} \\ \text{degree } 2}}\\
	{\substack{\Gamma=\bigoplus_{p \in {\mathbb P}} {\mathbb F}_p \\ \text{order } 2}} && {\mathrm{Abramov}_C} \\
    {\substack{\Gamma = {\mathbb F}_2^\omega\\\text{order }5}} && {\mathrm{Abramov}_5}\\
	{\substack{\Gamma={\mathbb F}_2^\omega \\ \mathrm{Abramov}_3}} && {\mathrm{Weyl}_3} \\
	{\substack{\Gamma=\bigoplus_{p \in {\mathbb P}} {\mathbb F}_p \\ \mathrm{Weyl}_1}} && {\text{totally disc.}} 
    \arrow["\shortmid"{marking}, from=1-1, to=1-3]	
    \arrow["\shortmid"{marking}, from=1-1, to=2-3]
	\arrow["\shortmid"{marking}, swap, from=2-1, to=2-3]
	\arrow["\shortmid"{marking}, from=3-1, to=3-3]
	\arrow["\shortmid"{marking}, "\substack{\text{Prop. } \ref{second-counter}}", Rightarrow, from=4-1, to=4-3]
	\arrow["\shortmid"{marking}, from=5-1, to=5-3]
\end{tikzcd}\]
    \caption{A schematic depiction of the counterexamples.}
  \label{fig:counter}
\end{figure} 

The above counterexamples are summarized in Figure \ref{fig:counter}.
We close this section with some questions left open by the above positive and negative results.

\begin{question} Is there an ergodic $\Z^\omega$-system of order $k$ that is not isomorphic to an inverse limit of translational systems of degree at most $k$?
\end{question}

\begin{question}  Can every ergodic $\F_p^\omega$-system of order $k$ be extended (within the category of $\F_p^\omega$-systems) to an Abramov system of order $k$?
\end{question}
\begin{remark}
    The question above was recently answered in the affirmative by Candela, Gonz\'alez-S\'anchez, and Szegedy in \cite{CGSSextensions}.
\end{remark}
\begin{question}  Is every ergodic $\F_p^\omega$-system of order $k$ a translational system of degree $C(k,p)$ for some $C(k,p)$?  This is currently known for $k \leq p-1$. 
\end{question}
\begin{question}
Can the value $C(k,p)$ from the previous question be taken to be $k$ when $k=p$? What is the situation for $k=p+1$?
\end{question}

\begin{question}  Is every translational ergodic $\F_p^\omega$-system of degree $k$ Abramov of order $k$?  Perhaps the counterexample in \cite{jstcounterexample} already gives a negative answer to this question.
\end{question}

\begin{question} Theorem \ref{TZinverse} asserts that we can take $C(k,m)=k$ in Theorem \ref{gowers-main} when $m$ is prime.  Can one do so when $m$ is not prime?
\end{question}
\begin{remark}
    The question above was recently answered in the affirmative by the authors in \cite{jst-polynomialtowers}.
\end{remark}

\subsection*{Acknowledgements}
AJ was funded by the German Research Foundation under Germany's Excellence Strategy -- EXC-2047 -- 390685813 and its Heisenberg Programme --  547294463. 
OS was supported by NSF grant DMS-1926686 and Alon Fellowship. TT was supported by a Simons Investigator grant, the James and Carol Collins Chair, the Mathematical Analysis \& Application Research Fund Endowment, and by NSF grant DMS-1764034. Part of this work was done while the second author was a post-doc at the Hebrew University under the supervision of Mike Hochman and was supported by the ISF research grant 3056/21.  We thank Pablo Candela, Diego Gonz\'alez-S\'anchez, and Bal\'azs Szegedy for sharing their recent preprint \cite{cgss-2023}, a result from which (Lemma \ref{double-nilspace}) we use in our current paper. We thank the anonymous referee for their thorough review, positive comments, and constructive remarks on this manuscript.

\section{Groups of bounded torsion, totally disconnected systems, and a Sylow decomposition}\label{sec-mtorsion}

In this section we prove Theorem \ref{bounded-tor}, as well as a related Sylow decomposition (Theorem \ref{pSylow}) that will be useful in the proof of Theorem \ref{gowers-main}.

The key lemma used to prove Theorem \ref{bounded-tor} is

\begin{lemma}[Multiplication by $m$ lowers type]\label{times p type}
Let $\Gamma$ be a countable abelian $m$-torsion group for some $m \geq 1$, let $\XX$ be an ergodic $\Gamma$-system, let $U$ be a compact abelian metrizable group, and let $\rho \colon \Gamma \to \M(\XX,U)$ be a cocycle of type $k$. Then $m \cdot \rho$ is of type $k-1$.
\end{lemma}

\begin{proof}
Since $\rho$ is of type $k$ on $\XX$, $\Delta^{[k-1]}\rho$ is of type $1$ on $\XX^{[k-1]}$ thanks to Definition \ref{type:def}. Let $\xi\in \hat U$ be a character (i.e., a continuous homomorphism from $U$ to $\T$), which we write as $\xi \colon u \mapsto \xi \cdot u$. By Proposition \ref{type}(vi), for every ergodic component of $\XX^{[k-1]}$, we can find a homomorphism $c \colon \Gamma \rightarrow \T$ and a map $F \in \M( \XX^{[k-1]}, \T)$ such that $\Delta^{[k-1]}(\xi\cdot \rho) = c + dF$. Since $\Gamma$ is $m$-torsion, $mc=0$, and thus $\Delta^{[k-1]}(m \xi\cdot \rho) = d(m F)$. Therefore $\Delta^{[k-1]}(m\xi\cdot \rho)$ is a coboundary on each ergodic component of $\XX^{[k-1]}$. By Proposition \ref{type}(vii),  $\Delta^{[k-1]} m\xi\cdot \rho = \Delta^{[k-1]} (\xi \cdot m \rho)$ is a coboundary on $\XX^{[k-1]}$ for every $\xi \in\hat U$, and so $m \rho$ is of type $k-1$ thanks  to Proposition \ref{type}(ii) and Definition \ref{type:def}. 
\end{proof}

Now we prove Theorem \ref{bounded-tor} by induction on $k$.  If $k=1$, by \eqref{structuregroups} $\XX$ is a rotational system on a compact abelian group $U_1$. It is classical (see e.g., \cite[\S 1.5]{jst}) that $U_1$ can be identified with a closed subgroup of the Pontryagin dual $\hat \Gamma$ of $\Gamma$.  Since $\Gamma$ is $m$-torsion, so is $\hat \Gamma$ (see e.g., \cite[Theorem 18]{morris}), and hence $U_1$ is also $m$-torsion as required.

Now suppose that $k\geq 2$ and assume inductively that the claim holds for smaller values of $k$. We can write $\XX = \ZZ^{k-1}(\XX)\rtimes_{\rho_{k-1}} U_k$ where $\rho_{k-1}$ is of type $k$. Let $\xi\in\hat U_k$ be a character. By Lemma \ref{times p type}, the factor $\XX'=\ZZ^{k-1}(\XX)\rtimes_{m \rho_{k-1}} mU_k$ (with factor map $(x,u) \mapsto (x,mu)$ for $x \in \ZZ^{k-1}(\XX)$ and $u \in U_k$) is a system of order $k-1$ by Proposition \ref{type}(iii). Since $\ZZ^{k-1}(\XX)$ is the maximal factor of order $k-1$, we deduce that $mU_k=0$ and so $U_k$ is $m$-torsion.  

If $U$ is $m$-torsion, then so is its Pontryagin dual $\hat U$ (again, see \cite[Theorem 18]{morris}).  Since groups with torsion Pontryagin duals are totally disconnected (using characters of finite order to separate points), we conclude that all the structure groups $U_1,\dots,U_k$ are totally disconnected. This completes the proof of Theorem \ref{bounded-tor}.
\begin{example}
We construct an example of an ergodic $(\mathbb{Z}/4\mathbb{Z})^\omega$-system whose structure groups are merely $2$-torsion, showing that in general $m$ may not be the minimal torsion of the structure groups of an ergodic $\Gamma_m$-system. Let $Z = (\mathbb{Z}/2\mathbb{Z})^\mathbb{N}$ and $X = Z\times \mathbb{Z}/2\mathbb{Z}$ equipped with Borel $\sigma$-algebra, Haar measure, and the action induced by the commuting transformations
$$T_i (z,u) = (z+e_i, u+z_i)$$
where $e_i$ denotes the $i$-th generator of $(\mathbb{Z}/4\mathbb{Z})^\omega$. 
Observe that the orbit of $(z,u)$ under $T_i$ looks like $$(z,u)\to(z+e_i,u+z_i)\to (z,u+e_i)\to (z+e_i,u+e_i+z_i)\to (z,u)$$ and so $T_i^4=\id$ for all $i$. We therefore obtain a $(\mathbb{Z}/4\mathbb{Z})^\omega$-action on $X$. It is in fact not hard to see that $\XX=\ZZ\rtimes_\sigma \mathbb{Z}/2\mathbb{Z}$ with $\ZZ$ the Kronecker factor and $\sigma(\gamma,z) = \sum_{i=1}^\infty \gamma_i z_i + \binom{\gamma_i}{2} \mod 2$. Note that $\sigma$ is a polynomial of degree $1$ and so $\mathrm{X}$ is a Conze--Lesigne $(\mathbb{Z}/4\mathbb{Z})^\omega$-system. Note furthermore that $\mathrm{X}$ is not a Kronecker system because $\sigma$ is not cohomologous to a constant. 
\end{example}
\subsection{A totally disconnected system splits to a product of its $p$-Sylow subgroups}

It is a classical result that any finite abelian group is a direct product of its $p$-Sylow subgroups. This result extends to profinite (i.e., compact totally disconnected) abelian groups; see e.g., \cite[Proposition 2.3.8]{ribes-z} or \cite[Corollary 8.8]{hofmann}.  Dually, if $\Gamma$ is a countable abelian $m$-torsion group for some $m \geq 1$, we may make the identification
\begin{equation}\label{p-split}
\Gamma = \bigoplus_{p|m} \Gamma_p
\end{equation}
where $\Gamma_p$ is the $p$-Sylow subgroup of $\Gamma$ (those elements of $\Gamma$ whose order is a power of $p$), where the product runs over all the prime factors $p$ of $m$.  This decomposition can also be readily obtained from the Chinese remainder theorem.

We now obtain an analogous decomposition for ergodic $\Gamma$-systems of finite order. Define the \emph{generalized product}\footnote{This should not be confused with the product $\XX_1 \times \XX_2$ of two $\Gamma$-systems $\XX_1$, $\XX_2$, which is again a $\Gamma$-system rather than a $\Gamma \oplus \Gamma$ system.  The relation between the two is that $\XX_1 \times \XX_2$ is essentially the $\Gamma \oplus  \Gamma$-system $\XX_1 \genprod \XX_2$ but with the action restricted to the diagonal group $\{ (\gamma,\gamma): \gamma \in \Gamma\}$.} $\XX_1 \genprod \XX_2$ of a $\Gamma_1$-system $\XX_1 = (X_1, \X_1, \mu_1, T_1)$ and a $\Gamma_2$-system $\XX_2 = (X_2, \X_2, \mu_2, T_2)$ to be the $\Gamma_1 \oplus \Gamma_2$-system $\XX_1 \genprod \XX_2 = (X_1 \times X_2, \X_1 \times \X_2, \mu_1 \times \mu_2, T_1 \genprod T_2)$, with the shift $T_1 \genprod T_2$ defined by the formula
$$ (T_1 \genprod T_2)^{(\gamma_1,\gamma_2)}(x_1,x_2) \coloneqq (T_1^{\gamma_1} x_1, T_2^{\gamma_2} x_2).$$
Note that the $\Gamma_1 \oplus \Gamma_2$-system $\XX_1 \genprod \XX_2$ is a generalized extension of the $\Gamma_1$-system $\XX_1$ and the $\Gamma_2$-system $\XX_2$ (i.e., $\XX_1, \XX_2$ are generalized factors of  $\XX_1 \genprod \XX_2$). We may similarly define direct products $\biggenprod_{\alpha \in A} \XX_\alpha$ of any finite number of $\Gamma_\alpha$-systems $\XX_\alpha$ (indeed one could even define countable products in this fashion, but we will not need to do so here).

\begin{theorem}[Sylow decomposition for systems of $m$-torsion groups]\label{pSylow}
Let $m$ be a fixed natural number and let $\Gamma$ be an $m$-torsion group, which we decompose into $p$-Sylow factors $\Gamma_p$ as in \eqref{p-split}. Let $k\geq 1$ and let $\XX$ be an ergodic $\Gamma$-system of order $k$. Then $\XX$ is isomorphic (as a $\Gamma$-system) to a generalized product $\biggenprod_{p|m} \XX_p$ of ergodic $\Gamma_p$-systems $\XX_p$ of order $k$, where $p$ ranges over the primes dividing $m$.
\end{theorem}

Before we prove this theorem we first record an algebraic lemma that computes the order of vanishing of $T^{m^r}-1$ at $T-1$ (i.e., the largest power of $T-1$ that divides $T^{m^r}-1$) in an $m$-characteristic ring.

\begin{lemma}[Vanishing of $T^{m^r}-1$]\label{alg-lemma} Let $R$ be any ring with identity which is $m$-characteristic for some $m \geq 1$ in the sense that $m \cdot 1 = 0$ in $R$.  Then for any $T \in R$ and $r \geq 1$, we have
$$ T^{m^r} - 1 \in R \cdot (T-1)^{d_{m,r}}$$
and
$$ \sum_{j=0}^{m^r-1} T^j \in R \cdot (T-1)^{d_{m,r}-1}$$
where $d_{m,r}$ is defined by the formula
$$ d_{m,r} = \min_{1 \leq i \leq k} p_i^{a_i(r-1)+1}$$
and $m = \prod_{i=1}^k p_i^{a_i}$ is the prime factorization of $m$.  In particular, if $m > 1$ is a power of a prime $p$, then
$$ T^{p^r} - 1 \in R \cdot (T-1)^{p^r}$$
and
$$ 1 + T + \dots + T^{p^r-1} \in R \cdot (T-1)^{p^r-1}.$$
\end{lemma}

\begin{proof}  Write $D \coloneqq T-1$, then by the binomial formula
$$ T^{m^r}-1 = (1+D)^{m^r} - 1 = \sum_{j=1}^{m^r} \binom{m^r}{j} D^j$$
and similarly
$$ \sum_{j=0}^{m^r-1} T^j = \sum_{j=1}^{m^r} \binom{m^r}{j} D^{j-1}$$
(which follows from the previous identity by dividing by $D=T-1$ in the formal polynomial ring of $T$).  It thus suffices to show that $\binom{m^r}{j}$ is divisible by $m$ whenever $1 \leq j < d_{m,r}$.

Suppose for contradiction that there was $1 \leq j < d_{m,r}$ such that $\binom{m^r}{j}$ is not divisible by $m$, thus there exists a prime $p_i, 1 \leq i \leq k$ such that
$$ \nu_{p_i}\left( \binom{m^r}{j} \right) < \nu_{p_i}(m) = a_i$$
where $\nu_p(n)$ denotes the number of times $p$ divides $n$.  But from Kummer's identity
$$ \nu_p\left(\binom{a}{b}\right) = \nu_p(a!)-\nu_p(b!)-\nu_p((a-b)!) = \sum_{l=1}^\infty \left\lfloor \frac{a}{p^l} \right\rfloor - \left\lfloor \frac{b}{p^l} \right\rfloor  - \left\lfloor \frac{a-b}{p^l} \right\rfloor$$
and the pigeonhole principle (noting that the summands are non-negative integers), we conclude that the quantity
$$ \left\lfloor \frac{m^r}{p_i^l} \right\rfloor -\left \lfloor \frac{j}{p_i^l} \right\rfloor - \left\lfloor \frac{m^r - j}{p_i^l} \right\rfloor$$
must vanish for at least one $a_i(r-1)+1 \leq l \leq ra_i$.  But this implies that $j$ is divisible by $p_i^l \geq p_i^{a_i(r-1)+1} \geq d_{m,r}$, giving the required contradiction.
\end{proof}

Now we prove Theorem \ref{pSylow}.

\begin{proof}
We induct on $k$. The case $k=1$ follows from the Pontryagin duality between ergodic $\Gamma$-systems and countable subgroups of $\hat \Gamma$ (see e.g., \cite[\S 1.5]{jst}), so now suppose that $k>1$ and the claim has already been proven for $k-1$.  Thus we may write
$$\ZZ^{k-1}(\XX) = \biggenprod_{p|m} \YY_p$$
for some ergodic $\Gamma_p$-systems $\YY_p$ of order $k$.  By Proposition \ref{abelext} and Theorem \ref{bounded-tor}, we can therefore write
$$ \XX = \left(\biggenprod_{p|m} \YY_p\right) \rtimes_{\rho_{k-1}} U_k$$
for some $m$-torsion compact abelian metrizable group $U_k$ (so in particular $U_k$ is totally disconnected), and some $U_k$-valued cocycle $\rho_{k-1}$ of type $k$ on $\biggenprod_{p|m} \YY_p$.  Applying the Sylow decomposition for profinite abelian groups (\cite[Proposition 2.3.8]{ribes-z} or \cite[Corollary 8.8]{hofmann}) to $U_k$, one can split
$$ U_k = \prod_{p|m} U_{k,p}$$
where $U_{k,p}$ is the closed subgroup of $U_k$ consisting of elements whose order is a power of $p$.  We can thus write
$$ \rho_{k-1} = (\rho_{k-1,p})_{p|m}$$
where for each prime $p$ dividing $m$, $\rho_{k-1,p}$ is a $U_{k,p}$-valued cocycle of order $k$ on $\biggenprod_{q|m} \YY_{q}$ (with $q$ also understood to be a prime dividing $m$).  Suppose we could show that for each $p|m$, $\rho_{k-1,p}$ is cohomologous to another $U_{k,p}$-valued cocycle $\rho'_{k-1,p}$ that is measurable with respect to the $\YY_p$ factor.  Then, for any prime $q|m$ distinct from $p$, we would have
$$ \partial_{\gamma_q} (\rho'_{k-1,p})_\gamma = 0$$
for all $\gamma_q \in \Gamma_q$ and $\gamma \in \Gamma$.  By the cocycle property, the left-hand side is also equal to $\partial_\gamma (\rho'_{k-1,p})_{\gamma_q}$, thus by ergodicity the function $(\rho'_{k-1,p})_{\gamma_q}$ is equal almost everywhere to a constant $c_p(\gamma_q)$.  From the cocycle equation $c_p$ is a homomorphism from $\Gamma_q$ to $U_{k,p}$, hence is trivial since the torsions are coprime.  We conclude that $(\rho'_{k-1,p})_{\gamma_q} = 0$ almost surely, thus by the cocycle equation $(\rho'_{k-1,p})_{\gamma}$ only depends on the $\Gamma_p$-component $\gamma_p$ of $\gamma$, and so (by abuse of notation) $\rho'_{k-1,p}$ can now be thought of as a $U_{k,p}$-valued cocycle on the $\Gamma_p$-system $\YY_p$ rather than the $\Gamma$-system $\biggenprod_{q|m} \YY_{q}$.  By hypothesis, $\rho_{k-1}$ is cohomologous to the $U_k$-valued cocycle $\rho'_{k-1} \coloneqq (\rho'_{k-1,p})_{p|m}$, and $\XX$ would then be equivalent as a $\Gamma$-system to
$$ \left(\biggenprod_{p|m} \YY_p\right) \rtimes_{\rho'_{k-1}} U_k = \biggenprod_{p|m} (\YY_p \rtimes_{\rho'_{k-1,p}} U_{k,p})$$
In particular one can view the $\Gamma_p$-system $\YY_p \rtimes_{\rho'_{k-1,p}} U_{k,p}$ as a (generalized) factor of the $\Gamma$-system $\XX$; since $\XX$ was ergodic and of order $k$, each $\YY_p \rtimes_{\rho'_{k-1,p}} U_{k,p}$ is so as well, and this would prove the theorem. 

It remains to show, for each fixed prime $p$ dividing $m$, that $\rho_{k-1,p}$ is cohomologous to a $\YY_p$-measurable cocycle.  Fix such a prime $p$ and split
\begin{align*}
m &= m_p \times m_{(p)} \\
\Gamma &= \Gamma_p \oplus \Gamma_{(p)} \\
\ZZ^{k-1}(\XX) &= \YY_p \genprod \YY_{(p)}
\end{align*}
where $m_p$ is the largest power of $p$ dividing $m$, $m_{(p)}$ is the largest factor of $m$ coprime to $p$, $\Gamma_{(p)} \coloneqq \bigoplus_{q|m: q \neq p} \Gamma_q$ is the prime-to-$p$ component of $\Gamma$, and $\YY_{(p)} \coloneqq \biggenprod_{q|m: q \neq p} \YY_q$ is the prime-to-$p$ component of $\ZZ^{k-1}(\XX)$.  Then $\Gamma_{(p)}$ is an $m_{(p)}$-torsion group and $\YY_{(p)}$ is an ergodic $\Gamma_{(p)}$-system of order $k-1$.  Also, since $U_k$ is $m$-torsion, the $p$-part $U_{k,p}$ is $m_p$-torsion.  

By Proposition \ref{abelext} and Theorem \ref{bounded-tor}, we may write
$$ \YY_{(p)} = U_{1,(p)} \rtimes_{\rho_{1,(p)}} U_{2,(p)} \rtimes_{\rho_{2,(p)}} \dots \rtimes_{\rho_{k-2,(p)}} U_{k-1,(p)}$$
for some $m_{(p)}$-torsion compact abelian metrizable groups $U_{1,(p)},\dots,U_{k-1,(p)}$, and some type $i$ $U_{i,(p)}$-valued cocycles $\rho_{i-1,(p)}$ on
$$ \ZZ^{i-1}(\YY_{(p)}) = U_{1,(p)} \rtimes_{\rho_{1,(p)}} U_{2,(p)} \rtimes_{\rho_{2,(p)}} \dots \rtimes_{\rho_{i-2,p}} U_{i-1,(p)}$$
for $i=1,\dots,k-1$.

We will show by downward induction on $i$ that $\rho_{k-1,p}$ is cohomologous (as a $U_{k,p}$-valued cocycle on $\YY_p \genprod \YY_{(p)}$) to a cocycle that is $\YY_p \genprod \ZZ^i(\YY_{(p)})$-measurable for $i=0,\dots,k-1$; taking $i=0$ will give us the claim.  The claim for $i=k-1$ is trivial; now suppose that $0 \leq i < k-1$ and that the claim has already been proven for $i+1$, thus after modification by a coboundary one can view $\rho_{k-1,p}$ as being a $U_{k,p}$-valued cocycle on the $\Gamma$-system
$$\YY_{i+1;p} \coloneqq \YY_p \genprod (\ZZ^i(\YY_{(p)}) \rtimes_{\rho_{i,(p)}} U_{i+1,(p)}).$$
By Proposition \ref{type}(v), $\rho_{k-1,p}$ remains of type $k$ on this factor system.  

The group $U_{i+1,(p)}$ acts measure-preservingly on $\YY_{i+1;p}$ by vertical rotations.  By\footnote{By abuse of notation, we use $\partial_u$ for $u \in U$ to denote differentiation with respect to the action of $U$, and $\partial_\gamma$ for $\gamma \in \Gamma$ to denote differentiation with respect to the action of $\Gamma$, thus for instance $\partial_u$ and $\partial_\gamma$ commute.} Proposition \ref{type}(iv), $\partial_u \rho_{k-1,p}$ is a cocycle of type $k-i-1$ for every $u \in U_{i+1,(p)}$.  More generally, if we let $\langle \partial_u \rangle$ be the ring of operators (endomorphisms) on $\M( \YY_{p,i+1}, U_{k,p})$ generated by  $\partial_u$, then $\langle \partial_u \rangle$ is a (commutative) ring of characteristic $m_p$ since $U_{k,p}$ is $m_p$-torsion.  Hence by Lemma \ref{alg-lemma} we see that for any $r \geq 1$ that
$$ \partial_{m_p^r u} \in \langle \partial_u \rangle \partial_u^{p^r}$$
and hence by many applications of Proposition \ref{type}(iv), there exists $r>0$ such that $\partial_{m_p^r u} \rho_{k-1,p}$ is a cocycle of type $0$ - that is to say, a coboundary - for every $u \in U_{i+1,(p)}$. On the other hand, since $U_{i+1,(p)}$ is $m_{(p)}$-torsion and $m_p^r$ is coprime to $m_{(p)}$, we have $m_p^r U_{i+1,(p)} = U_{i+1,(p)}$.  Thus $\partial_u \rho_{k-1,p}$ is a coboundary for every $u \in U_{i+1,(p)}$, that is to say for every $u \in U_{i+1,(p)}$ there exists $F_u \in \M( \YY_{i+1;p}, U_{k,p})$ solving the equation
\begin{equation}\label{upf}
 \partial_u \rho_{k-1,p} = dF_u.
 \end{equation}
Since $\YY_{i+1;p}$ is an ergodic $\Gamma$-system, we see that $F_u$ is determined up to a constant shift by $U_{k,p}$.  Thus if we let $H$ denote the subgroup of the semi-direct product $U_{i+1,(p)} \ltimes \M( \YY_{i+1;p}, U_{k,p})$ (defined in Remark \ref{cocyc-skew}, and equipped with the product topology) consisting of all pairs $(u, F_u) \in U_{i+1,(p)} \ltimes \M( \YY_{i+1;p}, U_{k,p})$ solving the equation \eqref{upf}, then $H$ is a closed subgroup of $U_{i+1,(p)} \ltimes \M( \YY_{i+1;p}, U_{k,p})$ and we have a short exact sequence
\begin{equation}\label{uhp}
 0 \to U_{k,p} \to H \to U_{i+1,(p)} \to 0
 \end{equation}
 of topological groups. Since the two factors $U_{k,p}, U_{i+1,(p)}$ are compact (and the map from $U_{k,p}$ to $H$ is an embedding, and $U_{i+1,(p)}$ has the quotient topology from $H$), $H$ is also compact thanks to \cite[Theorem 5.25]{hewittross}.

Since $U_{i+1,(p)}$ is abelian, \eqref{uhp} implies that $[H,H]\leq U_{k,p}$. In particular, $[H,H]$ is central, so $H$ is at most $2$-step nilpotent. Therefore, the commutator map $[\cdot,\cdot]$ on $H$ induces a bilinear map from $U_{i+1,(p)} \times U_{i+1,(p)}$ to $U_{k,p}$.  Since $U_{k,p}$ is $m_p$-torsion and $U_{i+1,(p)}$ is $m_{(p)}$-torsion, this bilinear map must be trivial; that is to say, $H$ is abelian.    Since $U_{k,p}$ is $m_p$-torsion and $U_{i+1,(p)}$ is $m_{(p)}$-torsion, with $m_p$ and $m_{(p)}$ coprime, the profinite Schur--Zassenhaus theorem (see e.g., \cite[Proposition 2.33]{wilson}) implies that this sequence splits in the category of profinite abelian groups.  Thus we have a continuous homomorphism $u \mapsto (u,F_u)$ from $U_{i+1,(p)}$ to $H$.  The homomorphism property is equivalent to the cocycle property
$$ F_{u+u'} = F_u + F_{u'} \circ V^u$$
for $u,u' \in U_{i+1,(p)}$.  If we then define $F \in \M( \YY_{i+1;p}, U_{k,p} )$ by the formula
$$ F(y_p, y_{(p)}, u+u_0) \coloneqq F_u( y_p, y_{(p)}, u_0)$$
for all $y_p \in \YY_p$, $y_{(p)} \in \ZZ^i(\YY_{(p)})$, $u \in U_{i+1,(p)}$ and a generic $u_0 \in U_{i+1,(p)}$ (cf., \cite[Lemma B.6]{btz}), then one easily sees that $F_u = \partial_u F$ for all $u \in U_{i+1,(p)}$.  From \eqref{upf} we conclude that the cocycle $\rho_{k-1,p} - dF$ is $U_{i+1,(p)}$-invariant, and therefore $\YY_p \genprod \ZZ^i(\YY_{(p)})$-measurable.  Since this cocycle is cohomologous to $\rho_{k-1,p}$, we have closed the induction.
\end{proof}

\section{Abramov and divisible imply Weyl}\label{sec-abweyl}

In this section we prove Theorem \ref{abramov-weyl}.  The arguments used to prove this theorem are not needed elsewhere in the paper.

Let $k, \Gamma, \XX$ be as in the theorem, then by Proposition \ref{abelext} we can write $\XX=\ZZ^{k-1}(\XX)\rtimes_\rho U$ for some (type $k$) cocycle $\rho$ and a compact abelian metrizable group $U$.  Suppose that $P \in \Poly_{\leq k}(\XX)$ is a polynomial of degree $\leq k$. From Proposition \ref{ppfacts}(iii) we know that for every $u \in U$, $\partial_u P$ is a polynomial of degree at most zero, and thus constant almost everywhere by ergodicity. Thus there is a map $\xi_P \colon U \to \T$ such that $\partial_u P = \xi_P(u)$ for all $u \in U$.  From the cocycle equation we see that $\xi_P$ is a homomorphism and thus lies in $\hat U$.  Thus $e(P)$ is an eigenfunction of the $U$-action, with eigenvalue $e(\xi_P)$.

Conversely, if $\xi \in \hat U$, then by the Abramov hypothesis, the function $(y,u) \mapsto e(\xi \cdot u)$ on $\XX=Z^{k-1}(\XX)\rtimes_\rho U$ must have a non-zero inner product with $e(P)$ for some $P \in \Poly_{\leq k}(\XX)$.  Note that $(y,u) \mapsto e(\xi \cdot u)$ is an eigenfunction of the $U$-action with eigenvalue $e(\xi)$. It follows from the unitary nature of this action that distinct $U$-eigenfunctions are orthogonal. Therefore, $\xi_P = \xi$.  We thus have
$$ P(y,u) = F_\xi(y) + \xi \cdot u$$
for almost every $y \in \ZZ^{k-1}(\XX)$, $u \in U$, and some $F_\xi \in \M( \ZZ^{k-1}(\XX), \T)$.  Taking a derivative $\partial_\gamma$ for any $\gamma \in \Gamma$, we conclude that
$$ \partial_\gamma P(y,u) = \partial_\gamma F_\xi(y) + \xi \cdot \rho_\gamma(y).$$
Recall that $\partial_u P = \xi_P(u)$ for all $u\in U$, and so $\partial_\gamma P$ is $U$-invariant for all $\gamma\in \Gamma$. Thus, $dF_\xi + \xi \cdot \rho$ is a $\ZZ^{k-1}(\XX)$-measurable polynomial of degree $\leq k-1$.  We conclude that $\xi \cdot \rho$ is a $\T$-valued quasi-coboundary on $\ZZ^{k-1}(\XX)$ of degree $\leq k-1$.  Since this holds for every $\xi\in\hat U$, Proposition \ref{typecircle:prop} implies that $\rho$ is a $U$-valued quasi-coboundary on $\ZZ^{k-1}(\XX)$ of degree $\leq k-1$. By Remark \ref{cohiso}, cohomologous cocycles define isomorphic extensions, thus we may therefore assume without loss of generality that $\rho_\gamma$ is a polynomial of degree $\leq k-1$ for every $\gamma \in \Gamma$; since $\ZZ^{k-1}(\XX)$ was Weyl of order $k-1$, we conclude that $\XX$ is Weyl of order $k$ as required.

\section{Obtaining divisibility by extension}\label{divis-ext-sec}

The purpose of this section is to establish Theorem \ref{divis}.

Let $\Gamma$ be a countable abelian group.  Recall from Definition \ref{def-multilinear} that for any $k \geq 0$, $\SML_k(\Gamma,\T)$ denotes the space of symmetric multilinear forms $b \colon \Gamma^k \to \T$.  This can be viewed as a compact abelian subgroup of $\T^{\Gamma^k}$.  We have $\SML_0(\Gamma,\T) = \T$, while $\SML_1(\Gamma,\T)$ is essentially the Pontryagin dual $\hat \Gamma$ of $\Gamma$; thus one can view $\SML_k(\Gamma,\T)$ for $k > 1$ as ``higher order'' Pontryagin duals of $\Gamma$.

If $\XX$ is an ergodic $\Gamma$-system and $P \in \Poly_{\leq k}(\XX)$ is a polynomial of degree $\leq k$, then for any $\gamma_1,\dots,\gamma_k \in \Gamma$, the derivative $\partial_{\gamma_1} \dots \partial_{\gamma_k} P$ is a polynomial of degree $\leq 0$ and is thus constant by ergodicity.  Thus there is a map $\nabla^k P \colon \Gamma^k \to \T$ defined by
$$ \nabla^k P(\gamma_1,\dots,\gamma_k) \coloneqq \partial_{\gamma_1} \dots \partial_{\gamma_k} P.$$
From the cocycle equation we see that $\nabla^k P$ is multilinear, and by commutativity it lies in $\SML_k(\Gamma,\T)$, and so $\nabla^k$ is a homomorphism from $\Poly_{\leq k}(\XX)$ to $\SML_k(\Gamma,\T)$.  The image of this homomorphism was called the \emph{$k^{\mathrm{th}}$ discrete spectrum}\footnote{The case $k=1$ corresponds to the classical discrete spectrum of $\XX$, as defined for instance in \cite[Chapter "Discrete spectrum"]{halmos2013lectures}.} of $\XX$ in \cite{shalom2}; we denote it as $\Spec_k(\XX)$.  The kernel of this homomorphism is $\Poly_{<k}(\XX)$, thus we have a short exact sequence
\[\begin{tikzcd}
	0 & {\Poly_{<k}(\XX)} & {\Poly_{\leq k}(\XX)} & {\Spec_k(\XX)} & 0
	\arrow[from=1-4, to=1-5]
	\arrow["{\nabla^k}", from=1-3, to=1-4]
	\arrow[from=1-2, to=1-3]
	\arrow[from=1-1, to=1-2]
\end{tikzcd}\]
of abelian groups.  Note from Proposition \ref{ppfacts}(i) that $\Spec_k(\XX)$ is at most countable.

In general, a symmetric multilinear form $b \in \SML_k(\Gamma.\T)$ need not lie in the $k$-discrete spectrum of a given $\Gamma$-system $\XX$.  However, we have the following construction:

\begin{proposition}\label{bkg}  Let $b \in \SML_k(\Gamma,\T)$.  Then there exists an ergodic $\Gamma$-system $\XX$ such that $k! b \in \Spec_k(\XX)$.
\end{proposition}

\begin{proof}  To motivate the construction, define the functions $f_i \colon \Gamma \to \SML_i(\Gamma,\T)$ for $i=0,\dots,k$ by the formula
$$ f_i(x)(\gamma_1,\dots,\gamma_i) \coloneqq b( x^{\times k-i}, \gamma_1,\dots,\gamma_i)$$
for $x, \gamma_1,\dots,\gamma_i \in \Gamma$, where $x^{\times k-1}$ denotes $k-i$ repetitions of $x$. We abbreviate this formula as
$$ f_i(x) \coloneqq b( x^{\times k-i}, \cdot).$$
Clearly $f_k = b$, and we have the recursive identities
$$ f_i(x + \gamma) = \sum_{j=0}^{k-i} \binom{k-i}{j} f_j(x)(\gamma^{\times k-i-j}, \cdot)$$ 
while $f_0 \colon x \mapsto b(x,\dots,x)$ is a polynomial of degree $k$ with $\nabla^k f_0 = k! b$.

Inspired by this calculation, we let $\XX$ denote the product space
$$ X \coloneqq \SML_0(\Gamma,\T) \times \SML_1(\Gamma,\T) \times \dots \times \SML_{k-1}(\Gamma,\T) \times \{b\}$$
of tuples $(b_i)_{i=0}^k$ with $b_i \in \SML_i(\Gamma,\T)$ and $b_k = b$, which we equip with the Haar probability measure, and with the shift
$$ T^\gamma (b_i)_{i=0}^k \coloneqq \left( \sum_{j=0}^{k-i} \binom{k-i}{j} b_{i+j}(\gamma^{\times j},\cdot) \right)_{i=0}^k.$$
Thus for instance if $k=3$, then
$$ T^\gamma (b_0,b_1,b_2,b_3) = (b_{0,\gamma}, b_{1,\gamma}, b_{2,\gamma}, b_{3,\gamma})$$
where
\begin{align*}
b_{0,\gamma} &= b_0 + 3 b_1(\gamma) + 3 b_2(\gamma,\gamma) + b_3(\gamma,\gamma,\gamma) \\
b_{1,\gamma}(\gamma_1) &= b_1(\gamma_1) + 2 b_2(\gamma,\gamma_1) + b_3(\gamma,\gamma,\gamma_1) \\
b_{2,\gamma}(\gamma_1,\gamma_2) &= b_2(\gamma_1,\gamma_2) + b_3(\gamma,\gamma_1,\gamma_2) \\
b_{3,\gamma}(\gamma_1,\gamma_2,\gamma_3) &= b_3(\gamma_1,\gamma_2,\gamma_3).
\end{align*}
It is easy to see from Fubini's theorem that each $T^\gamma$ is a measure-preserving map on $X$.  It is also an action, since
\begin{align*}
T^\gamma T^{\gamma'} (b_i)_{i=0}^k &= \left( \sum_{j=0}^{k-i} \sum_{l=0}^{k-i-j} \binom{k-i}{j} \binom{k-i-j}{l} b_{i+j+l}(\gamma^{\times j},(\gamma')^{\times l}, \cdot) \right)_{i=0}^k \\
&= \left(\sum_{m=0}^{k-i}  \binom{k-i}{m} \sum_{j=0}^m \binom{m}{j} b_{i+m}(\gamma^{\times j},(\gamma')^{\times m-j}, \cdot) \right)_{i=0}^k \\
&= \left(\sum_{m=0}^{k-i} \binom{k-i}{m} b_{i+m}((\gamma+\gamma')^{\times m}, \cdot) \right)_{i=0}^k \\
&= T^{\gamma+\gamma'} (b_i)_{i=0}^k.
\end{align*}
Thus $\XX$ is a $\Gamma$-system (though not necessarily an ergodic one).  Let $P \in \M(\XX,\T)$ denote the function
$$ P( (b_i)_{i=0}^k ) \coloneqq b_0.$$
For any $\gamma_1,\dots,\gamma_k \in \Gamma$ and $\epsilon_1,\dots,\epsilon_k \in \{0,1\}$, we have
$$ P( T^{\epsilon_1 \gamma_1 + \dots + \epsilon_k \gamma_k} (b_i)_{i=0}^k )
= \sum_{j=0}^{k} \binom{k}{j} b_j((\epsilon_1 \gamma_1+\dots+\epsilon_k \gamma_k)^{\times j}).$$
Using the symmetric multilinear nature of the $b_j$, the right-hand side can be expanded as
$$ k! \epsilon_1 \dots \epsilon_k b_k(\gamma_1,\dots,\gamma_k)$$
plus finitely many additional terms, each of which are independent of at least one of the $\epsilon_i$.  Taking alternating differences and using $b_k = b$, we conclude that
$$ \partial_{\gamma_1} \dots \partial_{\gamma_k} P = k! b.$$
In particular $P \in \Poly_{\leq k}(\XX)$.  If we decompose $\XX$ into ergodic components, then for almost every ergodic component $\XX'$ we have $P \in \Poly_{\leq k}(\XX')$ and $\nabla^k P = k! b$ in $\XX'$, giving the claim.
\end{proof}

The above proposition is particularly useful when $\Gamma$ is torsion-free:

\begin{corollary}  Let $\Gamma$ be torsion-free, let $\XX$ be an ergodic $\Gamma$-system, let $P \in \Poly_{\leq k}(\XX)$ for some $k \geq 1$, and let $m \geq 1$ be an integer.  Then there exists an ergodic extension $\YY$ of $\XX$ and $Q \in \Poly_{\leq k}(\YY)$ such that $mQ = P$ (where we identify $\Poly_{\leq k}(\XX)$ with a subgroup of $\Poly_{\leq k}(\YY)$ in the obvious fashion).
\end{corollary}

We remark that this result was first proven (by a different method) in \cite[Proposition 3.15]{shalom2}.

\begin{proof}  By Proposition \ref{symdivdiscrete} and the torsion-free nature of $\Gamma$, we can find $b \in \SML_k(\Gamma,\T)$ such that $m k! b = \nabla^k P$.  By Proposition \ref{bkg}, we can find an ergodic $\Gamma$-system $\ZZ$ and $Q \in \Poly_{\leq k}(\ZZ)$ such that $\nabla^k Q = k! b$.  If we view $P,Q$ as degree $k$ polynomials in the product system $\XX \times \ZZ$, we conclude that $\nabla^k (P-mQ) = 0$, thus $P-mQ \in \Poly_{<k}(\XX \times \ZZ)$.  The product system $\XX \times \ZZ$ need not be ergodic, but almost every ergodic component $\YY$ of $\XX \times \ZZ$ will be an extension\footnote{Indeed, the conditional expectation operator $\E(|\XX)$ maps invariant functions of $\XX \times \ZZ$ to invariant (and hence constant) functions in $\XX$, hence by duality mean zero functions in $\XX$ are orthogonal to the measure of almost every ergodic component $\YY$, so that such measures pushforward to the measure on $\XX$.} of $\XX$, and $P-mQ \in \Poly_{<k}(\YY)$.  The claim now follows by an induction on $k$.
\end{proof}

We can iterate this\footnote{A similar iteration appears in the construction of sated extensions in \cite{austin-norm, austin2015pleasant}.  We thank Tim Austin for this observation.} to conclude

\begin{corollary}  Let $\Gamma$ be torsion-free, let $\XX$ be an ergodic $\Gamma$-system, and let $k \geq 1$.  Then there is an ergodic extension $\YY$ of $\XX$ which is $k$-divisible.
\end{corollary}

\begin{proof}  By iterating the above construction and taking inverse limits (noting from Proposition \ref{ppfacts}(i) that $\Poly_{\leq k}(\XX_n)$ is countable modulo constants for every $n$), we can construct a sequence of ergodic extensions
$$ \XX = \XX_0 \leq \XX_1 \leq \XX_2 \dots$$
such that for any $1 \leq i \leq k$, $n \geq 1$, $P \in \Poly_{\leq i}(\XX_n)$, and $m \geq 1$, there exists $Q \in \Poly_{\leq i}(\XX_{n+1})$ such that $mQ = P$.  Now let $\YY$ be the inverse limit of the $\XX_n$, then $\YY$ is also an ergodic extension of $\XX$, as an inverse limit of ergodic systems. From Proposition \ref{ppfacts}(vi), $\Poly_{\leq i}(\YY)$ is the union of the $\Poly_{\leq i}(\XX_n)$ and is thus divisible for each $1 \leq i \leq k$, giving the claim.
\end{proof}

Now we can easily establish the first part of Theorem \ref{divis}, because if $\XX$ is of order $k$, and $\YY$ is the ergodic extension constructed by the above corollary, then $\ZZ^k(\YY)$ is an extension of $\XX$, being the maximal factor of $\YY$ of order $k$.  On the other hand, since every element of $\Poly_{\leq k}(\YY)$ is measurable in $\ZZ^k(\YY)$ by Proposition \ref{ppfacts}(ii), we have $\Poly_{\leq i}(\YY) = \Poly_{\leq i}(\ZZ^k(\YY))$, and so the $k$-divisibility of $\YY$ implies the $k$-divisibility of $\ZZ^k(\YY)$, and so $\ZZ^k(\YY)$ is a $k$-divisible order $k$ ergodic extension of $\XX$, as required.

To prove the second part of Theorem \ref{divis} we need to know how the property of being totally disconnected interacts with polynomials and with basic operations on $\Gamma$-systems.  We begin with

\begin{proposition}\label{prop-div-invlim}
    Let $\Gamma$ be a countable abelian group and let $k\geq 1$.  
    Let $\XX,\XX_i, i\in I$ be ergodic totally disconnected $\Gamma$-systems of order $k$. 
    \begin{itemize}
        \item[(i)] Any factor $\YY$ of $\XX$ is totally disconnected. 
        \item[(ii)] If $\WW$ is the inverse limit of the $\XX_i$ in the category of $\Gamma$-systems, then $\WW$ is a totally disconnected $\Gamma$-system of order $k$. 
    \end{itemize}
\end{proposition}
\begin{proof}
Compact Hausdorff totally disconnected groups are profinite. Thus, factors and inverse limits of totally disconnected groups are totally disconnected in the category of topological groups, see e.g., \cite[\S 1.1]{ribes-z}. 
The claims then follow from Proposition \ref{prop-invlim}. 
\end{proof}

We obtain that any $\Gamma$-system of bounded order has a maximal totally disconnected factor: 

\begin{corollary}[Maximally totally disconnected factor]\label{cor-max-disconnected}
    Let $\Gamma$ be a countable abelian group and let $k\geq 1$. 
    Let $\XX=(X,\mu,T)$ be an ergodic $\Gamma$-system of order $k$. 
    Then there is a maximal totally disconnected factor $\YY$ of $\XX$ in the sense that if $\WW$ is any totally disconnected factor of $\XX$ then $\WW$ is a factor of $\YY$. 
\end{corollary}
\begin{proof}
Let $\mathcal{Y}$ be the collection of totally disconnected factors of $\XX$. 
We can order $\mathcal{Y}$ by the factor relation. For two totally disconnected factors $\YY_1,\YY_2\in\mathcal{Y}$ with factor maps $\pi_1\colon \XX\to \YY_1, \pi_2\colon \XX\to \YY_2$, we can form their joining $\YY_1\times \YY_2$ equipped with the pushforward measure $(\pi_1\times\pi_2)_*\mu$ (cf.~\cite[\S 3.5]{hk-book}).  From the representation \eqref{structuregroups}, we obtain that  $\YY_1\times \YY_2$ is totally disconnected. Thus $\mathcal{Y}$ is directed, and we can form its inverse limit which is totally disconnected by Proposition \ref{prop-div-invlim}(ii). 
\end{proof}

\begin{remark}
The maximal totally disconnected factor cannot be obtained simply by quotienting out the connected components of the structure groups. Indeed, consider the ergodic $\Z^\omega$-system $\XX \coloneqq(\T\times \prod_{i=1}^\omega \Z/p\Z)\rtimes_\rho \Z/p\Z$ for a prime $p$, with the $\Z^\omega$-action defined by the countably many commuting transformations $T_i(t,u,w) = (t+\alpha,u+e_i,w+u+(\{t+\alpha\}-\{t\}-\alpha))$, for $1\leq i<\omega$, where $\alpha\in\T$ is irrational, $t\in\T$, $u,w\in \Z/p\Z$, $e_i\in\prod_{i=1}^\omega \Z/p\Z$ has $1$ in the $i$-th coordinate and $0$ otherwise,  and $\{\cdot\}$ denotes the fractional part. The maximal totally disconnected factor of $\XX$ is isomorphic to $\prod_{i=1}^\omega  \Z/p\Z$ equipped with the rotations $u\mapsto u+e_i$. The reason is that the cocycle $\rho$ is not cohomologous to a cocycle which is invariant to translations by $t\in \T$ as cocycles taking values in $\Z/p\Z$. 
\end{remark}

As is well known, the maximal totally disconnected factor $Y$ of a compact abelian group $X$ is formed by quotienting $X$ by the connected component of the identity. Namely, $\YY=\XX/\XX^0$.  Since the Pontryagin dual of the connected group $\XX^0$ is torsion-free, we conclude that $\hat X/\hat Y\cong \widehat{\XX^0}$ is torsion-free.  The following proposition can be viewed as a higher order version of this statement.

\begin{proposition}\label{maxtdfactor}
 Let $\Gamma$ be a countable abelian group, let $k\geq 1$, and let $\YY=(Y,\nu,S)$ be the maximal totally disconnected factor of an ergodic $\Gamma$-system $\XX=(X,\mu,T)$ of order $k$. Then $\Poly_{\leq k}(\XX) / \Poly_{\leq k}(\YY)$ is torsion-free.  
\end{proposition}

\begin{proof}  We introduce intermediate factors
$$ \YY = \WW_k \leq \WW_{k-1} \leq \dots \leq \WW_0=\XX$$ 
where for each $0\leq i \leq k$, $\WW_i$ is the maximal factor of $\XX$ with the last $i$ structure groups being totally disconnected (these factors exist using the same argument as in the proof of Corollary \ref{cor-max-disconnected}). Since the extension of a torsion-free group by a torsion-free group is again torsion-free, it suffices to show that $\Poly_{\leq k}(\WW_{i-1}) / \Poly_{\leq k}(\WW_{i})$ is torsion-free for each $1 \leq i \leq k$.  That is to say, if $Q \in \Poly_{\leq k}(\WW_{i-1})$ and $n \geq 1$ are such that $nQ$ is measurable in $\WW_i$, then we need to show that $Q$ is measurable in $\WW_i$.

Write
\begin{align*}
\WW_{i-1} &= U_0\rtimes_{\rho_0} U_1\rtimes_{\rho_1} \ldots\rtimes_{\rho_{k-1}}U_k \\
\WW_i &=  V_0\rtimes_{\sigma_0}  V_1\rtimes_{\sigma_1} \ldots\rtimes_{\sigma_{k-1}} V_{k}.
\end{align*}
By Proposition \ref{prop-invlim}, for each $0\leq j\leq k$, we obtain surjective homomorphisms $\varphi_j: U_j \rightarrow V_j$, whose kernels we denote by $K_j$, such that $\varphi_{j+1}\circ \rho_j$ is cohomologous to $\sigma_{j}\circ \pi_{j}$, where $\pi_j:\ZZ^j(\WW_{i-1})\rightarrow \ZZ^j(\WW_i)$ is the factor map, for all $0\leq j\leq k-1$. Let $(U_{k-i})_0$ denote the connected component of the identity in the group $U_{k-i}$. Since $V_{k-i}$ is totally disconnected, $(U_{k-i})_0$ is a subgroup of $K_{k-i}$. We make a small reduction. Since $\sigma_j\circ\pi_j$ and $\varphi_{j+1}\circ \rho_j$ are cohomologous, we can write
$\varphi_{j+1} \circ \rho_j = \sigma_j \circ \pi_j + dF_j$. 
Let $s \colon W_j\rightarrow U_j$ be any measurable cross-section. We may replace $\rho_j$ with $\rho_j-F_j\circ s$, and  assume without loss of generality that $\varphi_{j+1}\circ\rho_j = \sigma_j\circ\pi_j$. In particular, we note that  $\varphi_{j+1}\circ\rho_j$ is invariant to translations by $(U_{k-i})_0$. We can now iteratively apply the first part of Lemma \ref{lem-goingup}, and find a compact connected abelian group $\tilde{U}_{k-i}$ that acts freely on $\WW_{i-1}$. Since the orbit of $(0,\ldots,0)$ under $\tilde{U}_{k-i}$ is included in $K_{k-i+1}, \ldots,K_k$, we deduce that $nQ$ is invariant with respect to translations by $\tilde{U}_{k-i}$.  The action of $\tilde {U}_{k-1}$ is strongly continuous in $L^2$, and thus continuous in $L^\infty$ when acting in $Q$, thanks to Proposition \ref{ppfacts}(i).  Thus for $u \in \tilde{U}_{k-1}$ close enough to the identity, $\partial_u Q$ is close to zero in $L^\infty$ and vanishes when multiplied by $n$, thus must vanish identically.  Since $\tilde{U}_{k-i}$ is connected, we conclude that $Q$ is invariant with respect to translations by this group.  Therefore, $Q$ is measurable with respect to $\WW_{i-1}/\tilde{U}_{k-i}$. From the second part of Lemma \ref{lem-goingup} and construction, this is a factor of $\WW_i$ and this completes the proof.
\end{proof}

Now we can prove the second part of Theorem \ref{divis}.  Let $\XX$ be a totally disconnected order $k$ ergodic $\Gamma$-system.  By the first part of this theorem, we can find an order $k$ ergodic extension $\YY$ of $\XX$ that is $k$-divisible.  This $\Gamma$-system $\YY$ need not be totally disconnected, however by Corollary \ref{cor-max-disconnected} it has a maximal totally disconnected factor $\ZZ$. Clearly $\ZZ$ is an extension of the totally disconnected factor $\XX$, and is also totally disconnected, order $k$, and ergodic by construction.  For any $1 \leq i \leq k$, $\Poly_{\leq i}(\YY)$ is divisible by construction, while $\Poly_{\leq i}(\YY)/\Poly_{\leq i}(\ZZ)$ is torsion-free thanks to Proposition \ref{maxtdfactor}.  This implies that $\Poly_{\leq i}(\ZZ)$ is divisible as a subgroup of a divisible group with torsion free quotient, and the claim follows.

\section{From divisibility to Weyl}\label{divis-weyl-sec}

The purpose of this section is to establish Theorem \ref{infty-weyl}.  The arguments used to prove this theorem are not needed elsewhere in the paper.

We will establish this result by combining two key propositions, which we now state.

\begin{proposition}[$k$-divisibility and Weyl imply $\infty$-divisibility]\label{k-infty}
Let $\Gamma$ be torsion-free, let $k \geq 1$, and let $\XX$ be an ergodic $k$-divisible $\Gamma$-system which is Weyl of order $k$.  Then $\XX$
is $\infty$-divisible.
\end{proposition}

\begin{proposition}[Totally disconnected $\infty$-divisible systems have trivial quasicohomology]\label{trivial-quasi}  Let $\Gamma$ be torsion-free, let $j \geq 0$ and $k \geq 0$, and let $\XX$ be an ergodic totally disconnected $\infty$-divisible $\Gamma$-system that is Weyl of order $j$.  Let $f \colon \Gamma \to \M(\XX,\T)$ be a $\T$-valued quasi-cocycle on $\XX$ of degree $k-2$ that is of type $k$.  Then $f$ is also a quasi-coboundary of degree $k-1$.
\end{proposition}

The notions of quasi-cocycle and quasi-coboundary are defined in Definition \ref{coh:def}(iv).  Proposition \ref{trivial-quasi} can be compared with \cite[Theorem 8.6]{btz}, which established a similar result when $\Gamma = \F_p^\omega$ in the high characteristic case $p \geq k$, though due to the torsion nature of $\F_p^\omega$ an additional ``line cocycle'' condition needed to be imposed on $f$ in that paper.

\begin{remark}
We note that Proposition \ref{trivial-quasi} fails if the divisibility hypothesis is dropped. For example take $k=2$, and let $Z\coloneqq \Z/4\Z\times \Z/4\Z$ be the rotational $\Z^2$-system with translation action $T^{(n,m)}(x_1,x_2) \coloneqq (x_1+n \mod 4 , x_2+m \mod 4)$.
Let $\rho \colon \Gamma \to \M(Z,\T)$ be the map 
$$ \rho_{\gamma_1,\gamma_2}(x_1,x_2) = - \frac{\gamma_1 x_2}{4}.$$
This is a quasi-cocycle of order $k-2=0$ since
$$ \rho_{\gamma+\gamma'}(x) - \rho_\gamma(x) - \rho_{\gamma'}(x+\gamma) = \frac{\gamma'_1 \gamma_2}{4}$$
is a polynomial of degree $0$ for every $\gamma,\gamma' \in \Z^2$.  If we then let $\rho' \colon \Gamma \to \M(Z,\T)$ be the map $\rho' = \rho + dF$ where $F(x_1,x_2) \coloneqq \frac{x_1^2 x_2}{4}$, then $\rho'$ is cohomologous to $\rho$ and is thus also a quasi-cocycle of order $k-2$.  It is also a quasi-coboundary of order $k-1=1$, so this is not yet a counterexample.  However, we observe that $\rho'_\gamma$ is $(2\Z)^2$-periodic.  Indeed,
\begin{align*}
 \partial_{(2,0)} \rho'_\gamma(x) &= \partial_{(2,0)} \rho_\gamma(x) + \partial_\gamma \partial_{(2,0)} F(x) \\
 &= 0 + \partial_\gamma 0 \\
 &= 0
 \end{align*}
 and
\begin{align*}
 \partial_{(0,2)} \rho'_\gamma(x) &= \partial_{(0,2)} \rho_\gamma(x) + \partial_\gamma \partial_{(0,2)} F(x) \\
 &= - \frac{\gamma_1}{2} + \partial_\gamma \frac{x_1^2}{2} \text{ mod } 1  \\
 &= - \frac{\gamma_1}{2} + \partial_\gamma \frac{x_1}{2} \text{ mod } 1 \\
 &= 0.
 \end{align*}
 Thus $\rho'$ descends to a quasi-cocycle $\rho''$ of order $k-2$ on the factor $W \coloneqq \Z/2\Z \times \Z/2\Z$.  We claim that $\rho''$ is not a quasi-coboundary of order $k-1$, for if we did have
 $$ \rho'' = dF' + q$$
 for some $F' \in \M(W,\T)$ and $q \colon \Z^2 \to \Poly_{\leq 1}(W)$, then on taking two derivatives we would have
 $$ \partial_s \partial_t \rho''_\gamma = \partial_\gamma \partial_s \partial_t F'$$
 for all $\gamma,s,t \in \Z^2$.  But
 \begin{align*}
 \partial_s \partial_t \rho'_\gamma(x)  &= \partial_\gamma \partial_s \partial_t F(x) \\
 &= \frac{s_1 t_1 \gamma_2 + s_1 t_2 \gamma_1 + s_2 t_1 \gamma_1}{2} \text{ mod } 1
 \end{align*}
 so on descending from $Z$ back to $W$ we conclude that
 $$  \partial_\gamma \partial_s \partial_t F'(x) = \frac{s_1 t_1 \gamma_2 + s_1 t_2 \gamma_1 + s_2 t_1 \gamma_1}{2} \text{ mod } 1.$$
 On the other hand, since $\partial_{2s} = \partial_s \partial_s + 2\partial_s$ and $\partial_{2t} = \partial_t \partial_t + 2 \partial_t$ annihilate $F'$, we have the identity
 $$ \partial_s \partial_s \partial_t F' = -2 \partial_s \partial_t F' = \partial_s \partial_t \partial_t F'$$
and hence
$$
 \frac{s_1 s_1 t_2 + s_1 s_2 t_1 + s_2 s_1 t_1}{2} =  \frac{s_1 t_1 t_2 + s_1 t_2 t_1 + s_2 t_1 t_1}{2} \text{ mod } 1$$
 for all $s,t \in \Z^2$, which is false.
\end{remark}

Assume Propositions \ref{k-infty}, \ref{trivial-quasi} for now.  Proposition \ref{trivial-quasi} has the following consequence:

\begin{corollary}[Totally disconnected $\infty$-divisible base implies Weyl]\label{cor-weyl}  Let $\Gamma$ be torsion-free, let $k \geq 1$, and let $\XX$ be an ergodic $\Gamma$-system of order $k$.  If 
$\ZZ^{k-1}(\XX)$ is totally disconnected, $\infty$-divisible, and Weyl of order $k-1$, then $\XX$ is Weyl of order $k$.
\end{corollary}

\begin{proof} By Proposition \ref{abelext} we may write $\XX = \ZZ^{k-1}(\XX) \rtimes_\rho U$ for some compact abelian metrizable $U$ and some cocycle $\rho$ of type $k$.  For any character $\xi \in \hat U$, Proposition \ref{type}(ii) implies that $\xi \circ \rho$ is a $\T$-valued cocycle on $\XX$ of type $k$. Since a cocycle is trivially also a quasi-cocycle we may apply Proposition \ref{trivial-quasi}. From that Proposition, applied with $k+1$ in place of $k$ and $j=k-1$, we see that $\xi \circ \rho$ is a $\T$-valued quasi-coboundary on $\XX$ of degree $k-2$.
By Proposition \ref{typecircle:prop} and the $\infty$-divisibility hypothesis, $\rho$ is a $U$-valued quasi-coboundary on $\XX$ of degree $k-1$.  By Remark \ref{cohiso}, we may thus assume without loss of generality that $\rho_\gamma$ is a polynomial of degree $\leq k-1$ for all $\gamma \in \Gamma$.  Since $\ZZ^{k-1}(\XX)$ was Weyl of order $k-1$, this implies that $\XX$ is Weyl of order $k$ as required.
\end{proof}

We can now use Proposition \ref{k-infty} and Corollary \ref{cor-weyl} to prove Theorem \ref{infty-weyl} as follows.  We induct on $k$.  The $k=1$ 
case of Theorem \ref{infty-weyl} follows from Proposition \ref{k-infty} and the fact (immediate from \eqref{structuregroups}) that all order $1$ systems are Weyl.  Now suppose that $k > 1$ and that the theorem has already been proven for $k-1$.  Since $\XX$ is ergodic, totally disconnected and $k$-divisible, $\ZZ^{k-1}(\XX)$ is ergodic, totally disconnected and $k-1$-divisible, hence by induction hypothesis $\ZZ^{k-1}(\XX)$ is $\infty$-divisible and Weyl of order $k-1$.  By Corollary \ref{cor-weyl}, $\XX$ is Weyl of order $k$, and then by Proposition \ref{k-infty}, $\XX$ is $\infty$-divisible, thus closing the induction.

It remains to establish Proposition \ref{k-infty} and Proposition \ref{trivial-quasi}.  This will be done in the next two subsections.

\subsection{Obtaining $\infty$-divisibility}

We now prove Proposition \ref{k-infty}.  Let $\Gamma$, $k$, $\XX$ be as in that proposition.  We assume inductively that the claim holds for $k-1$ (this is trivial for $k=1$).  As $\XX$ is Weyl, we can write
$$ \XX = \ZZ^{k-1}(\XX) \rtimes_\rho U$$
for some compact abelian metrizable group $U$ and some $U$-valued cocycle $\rho$ on $\ZZ^{k-1}(\XX)$ which is polynomial of degree at most $k-1$.  By Proposition \ref{ppfacts}(ii), $\ZZ^{k-1}(\XX)$ is $k-1$-divisible and hence (by induction hypothesis) $\infty$-divisible.  The $k$-divisibility of $\XX$ also gives an important additional property on $U$:

\begin{lemma}\label{U-divis} The Pontryagin dual $\hat U$ of $U$ is divisible.
\end{lemma}

\begin{proof}
    Let $\xi \in \hat U$ and let $n\geq 1$.  Observe that the vertical coordinate function $\mathrm{u} \in \M(\XX,U)$ defined by $\mathrm{u} \colon (y,u) \to u$ is a polynomial of degree $\leq k$, since $\partial_\gamma \mathrm{u}(y,u) = \rho_\gamma(y)$ for all $y \in \ZZ^{k-1}(\XX)$ and $u \in U$.  In particular, $\xi \cdot \mathrm{u} \in \Poly_{\leq k}(\XX)$.  By $k$-divisibility, we can write $\xi \cdot \mathrm{u} = nP$ for some $P \in \Poly_k(\XX)$.  By Proposition \ref{ppfacts}(iii), $\partial_u P$ is degree zero and thus equal to a constant $\chi(u)$ for each $u \in U$.  From the cocycle equation \eqref{cocycle-ident}, $\chi$ is a character, thus $\chi \in \hat U$.  Since $\partial_u (\xi \cdot \mathrm{u}) = \xi  \cdot u$, we conclude that $\xi = n \chi$, and so $\hat U$ is divisible as required.
\end{proof}

Let $P \in \Poly_{\leq m}(\XX)$ for some $m \geq 1$, and let $n \geq 1$.  We wish to show that $P-nQ = 0$ for some $Q \in \Poly_{\leq m}(\XX)$.  On the other hand, for each $u \in U$, the ``vertical derivative'' operator $\partial_u$ lowers degree by $k$, thanks to Proposition \ref{ppfacts}(iii).  In particular, we have $\partial_{u_1} \dots \partial_{u_\ell} P = 0$ whenever $k \ell > m$ and $u_1,\dots,u_\ell \in U$.  The claim now follows from iterating the following claim downwards in $\ell$ (starting with the largest $\ell$ for which $k \ell \leq m$).

\begin{proposition}[Lowering the vertical degree]\label{vert}  Let $\Gamma, k, \XX, P, m, n$ be as above.  Suppose that there is $\ell \geq 0$ with $k \ell \leq m$ such that
$$ \partial_{u_1} \dots \partial_{u_{\ell+1}} P = 0$$
for all $u_1,\dots,u_{\ell+1} \in U$.  Then there exists $Q \in \Poly_{\leq m}(\XX)$ such that
$$ \partial_{u_1} \dots \partial_{u_{\ell}} (P-nQ) = 0$$
for all $u_1,\dots,u_{\ell} \in U$.
\end{proposition}

It remains to establish Proposition \ref{vert}.   Let the hypotheses be as in that proposition, then for any $u_1,\dots,u_\ell \in U$, the function $\partial_{u_1} \dots \partial_{u_{\ell}} P$ is $U$-invariant, and is also a polynomial of degree $\leq m-k\ell$ by Proposition \ref{ppfacts}(iii).  Thus it can be identified with an element of $\Poly_{\leq m-k\ell}( \ZZ^{k-1}(\XX))$.  By the cocycle identity \eqref{cocycle-ident}, we see that the expression $\partial_{u_1} \dots \partial_{u_{\ell}} P$ is a homomorphism in each of the $u_1,\dots,u_\ell$, and is also symmetric.  Thus we have constructed an element $\nabla_U^\ell P$ of $\SML_\ell(U, \Poly_{\leq m-k\ell}( \ZZ^{k-1}(\XX)))$, defined by the formula
$$ \nabla_U^\ell P \colon (u_1,\dots,u_\ell) \mapsto \partial_{u_1} \dots \partial_{u_{\ell}} P.$$
By induction hypothesis, the group $\Poly_{\leq m-k\ell}( \ZZ^{k-1}(\XX))$ is divisible; by Proposition \ref{ppfacts}(i), this group is also
a countable union of copies of $\T$.  From this, Lemma \ref{U-divis}, and Proposition \ref{symdiv}(iii), we conclude that $\SML_\ell(U, \Poly_{\leq m-k\ell}( \ZZ^{k-1}(\XX)))$ is divisible.  In particular, we can write
$$ \nabla_U^\ell P = n \ell! b$$ 
for some $b \in \SML_\ell(U, \Poly_{\leq m-k\ell}( \ZZ^{k-1}(\XX)))$.

We will shortly prove the following claim.

\begin{lemma}[Polynomial integration lemma]\label{integ}  Let the notation and hypotheses be as above.  Then for any $b \in \SML_\ell(U, \Poly_{\leq m-k\ell}( \ZZ^{k-1}(\XX)))$, there
exists $Q \in \Poly_{\leq m}(\XX)$ such that $\nabla_U^\ell Q = \ell! b$.
\end{lemma}

Assuming this lemma, we conclude that
$$ \nabla_U^\ell (P-nQ) = 0$$
and Proposition \ref{vert} follows.

It remains to prove Lemma \ref{integ}.  We make the explicit choice
$$ Q(y,u) \coloneqq b( u,\dots, u)(y)$$
for $u \in U$ and $y \in \ZZ^{k-1}(\XX)$, where $u$ is repeated $\ell$ times.  By the symmetric multilinear nature of $b$, it is clear that $\nabla_U^\ell Q = \ell! b$.  It remains to show that $Q$ is a polynomial of degree at most $m$.  We can achieve this from the following simple algebraic lemma, which asserts that the (tensor) product of a polynomial of degree $m$ and a polynomial of degree $n$ is a polynomial of degree $m+n$:

We can view $b$ as an element of $\Poly_{\leq m-k\ell}( \ZZ^{k-1}(\XX), \SML_\ell(U,\T) )$ since this is canonically identified with $\SML_\ell(U, 
\Poly_{\leq m-k\ell}( \ZZ^{k-1}(\XX) ))$.  In particular, $b$ lifts to an element of $\Poly_{\leq m-k\ell}( \XX, \SML_\ell(U,\T) )$.
Meanwhile, recall from the proof of Lemma \ref{U-divis} that the vertical coordinate function $\mathrm{u} \colon (y,u) \mapsto u$ is an element of $\Poly_{\leq k}(\XX, U)$.   Applying Lemma \ref{Tensor} repeatedly, we conclude that
$$ b \otimes \mathrm{u} \otimes \dots \otimes \mathrm{u} \in \Poly_{\leq m}(\XX, \SML_\ell(U,\T) \otimes U^{\otimes \ell} )$$
where $\mathrm{u}$ appears $\ell$ times on the right-hand side.  On the other hand, by the universal nature of the tensor product, we see that $Q$ is the image of $b \otimes \mathrm{u} \otimes \dots \otimes \mathrm{u}$ under the canonical homomorphism from $\SML_\ell(U,\T) \otimes U^{\otimes \ell}$ to $\T$, that maps $c \otimes u_1 \otimes \dots \otimes u_\ell$ to $c(u_1,\dots,u_\ell)$ for any $c \in \SML_\ell(U,\T)$ and $u_1,\dots,u_\ell \in U$.  Hence $Q \in \Poly_{\leq m}(\XX,\T)$ as desired.

\subsection{Ensuring trivial quasicohomology}

We now prove Proposition \ref{trivial-quasi}.  We induct on $k$.  The case $k=0$ is trivial (type $0$ functions are coboundaries by definition), and when $k=1$, the claim follows from Proposition \ref{type}(vi), so assume that $k \geq 2$ and that the claim has already been proven for $k-1$.

Now we induct on $j$.  When $j=0$, $\XX$ is a point, and the claim is trivial, so suppose that $j \geq 1$ and the claim has already been proven for $j-1$ (with $k$ held fixed).  As in \cite{btz}, we divide into the low-order case $j \leq k$ and the high-order case $j>k$.

We begin with the low-order case $j \leq k$.  By Proposition \ref{abelext} we may write
$$ \XX = \ZZ^{j-1}(\XX) \rtimes_\rho U$$
for some compact abelian metrizable $U$ and some $U$-valued cocycle $\rho$ on $\ZZ^{j-1}(\XX)$ of type $j$. By Proposition \ref{ppfacts}(iii) and Proposition \ref{type}(iv), we have that for every $u\in U$, $\partial_u f$ is a $\T$-valued quasicocycle on $\XX$ of degree $k-2-j$ that is of type $k-j$.  By induction hypothesis, we conclude that $\partial_u f$ is a quasi-coboundary of degree $k-1-j$, thus there exists $q_u \colon \Gamma \to \Poly_{\leq k-1-j}(\XX)$ and $F_u \in \M(\XX,\T)$ such that the ``higher order Conze--Lesigne equation''
\begin{equation}\label{hcl}
\partial_u f = dF_u + q_u
\end{equation}
is satisfied.

Motivated by this, we recall from Remark \ref{cocyc-skew} that we have the semidirect product $U \ltimes \M(\XX,\T)$, which when equipped with the product topology becomes a Polish group.  Inside this group we introduce the subgroup
\begin{equation}\label{H-def}
H \coloneqq \{ (u,F) \in U \ltimes \M(\XX,\T): \partial_u f - dF \in \Poly_{\leq k-1-j}(\XX)^\Gamma \}
\end{equation}
of pairs $(u,v)$ such that $\partial_u f_\gamma - \partial_\gamma F$ is a polynomial of degree at most $k-1-j$ for all $\gamma \in \Gamma$.  It is easy to check from the cocycle equation \eqref{cocycle-ident} that this is a closed subgroup of $U \ltimes \M(\XX,\T)$, and the higher order Conze--Lesigne equation \eqref{hcl} can be restated as the assertion that the homomorphism $(u,F) \mapsto u$ is a surjection from $H$ to $U$.  The kernel of this homomorphism is the set of pairs $(0,F)$ such that $dF \in  \Poly_{\leq k-1-j}(\XX)^\Gamma$, or equivalently that $F \in \Poly_{\leq k-j}(\XX)$.  Thus we have a short exact sequence
\begin{equation}\label{short-exact}
0 \to \Poly_{\leq k-j}(\XX) \to H \to U \to 0.
\end{equation}
Suppose for the moment that this sequence splits in the category of topological groups, thus there exists a continuous homomorphism $u \mapsto (u,F_u)$ from $U$ to $H$.  Thus $F_u \in \M(\XX,\T)$ varies continuously with $U$ and obeys the cocycle equation
$$ F_{u_1+u_2} = F_{u_1} + F_{u_2} \circ V^{u_1}$$
for all $u_1,u_2 \in U$.  By Proposition \ref{C.8} we conclude that there exists $\Phi \in \M(\XX,\T)$ such that $F_u = \partial_u \Phi$ for all $u \in U$.  From \eqref{H-def} we then have
$$ \partial_u (f - d\Phi) \in \Poly_{\leq k-1-j}(\XX)^\Gamma$$
for all $u \in U$.  Thus, if we let $\tilde f \coloneqq f - d\Phi$, then $\tilde f$ is cohomologous to $f$ (and thus still a quasi-cocycle of degree $k-2$ of type $k$), and we have
\begin{equation}\label{puf}
\partial_u \tilde f_\gamma = q_{u,\gamma}
\end{equation}
for all $u \in U$, $\gamma \in \Gamma$, and some $q_{u,\gamma} \in \Poly_{\leq k-1-j}(\XX)$ that depends continuously on $u$.  From the cocycle equation we see that
$$ q_{u_1+u_2,\gamma} = q_{u_1,\gamma} + q_{u_2,\gamma} \circ V^{u_1}$$
for all $u_1,u_2 \in U$ and $\gamma \in \Gamma$. By Proposition \ref{PI1}, we can thus find $Q_\gamma \in \Poly_{\leq k-1}(\XX)$ for each $\gamma \in\Gamma$ such that $q_{u,\gamma} = \partial_u Q_\gamma$; indeed we can take
\begin{equation}\label{quu}
Q_\gamma(y,u u_0) = q_{u,\gamma}(y,u_0)
\end{equation}
for some generic $u_0 \in U$.  We conclude that $\partial_u (\tilde f_\gamma - Q_\gamma)=0$, thus
\begin{equation}\label{tf}
\tilde f = Q + \overline{f} \circ \pi
\end{equation}
for some $\overline{f} \colon \Gamma \to \M(\ZZ^{j-1}(\XX),\T)$, where $\pi \colon \XX \to \ZZ^{j-1}(\XX)$ is the factor map.  For future reference we observe that this argument is still valid if $\ZZ^{j-1}(\XX)$ were replaced by any intermediate factor between $\ZZ^{j-1}(\XX)$ and $\XX$.

As observed in \cite[Section 8.3]{btz}, the choice \eqref{quu} makes $Q$ into a quasi-cocycle of degree $k-2$.  We repeat the calculation here, taking the opportunity to correct some misprints in that paper.  It suffices to show that for any $\gamma_1,\gamma_2 \in \Gamma$, the expression
$$ Q_{\gamma_1+\gamma_2}(y, uu_0) - Q_{\gamma_1}(y,uu_0) - (Q_{\gamma_2} \circ T^{\gamma_1})(y,uu_0)$$
is a polynomial function of $(y,u)$ of degree at most $k-2$.  Since $\partial_{\gamma_1} Q_{\gamma_2}$ is already of degree at most $k-2$, it suffices to establish this claim for
$$ Q_{\gamma_1+\gamma_2}(y, uu_0) - Q_{\gamma_1}(y,uu_0) - Q_{\gamma_2}(y,uu_0)$$
which by \eqref{quu} is equal to
$$ q_{u,\gamma_1+\gamma_2}(y,u_0) - q_{u,\gamma_1}(y,u_0) - q_{u,\gamma_2}(y,u_0).$$
On the other hand, since $\tilde f$ is a quasi-cocycle of degree $k-2$, we see from \eqref{puf} and Proposition \ref{ppfacts}(iii) that $q_u$ is a quasi-cocycle of degree at most $k-2-j$, thus
$$ q_{u,\gamma_1+\gamma_2} - q_{u,\gamma_1} - q_{u,\gamma_2} \circ T^{\gamma_1}$$
is a polynomial of degree at most $k-2-j$; it is also a cocycle in $u$.  Applying \cite[Lemma 8.14(ii)]{btz}, we conclude (for generic $u_0$) that
$$\left(q_{u,\gamma_1+\gamma_2} - q_{u,\gamma_1} - q_{u,\gamma_2} \circ T^{\gamma_1}\right)(y,u_0)$$
is a polynomial function of $(y,u)$ of degree at most $k-2$.  Also, applying Proposition \ref{PI1} again to $\partial_{\gamma_1} q_{u,\gamma_2}$, we see (again for generic $u_0$) that
$$ (\partial_{\gamma_1} q_{u,\gamma_2})(y,u_0)$$
is also a polynomial function of $(y,u)$ of degree at most $k-2$.  Summing, we obtain the claim.  Again, this argument would remain valid if $\ZZ^{j-1}(\XX)$ were replaced by any intermediate factor between $\ZZ^{j-1}(\XX)$ and $\XX$.

Clearly $Q$ is also a quasi-coboundary of degree $k-1$, and thus of type $k$.  We conclude that $\pi^* \overline{f}$ is a quasi-cocycle on $\XX$ of degree $k-2$ and type $k$.  By \cite[Proposition 8.11]{btz}, this implies that $\overline{f}$ is also a quasi-cocycle on $\ZZ^{j-1}(\XX)$ of degree $k-2$ and type $k$.  By induction hypothesis, $\overline{f}$ is a quasi-coboundary on $\ZZ^{j-1}(\XX)$ of degree $k-1$; pulling back by $\pi$ and using \eqref{tf} we conclude that $\tilde f$ and hence $f$ are also quasi-coboundaries on $\XX$ of degree $k-1$, as claimed.

A modification of the above argument  (already implicit in \cite{btz}) allows us to utilize ``partial'' splittings of the sequence as follows.  If the group $U$ has a product structure $U_1 \times U_2$ for compact abelian metrizable subgroups $U_1,U_2$, and we let $H_1$ be the closed subgroup of $H$ consisting of those pairs $(u_1,F_1) \in H$ with $u_1 \in U_1$, then of course we also have a short exact sequence
\begin{equation}\label{h1}
0 \to \Poly_{\leq k-j}(\XX) \to H_1 \to U_1 \to 0.
\end{equation}
Also, letting $\pi_1 \colon U \to U_1$, $\pi_2 \colon U \to U_2$ be the projection maps, we may factorize
$$ \XX = (\ZZ^{j-1}(\XX) \rtimes_{\pi_2 \circ \rho} U_2) \rtimes_{\pi_1 \circ \rho \circ \pi} U_1$$
where $\pi \colon \ZZ^{j-1}(\XX) \rtimes_{\pi_2 \circ \rho} U_2 \to \ZZ^{j-1}(\XX)$ is the factor map. The remark is then that if the sequence \eqref{h1} splits, and the Proposition \ref{trivial-quasi} is already known to hold for $\ZZ^{j-1}(\XX) \rtimes_{\pi_2 \circ \rho} U_2$, then Proposition \ref{trivial-quasi} holds for $\XX$.  This follows by repeating the previous arguments but with $\ZZ^{j-1}(\XX)$ replaced by $\ZZ^{j-1}(\XX) \rtimes_{\pi_2 \circ \rho} U_2$, which is intermediate between $\ZZ^{j-1}(\XX)$ and $\XX$.  The practical upshot of this is that once we achieve a partial splitting on a direct factor $U_1$ of $U$, we can then reduce $U$ without loss of generality to the complementary factor $U_2$.

As a first use of this remark, we use the argument from\footnote{We will actually use a variant of that argument, relying on the Pettis lemma (a variant of the Steinhaus lemma) in place of Lusin's theorem, in order to avoid invoking measurable selection theorems.} \cite[Proposition 6.1]{btz} to obtain a splitting on an open subgroup $V'$ of $U$.  Let $\eps>0$ be a small number to be chosen later.  As $\M(\XX,\T)$ is separable, it can be covered by countably many balls $B( f_n, \eps)$, $n \in \N$ in the $L^2$ metric for some $f_n \in \M(\XX,\T)$.  For each $n$, the orbit of $U$ on $f_n$ is compact and the action of $U$ is strongly continuous, so we can cover $U$ by finitely many open sets $U_{n,j}$ such that $f_n \circ V_u$ only varies by at most $\eps$ in the $L^2$ metric for $u \in U_{n,j}$.  For each $n,j$, let $E_{n,j}$ be the set of those $u \in U_{n,j}$ which lift to at least one element $(u,F_u)$ in $H$ with $F_u \in B(f_n,\eps)$.  Then the $E_{n,j}$ form a countable cover of $U$ by analytic sets.  By the Baire category theorem, at least one of the $E_{n,j}$ is non-meager.

Since $E_{n,j}$ is analytic, it has the property of Baire by the Lusin--Sierpinski theorem (see e.g., \cite[Theorem 21.6]{kechris1995classical}), thus it agrees outside of a meager set with an open set $U'$, which must be non-empty since $E_{n,j}$ is non-meager.  By the Baire category theorem, $E_{n,j}-E_{n,j}$ must contain $U'-U'$, which is an open neighborhood of the identity (this is the Pettis lemma, see \cite{pettis} or \cite[Theorem 9.9]{kechris1995classical}).  Since $U$ is totally disconnected, we conclude (see e.g., \cite[Proposition 1.1.3]{neukirch}) that $E_{n,j}-E_{n,j}$ contains an open subgroup $V$ of $U$.

If $v \in V$, then we arbitrarily write $v = u-u'$ for some $u,u' \in E_{n,j}$.  The element $(v,\tilde F_v) \coloneqq (u,F_u) (u',F_{u'})^{-1}$ then lies in $H$, with $\tilde F_v$ within $5\eps$ of $0$ by the triangle inequality.  If we then let $K$ denote all the pairs $(v,F)$ in $H$ with $v \in V$ and $F$ within $20\eps$ of $\tilde F_v + c$ for some constant $c \in \T$, we see from Proposition \ref{ppfacts}(i) that (if $\eps$ is chosen to be sufficiently small) $K$ is a closed subgroup of $H$ and that one has the short exact sequence
$$ 0 \to \T \to K \to V \to 0.$$
In particular we have a continuous bijective homomorphism from the compact group $K/\T$ to the compact abelian group $V$, which must then be an isomorphism of topological groups.  In particular $K/\T$ is abelian, and the map $v \mapsto (v, \tilde F_v) \text{ mod } \T$ is continuous.  Taking commutators (which annihilate $\T$), the map $(v_1,v_2) \mapsto [(v_1,\tilde F_{v_1}), (v_2,\tilde F_{v_2})]$ will map to the interval $(-0.1,0.1) \text{ mod } 1$ in $\T$ if $v_1,v_2$ lie in some sufficiently small open subgroup $V'$ of $V$.  But this map is a homomorphism in $v_1,v_2$, and this interval contains no non-trivial subgroups of $\T$, thus this map in fact vanishes.  Thus if we let $K'$ be the preimage of $V'$ in $K$, then $K'$ is abelian and we have the short exact sequence
$$ 0 \to \T \to K' \to V' \to 0$$
of compact abelian groups.  By Pontryagin duality we see that $\T$ is an injective object in the category of compact abelian groups, so this sequence splits, thus we may find a continuous homomorphism $v \mapsto (v, F'_v)$ from $V'$ to $K'$.  This of course also induces a splitting of the short exact sequence
\begin{equation}\label{vp}
0 \to \Poly_{\leq k}(\XX) \to \{ (u,F) \in H: u \in V' \} \to V' \to 0.
\end{equation}
We would now like to use the previous remark, but we run into a new\footnote{This was not an issue in \cite{btz}, since the structure group $U$ was $p$-torsion and so \cite[Lemma D.2]{btz} was available.} technical difficulty that $V'$ is not necessarily a direct factor of $U$.    However, we can still obtain strong structural control on $U$ as follows.  By the Sylow theorem for profinite abelian groups (see e.g., \cite[Proposition 2.3.8]{ribes-z} or \cite[Corollary 8.8]{hofmann}), $U$ can be expressed as the direct product of its maximal $p$-profinite subgroups $U_p$.  By Lemma \ref{U-divis}, $U$ has divisible dual, and thus so do the subgroups $U_p$. A divisible $p$-profinite group is isomorphic to a (possibly infinite) direct product of $p$-Pr\"ufer groups $\Z[\frac{1}{p}]/\Z = \hat \Z_p$ (see e.g., \cite[Theorem 23.1]{Fuchs}), so we conclude that $U$ is the product of finite or infinitely many $p$-adic groups $\Z_p$.  As $V'$ is open in the product topology, it must contain the product of all but finitely many of these factors.  Thus there exists a direct product representation $U = U_1 \times U_2$ where $U_2$ is the product $U_2 = \Z_{p_1} \times \dots \times \Z_{p_m}$ of finitely many $p$-adic groups (where the primes $p_1,\dots,p_m$ are not necessarily distinct), and $U_1$ is a compact subgroup of $V'$.  In particular we have a splitting of the short exact sequence
$$ 0 \to \Poly_{\leq k}(\XX) \to \{ (u,F) \in H: u \in U_1 \} \to U_1 \to 0.$$
In view of the previous remark, we may now assume without loss of generality that $U=U_2$, that is to say $U$ is a product
$$ U = \Z_{p_1} \times \dots \times \Z_{p_m}$$
of $p$-adic groups.

We now proceed by induction on $m$.  If $m=0$ then there is nothing to prove.  For $m \geq 1$, it again suffices by the preceding remark and the inductive hypothesis to obtain a splitting
$$ 0 \to \Poly_{\leq k}(\XX) \to \{ (u,F) \in H: u \in \Z_{p_m} \} \to \Z_{p_m} \to 0$$
of the final factor $\Z_{p_m}$.  To simplify the notation we will now abbreviate $p_m$ as $p$.

By the previous arguments, we have already obtained a splitting \eqref{vp} for some open subgroup $V'$ of $U$.  This open subgroup $V$ need not contain all of $\Z_{p}$, but it must contain an open subgroup $p^n \Z_{p}$ of $\Z_{p}$ for some $n \geq 0$, thus there is a continuous homomorphism $u \mapsto (u,F_u)$ from $p^n \Z_{p}$ to some compact abelian subgroup $K$ of $H$.  We would like to extend this homomorphism to $\Z_{p}$.  If $H$ were abelian we could appeal to Lemma \ref{divis-inject} and the divisibility hypothesis to conclude (after quotienting out by $p^n \Z_p$).  Unfortunately, $H$ is not abelian in general, but we can instead take advantage of the fact that $\Z_{p}$ is \emph{monothetic} in the sense that it contains a dense cyclic group, namely the integers $\Z$ which are generated by the unit $1$.  From \eqref{short-exact} we can lift the unit $1$ to an element $(1, F'_1)$ of $H$, where $F'_1$ is unique modulo $\Poly_{\leq k}(\XX)$.  Raising to the power $p^n$ and using \eqref{short-exact} again, we have
$$ (1,F'_1)^{p^n} = (p^n, F_{p^n} + Q)$$
for some polynomial $Q \in \Poly_{\leq k}(\XX)$.  We can begin to eliminate $Q$ by using the divisibility hypothesis to write $Q = p^n P$, and then one can calculate that
$$ (1,(F'_1-P))^{p^n} = (p^n, F_{p^n} + Q')$$
for some polynomial $Q' \in \Poly_{\leq k-1}(\XX)$.  Iterating this procedure, we can eventually find a lift $(1,F''_1)$ of $1$ in $H$ such that
$$ (1,F''_1)^{p^n} = (p^n, F_{p^n})$$
which by the homomorphism property of $u \mapsto (u,F_u)$ implies that
\begin{equation}\label{1f1}
 (1,F''_1)^{mp^n} = (mp^n, F_{mp^n}).
 \end{equation}
In particular, $(1,F''_1)$ commutes with $(u,F_u)$ for $u \in p^n\Z$, and hence for $u \in p^n \Z_p$ by continuity.  We can therefore extend the homomorphism $u \mapsto (u,F_u)$ from $p^n \Z_p$ to $\Z_p = p^n\Z_p + \Z$ by the formula
$$ u+m \mapsto (1,F''_1)^m (u,F_u)$$
for any $u \in p^n \Z_p$ and $m \in \Z$, since \eqref{1f1} ensures that this definition is well-defined.  This homomorphism is already continuous on the open finite index subgroup $p^n \Z_p$ of $\Z_p$, hence is continuous on $\Z_p$ as well.  This gives the desired splitting.

Finally, we treat the high-order case $j > k$.  This largely follows the arguments in \cite[Section 8.5]{btz} and we only sketch the details here.  Define a \emph{good tuple} on $\XX$ to be a tuple $(f_\omega)_{\omega \in \{0,1\}^k}$ of $\T$-valued quasi-cocycles $f_\omega  \colon \Gamma \to \M(\XX,\T)$ of order $k-2$, such that $\sum_{\omega \in \alpha} f_\omega \circ \pi_\omega$ is a cocycle on $\XX^{[k]}$ for every face $\alpha$ of $\{0,1\}^k$, and $\sum_{\omega \in \{0,1\}^k} f_\omega \circ \pi_\omega$ is a coboundary on $\XX^{[k]}$, where $\pi_\omega \colon X^{[k]} \to X$ are the coordinate projections.  Since $f$ is of type $k$, we see that $((-1)^{|\omega|} f)_{\omega \in \{0,1\}^k}$ is a good tuple.  Following the arguments in \cite[Section 8.5]{btz}, we can deduce the $j>k$ case from the $j=k$ case once we establish the following analogue of \cite[Proposition 8.12]{btz}:

\begin{proposition}[Descent of type]\label{descent} Let $j > k \geq 1$, let $\Gamma$ be torsion-free, and let $\XX$ be an ergodic totally disconnected $\Gamma$-system that is Weyl of order $j$.  Let $(f_\omega)_{\omega \in \{0,1\}^k}$ be a good tuple on $\XX$.  Then there exists a good tuple $(\tilde f_\omega)_{\omega \in \{0,1\}^k}$ on $\ZZ^{j-1}(\XX)$ such that for each $\omega \in \{0,1\}^k$, $f_\omega$ is cohomologous to $\tilde f_\omega \circ \pi$, where $\pi \colon \XX \to \ZZ^{j-1}(\XX)$ is the factor map.
\end{proposition}

To prove this proposition we mimic the arguments used to prove \cite[Proposition 8.12]{btz}.  As $\XX$ is Weyl and totally disconnected, we can write $\XX = \ZZ^{j-1}(\XX) \rtimes_\rho U$ for some compact abelian totally disconnected metrizable $U$.

Arguing as in the proof of \cite[Lemma 8.8]{btz} and using the high order hypothesis $j>k$, we see that for every $\omega \in \{0,1\}^k$ and $u \in U$, that $\partial_u f_\omega$ is a coboundary, thus there exists $F_{u,\omega} \in \M(\XX,\T)$ such that
$$ \partial_u f_\omega = d F_{u,\omega}.$$
Thus if we let $H_\omega \leq U \ltimes \M(\XX,\T)$ denote the closed subgroup of pairs $(u,F) \in U \ltimes \M(\XX,\T)$ obeying the Conze--Lesigne type equation $\partial_u f_\omega = dF$, then by arguing as in in the low-order case we have a short exact sequence
$$ 0 \to \T \to H_\omega \to U \to 0.$$
Repeating previous arguments, we can find an open subgroup $V'$ of $U$ for which the short exact sequence
$$ 0 \to \T \to \{ (u,F) \in H_\omega: u \in V' \} \to V' \to 0$$
splits for every $\omega \in \{0,1\}^k$.  Since $U$ is totally disconnected and has divisible dual as before, we can find a direct product $U = U_1 \times U_2$ with $U_2$ the product of finitely many $p$-adic groups such that $V'$ contains $U_1$.  Then the short exact sequence
$$ 0 \to \T \to \{ (u,F) \in H_\omega: u \in U_1 \} \to U_1 \to 0$$
splits, so by Lemma \ref{C.8} as before there exists $\Phi_\omega \in \M(\XX,\T)$ such that $(u,\partial_u \Phi_\omega) \in H_\omega$ for all $u \in U_1$ and $\omega \in \{0,1\}^k$, thus by definition of $H$ we see that $f_\omega - d\Phi_\omega$ is $U_1$-invariant.
Writing 
$$ \XX = (\ZZ^{j-1}(\XX) \rtimes_{\rho_2} U_2) \rtimes_{\rho_1} U_1$$
for suitable cocycles $\rho_1, \rho_2$, we thus have
$$ f_\omega - d\Phi_\omega = \tilde f_\omega \circ \pi_2$$
for some tuple $(\tilde f_\omega)_{\omega \in \{0,1\}^k}$ of functions $\tilde f_\omega \colon\Gamma \to \M(\ZZ^{j-1}(\XX) \rtimes_{\rho_2} U_2,\T)$, where $\pi_2 \colon \XX \to \ZZ^{j-1}(\XX) \rtimes_{\rho_2} U_2$ is the factor map.  Since $\tilde f_\omega \circ \pi_2$ is a quasi-cocycle of order $k-2$ on $\XX$, $\tilde f_\omega$ is a quasi-cocycle of order $k-2$ on $\ZZ^{j-1}(\XX) \rtimes_{\rho_2} U_2$; similarly, for any face $\alpha \in \{0,1\}^k$, $\sum_{\omega \in \alpha} \tilde f_\omega \circ \pi_\omega$ is a cocycle on $(\ZZ^{j-1}(\XX) \rtimes_{\rho_2} U_2)^{[k]}$.  We cannot quite conclude that $\sum_{\omega \in \{0,1\}^k} \tilde f_\omega \circ \pi_\omega$ is a coboundary on $(\ZZ^{j-1}(\XX) \rtimes_{\rho_2} U_2)^{[k]}$, but we know that $\sum_{\omega \}in \{0,1\}^k} \tilde f_\omega \circ \pi_2 \circ \pi_\omega$ is a coboundary on $X^{[k]}$.  Because of the high order hypothesis, we have the decomposition
$$ X^{[k]} = (\ZZ^{j-1}(\XX) \rtimes_{\rho_2} U_2)^{[k]} \rtimes_{\rho_1^{[k]}} U_1^{[k]}$$
thanks to \cite[Lemma A.36]{btz}, and hence by \cite[Lemma B.11]{btz} (arguing as in the proof of \cite[Proposition 8.12]{btz}) we can find characters $\chi_\omega \in \hat U$ for $\omega \in \{0,1\}^k$ such that
$$ \sum_{\omega \in \{0,1\}^k} (\tilde f_\omega - \chi_\omega \cdot \rho_1) \circ \pi_\omega$$
is a coboundary on $(\ZZ^{j-1}(\XX) \rtimes_{\rho_2} U_2)^{[k]}$.  We conclude that the tuple $(\tilde f_\omega - \chi_\omega \cdot \rho_1)_{\omega \in \{0,1\}^k}$ is a good tuple on $\ZZ^{j-1}(\XX) \rtimes_{\rho_2} U_2$.  If we can establish Proposition \ref{descent} for this tuple, then we obtain Proposition \ref{descent} for the original tuple $(f_\omega)_{\omega \in \{0,1\}^k}$.  The practical upshot of this is that we can replace $U$ with $U_2$, so as before we can assume without loss of generality that
$$ U = \Z_{p_1} \times \dots \times \Z_{p_m}$$
is a product of of $p$-adic groups.

As before, it will now suffice by induction on $m$ to show that the short exact sequence
$$ 0 \to \T \to \{ (u,F) \in H_\omega: u \in \Z_{p_m} \} \to \Z_{p_m} \to 0$$
splits.  But this follows exactly as in the low order case.

\section{Obtaining Weyl extensions}\label{weyl-sec}

In this section we establish Corollary \ref{Weyl extensions}, basically by ``chasing'' the arrows in the upper half of Figure \ref{fig:main}.

We begin with $(i)$.  If $\Gamma$ is torsion-free and $\XX$ is a totally disconnected ergodic $\Gamma$-system of order $k$, then by Theorem \ref{divis} $\XX$ has  a totally disconnected ergodic extension $\YY$ of order $k$ which is also $k$-divisible.  By Theorem \ref{infty-weyl}, $\YY$ is also Weyl of order $k$, giving the claim.

Now we prove $(ii)$.  If $\Gamma$ is $m$-torsion and $\XX$ is an ergodic $\Gamma$-system of order $k$, then by Theorem \ref{bounded-tor} all the structure groups $U_1,\dots,U_k$ are totally disconnected.  Next, since $\Gamma$ is countably generated, we may express $\Gamma$ as the quotient of $\Z^\omega$.  The $\Gamma$-system $\XX$ can then also be viewed as a $\Z^\omega$-system in the obvious fashion; one can check from the definitions that extending the acting group from $\Gamma$ to $\Z^\omega$ does not affect the ergodicity, the Host--Kra factors or the notion of a polynomial.  In particular, the structure groups remain unchanged, so that $\XX$ is still a totally disconnected $\Z^\omega$-system of order $k$.  By part $(i)$, this system has an extension $\YY$ which is an ergodic totally disconnected $\infty$-divisible Weyl $\Z^\omega$-system of order $k$.  This is a generalized extension of $\XX$, giving the claim.

\section{Constructing translational and double coset systems}\label{sec-doublecoset}

In this section we prove Theorem \ref{structure-thms}.  The following simple example may be worth keeping in mind as motivation: if $\alpha \in \R$ is irrational, then the skew shift $\Z$-system defined as $\XX = \T^2$ with action
$$ T^n (x,y) = \left(x+n\alpha, y + nx + \frac{n(n-1)}{2} \alpha\right)$$
is clearly Weyl of order $2$, and can be identified with a translational system $\mG/\Lambda$, where $\mG$ is the group of continuous measure-preserving transformations $S \colon \T^2 \to \T^2$
of the form
$$ S \colon (x,y) \mapsto (x + \theta, y + m x + \beta )$$
for $m \in\Z$ and $\theta,\beta \in \T$, and $\Lambda$ is the stabilizer of the origin $(0,0)$ in $\mG$, that is to say those transformations with $\theta=\beta=0$.  Indeed, $\mG$ acts transitively on $\T^2$, and the shift $T$ provides a homomorphism from $\Z$
to $\mG$.

To treat the general case, we first need two lemmas.  The first describes how polynomials on translational systems interact with translations. We say that a group $\mG$ acting on a $\Gamma$-system $\XX$ \emph{fixes a factor} $\pi:\XX\rightarrow \YY$ if $\pi\circ g = \pi $ for all $g\in \mG$.

 \begin{lemma}[Polynomials on translational systems]\label{degrees}  
 Let $\Gamma$ be a countable abelian group, let $k\geq 1$, and let $\mG/\Lambda$ be an ergodic translational $\Gamma$-system of degree $k$.  Suppose furthermore that $\mG$ is equipped with a degree $k$ filtration $(\mG_j)_{j\geq 0}$, such that $\mG_j$ fixes the factor $\ZZ^{j-1}(\mG/\Lambda)$ for all $j\geq 1$.
 
 For any compact abelian metrizable $W$, every $d,m\geq 0$, every polynomial $P\in \Poly_{\leq d}(\mG/\Lambda,W)$, and every $g\in \mG_m$ we have then the following properties. 
 \begin{itemize}
     \item[(i)] $P\circ g$ is a polynomial of degree $d$ (i.e., $\mG$ acts on $\Poly_{\leq d}(\mG/\Lambda,W)$).
     \item[(ii)] $\partial_g P = P \circ g - P$ is\footnote{This generalizes Lemma \ref{ppfacts}(iii), which corresponds to the case where $\Gamma$ is centralized by $\mG$ .} a polynomial of degree $d-m$.   (In particular, $\partial_g P$ vanishes if $d<m$.)
\end{itemize}
 \end{lemma}
 
 \begin{proof}
 We prove the claim by induction on $d$ and then by downward induction on $m$. If $d=0$, then by ergodicity $P$ is a constant and the claims follow trivially. Let $d\geq 1$ and assume inductively that the claim holds for all smaller values of $d$. If $m>k$, then $g$ acts trivially and the claims are immediate; thus suppose inductively that $m \leq k$ and that the claims hold for all larger values of $m$. 
 
 Since $(i)$ follows from $(ii)$, it suffices to establish $(ii)$.  If $m \geq d+1$ then by Proposition \ref{ppfacts}(ii) $P$ is measurable with respect to $\ZZ^d(\mG/\Lambda)$ and hence fixed by $g$, so we may assume $m < d+1$.  For $\gamma \in \Gamma$, observe that
 $$\partial_\gamma \partial_g P = \partial_g \partial_\gamma P + \partial_{[g^{-1},\phi(\gamma)^{-1}]}(P\circ \phi(\gamma) \circ g).$$ 
As $g \in \mG_m$ and $\phi(\gamma) \in \mG = \mG_1$, $[g^{-1},\phi(\gamma)^{-1}]$ lies in $\mG_{m+1}$.  Applying the induction hypothesis, we see that both terms on the right-hand side lie in $\Poly_{\leq d-m-1}(\mG/\Lambda,W)$, hence $\partial_g P$ lies in $\Poly_{\leq d-m}(\mG/\Lambda)$ as required.
 \end{proof}

In order to handle part $(ii)$ of the theorem, we will also need a lifting lemma for polynomials in divisible systems.

\begin{lemma}[Polynomials lift in divisible systems]\label{polyrangelift}
Let $\Gamma$ be a countable abelian group, let $k\geq 0$, let $\XX$ be a $k$-divisible ergodic $\Gamma$-system, and let $\varphi \colon U\rightarrow W$ be a surjective homomorphism between two compact abelian metrizable groups $U, W$. 
\begin{itemize}
\item[(i)] The induced map $\varphi_* \colon \Poly_{\leq k}(\XX,U)\rightarrow \Poly_{\leq k}(\XX,W)$ defined by $\varphi_*(Q) \coloneqq \varphi\circ Q$ is surjective.
\item[(ii)]  Furthermore, for any compact subset $K_W$ of $\Poly_{\leq k}(\XX,W)$, there exists a compact subset $K_U$ of $\Poly_{\leq k}(\XX,U)$ and a measurable map $\lambda \colon K_W \to K_U$ such that $\varphi_*( \lambda( Q ) ) = Q$ for all $Q \in K_W$.
\end{itemize}
\end{lemma}

One can think of the map $\lambda$ in Lemma \ref{polyrangelift}(ii) as a ``local measurable section'' for $\varphi$; it is not continuous\footnote{Indeed, the simple example when $\XX$ is a point, $K_W=U=W=\T$, $k=0$ (so in particular $\Poly_{\leq k}(\XX,U) = \Poly_{\leq k}(\XX,W) = \T$), and $\varphi(u) = 2u$ for all $u \in U$ already shows that continuous sections do not always exist.} in general, but at least has the weaker property of mapping the compact domain $K_W$ into a compact range $K_U$, which turns out to be adequate for our applications.

\begin{proof}
By \eqref{fourier-equiv} we have $\Poly_{\leq k}(\XX,U) \equiv \Hom(\hat U, \Poly_{\leq k}(\XX))$ and $\Poly_{\leq k}(\XX,W) \equiv \Hom(\hat W, \Poly_{\leq k}(\XX))$.  As $U$ surjects onto $W$, one can identify $\hat W$ with a subgroup of $\hat U$.  By hypothesis, $\Poly_{\leq k}(\XX)$ is divisible, hence by Lemma \ref{divis-inject} every homomorphism from $\hat W$ to $\Poly_{\leq k}(\XX)$ extends to a homomorphism from $\hat U$ to $\Poly_{\leq k}(\XX)$.  The claim $(i)$ follows.

We now prove $(ii)$ by a refinement of the above argument (and in particular by a closer inspection of the Zorn's lemma argument used to prove Lemma \ref{divis-inject}).  We again use the identifications $\Poly_{\leq k}(\XX,U) \equiv \Hom(\hat U, \Poly_{\leq k}(\XX))$, $\Poly_{\leq k}(\XX,W) \equiv \Hom(\hat W, \Poly_{\leq k}(\XX))$. By enlarging $K_W$ if necessary, we may assume $K_W$ takes the form
$$ K_{W, \vec F_W} \coloneqq \{ Q \in \Hom(\hat W, \Poly_{\leq k}(\XX)): Q_\xi \in F_\xi \forall \xi \in \hat W \}$$
where $Q \colon \xi \mapsto Q_\xi$ denotes a homomorphism from $\hat W$ to $\Poly_{\leq k}(\XX)$, and $F_W = (F_\xi)_{\xi \in \hat W}$ is some tuple of compact subsets $F_\xi$ of $\Poly_{\leq k}(\XX)$.   Define a \emph{partial solution} to be a triple $(V, \vec F_V, \lambda_V)$ where $V$ is a compact metrizable group with $\hat W \leq \hat V \leq \hat U$, $\vec F_V = (F_\xi)_{\xi \in \hat V}$ is a tuple of compact subsets of $\Poly_{\leq k}(\XX)$ extending $F_W$, and $\lambda_V \colon K_{W, \vec F_W} \to K_{V, \vec F_V}$ is a measurable map such that $\lambda_V(Q)$ extends $Q$ for all $Q \in K_{W,\vec F_W}$.  Clearly we have a partial solution with $V=W$, and if we have a partial solution with $V=U$ then we are done (note that $K_{U,\vec F_U}$ is compact by Tychonoff's theorem).  We can partially order these partial solutions by declaring $(V, \vec F_V, \lambda_V) \leq (V', \vec F_{V'}, \lambda_{V'})$ if $\hat V \leq \hat V'$, $\vec F_{V'}$ extends $\vec F_V$, and $\lambda_{V'}(Q)$ extends $\lambda_V(Q)$ for all $Q \in K_{W, \vec F_W}$.  One can check that every chain of partial solutions is of at most countable length (because of the countable nature of $\hat U$) and has an obvious least upper bound.  By Zorn's lemma, it thus suffices to show that any partial solution $(V, \vec F_V, \lambda_V)$ with $\hat V < \hat U$ is not maximal.  Let $\eta$ be an element of $\hat U$ that is not in $\hat V$; and let $V'$ be such that $\hat V'$ is generated by $\hat V$ and $\eta$.  It will suffice to construct a partial solution $(V', \vec F_{V'}, \lambda_{V'})$ that is larger than $(V, \vec F_V, \lambda_V)$.

Suppose first that $n \eta \not \in \hat V$ for any positive integer $n$.  Then one can define the extension $\lambda_{V'}(Q) \colon \hat V' \to \Poly_{\leq k}(\XX)$ of $\lambda_V(Q) \colon \hat V \to \Poly_{\leq k}(\XX)$ for any $Q \in K_{W,\vec F_W}$ to be the unique extension which is a homomorphism with $\lambda_{V'}(Q)_\eta = 0$.  For $\xi \in \hat V' \backslash \hat V$, the $\lambda_{V'}(Q)_\xi$, $Q \in K_{W,\vec F_W}$ are easily seen to range inside some compact subset $F_\xi$ of $\Poly_{\leq k}(\XX)$; setting $\vec F_{V'} \coloneqq (F_\xi)_{\xi \in \hat V'}$ we obtain the desired larger partial solution $(V', \vec F_{V'}, \lambda_{V'})$.

Now suppose that $n \eta \in \hat V$ for some positive integer $n$, which we may take to be minimal.  By Lemma \ref{ppfacts}(i), the compact set $K_{n\eta}$ of $\Poly_{\leq k}(\XX)$ is contained in a finite number of cosets of $\T$.  Applying divisibility, we may then find a compact subset $K_\eta$ of $\Poly_{\leq k}(\XX)$ and a measurable map $\psi \colon K_{n\eta} \to K_\eta$ such that $n \psi(Q) = Q$ for all $Q \in K_{n\eta}$.  We now define extension $\lambda_{V'}(Q) \colon \hat V' \to \Poly_{\leq k}(\XX)$ of $\lambda_V(Q) \colon \hat V \to \Poly_{\leq k}(\XX)$ for any $Q \in K_{W,\vec F_W}$  to be the unique extension which is a homomorphism with $\lambda_{V'}(Q)_\eta = \psi( \lambda_{V'}(Q)_{n\eta} )$.  One can check that such an extension exists.   For  $\xi \in \hat V' \backslash \hat V$, the $\lambda_{V'}(Q)_\xi$, $Q \in K_{W,\vec F_W}$ are easily seen to range inside some compact subset $F_\xi$ of $\Poly_{\leq k}(\XX)$, and we again obtain the desired larger partial solution $(V', \vec F_{V'}, \lambda_{V'})$ as desired.
\end{proof}

Now we can prove Theorem \ref{structure-thms}.
First we remark from \cite[Lemma 2.1]{btz2} that any Weyl system has a canonical structure of a compact space, in such a way that every polynomial has a unique continuous representative.  By abuse of notation we shall identify each polynomial on such systems with that representative, so that all polynomials are now continuous.  In view of Corollary \ref{Weyl extensions}, it will suffice to establish the following claim.

\begin{proposition}[Inductive claim]\label{inductive}  Let $k \geq 1$, let $\Gamma$ be a countable abelian group, and let $\XX$ be an ergodic Weyl $\Gamma$-system of order $k$.  
\begin{itemize}
\item[(i)] $\XX$ is isomorphic (both as a $\Gamma$-system and as a compact space) to a translational $\Gamma$-system $\mG/\Lambda$, where $\mG$ is a Polish group equipped with a degree $k$ filtration $(\mG_j)_{j\geq 0}$ of Polish subgroups, such that $\mG_j$ fixes the factor $\ZZ^{j-1}(\mG/\Lambda)$ for all $j\geq 1$.
\item[(ii)]  Suppose further that $\XX$ is $k$-divisible. Then any factor of $\XX$ is isomorphic (as a $\Gamma$-system) to a double coset system $K \backslash \mG /\Lambda$, with the factor map given by $g \Lambda \mapsto K g \Lambda$ for all $g \in \mG$.  We also make the additional technical claim (convenient for inductive purposes) that there exists a compactly supported Radon probability measure $\mu$ on $K$ such that for every $x \in \mG/\Lambda$, $Kx$ is compact and the pushforward $\mu_x$ of $\mu$ from $K$ to $Kx \subset K \backslash \mG / \Lambda$ is $K$-invariant, or equivalently (by Riesz representation) that
$$ \int_K f(kx)\ d\mu(x) = \int_K f(k_0 kx)\ d\mu(x)$$
for all $x \in \mG/\Lambda$, $k_0 \in K$, and $f \in C(\mG/\Lambda)$.  Furthermore we have $\mu_x = \mu_{k_0 x}$ for all $k_0 \in K$.
\end{itemize}
\end{proposition}

We remark that in general, $K$, while a Polish group, will not be compact or even locally compact, and so does not necessarily come with a theory of Haar measure.  Nevertheless, $K$ will have compact orbits $Kx$, and the pushforward measures $\mu_x$ claimed in $(ii)$ will still exist (and be unique), and will serve as a ``Haar bundle'' for these orbits, cf. Remark \ref{rem-bundle}.  
  The measure $\mu$ is a sort of ``poor man's Haar measure'' for $K$ that will serve to conveniently organize this Haar bundle in a manner that interacts well with the topological structures on $\mG/\Lambda$ and $K \backslash \mG/\Lambda$.

We prove both parts of this proposition by induction on $k$.  We begin with $(i)$. The case $k=1$ is immediate since Weyl systems of order $1$ are rotational systems $Z$, and we can equip $Z$ with the abelian filtration in which $Z_0=Z_1=Z$ and $Z_i=\{0\}$ for $i>1$.  Now suppose that $k \geq 2$ and that Proposition \ref{inductive}(i) has already been proven for $k-1$.  Let $\XX$ be an ergodic Weyl $\Gamma$-system of order $k$.  By definition, we have
$$ \XX = \ZZ^{k-1}(\XX) \rtimes_\rho U$$
for some compact metrizable $U$ and some $U$-cocycle $\rho$ on $\ZZ^{k-1}(\XX)$ which is polynomial of degree $k-1$, with the canonical topology on $\XX$ being the product of the canonical topology of $\ZZ^{k-1}(\XX)$ and the topology of $U$.  By inductive hypothesis, we may write $\ZZ^{k-1}(\XX)$ as a translational system $\mG'/\Lambda'$, where $\mG'$ comes with a degree $k-1$ filtration $(\mG'_j)_{j\geq 0}$ of Polish subgroups, such that $\mG'_j$ fixes the factor $\ZZ^{j-1}(\mG'/\Lambda')$ for all $j\geq 1$.  

By Lemma \ref{degrees}, $\mG'$ acts continuously by homomorphisms on the abelian group $\Poly_{\leq k-1}(\ZZ^{k-1}(\XX),U)$.  We then define $\mG$
$$ \mG \coloneqq \mG' \ltimes \Poly_{\leq k-1}(\ZZ^{k-1}(\XX),U)$$
to be the semi-direct product of these two groups, that is to say the group of pairs $(s,q)$ with $s \in \mG'$ and $q \in \Poly_{\leq k-1}(\ZZ^{k-1}(\XX),U)$ with group law
$$ (s_1,q_1) (s_2,q_2) \coloneqq (s_1 s_2, q_1 \circ s_2 + q_2).$$
Since $\mG'$ and $\Poly_{\leq k-1}(\ZZ^{k-1}(\XX),U)$ are\footnote{Indeed, $\Poly_{\leq k-1}(\ZZ^{k-1}(\XX),U)$ is a closed subgroup of the Polish group $\M( \ZZ^{k-1}(\XX), U) = \Hom(\hat U, \M(\ZZ^{k-1}(\XX),\T)) \leq \M(\ZZ^{k-1}(\XX),\T)^{\hat U}$ and is thus also Polish.} Polish groups (with the former acting continuously on the latter), $\mG$ is also.  This group $\mG$ acts continuously on $\XX$ by the formula
\begin{equation}\label{sqy}
(s,q) (y,u) \coloneqq (sy, u + q(y))
\end{equation}
for $(s,q) \in \mG$, $y \in \ZZ^{k-1}(\XX) = \mG'/\Lambda'$, and $u \in U$.  By Fubini's theorem, the action is measure-preserving; by restricting attention to those polynomials $q$ that are constant (i.e., elements of $U$) we see that the action is transitive.  Using $0$ to denote the origin $\Lambda'/\Lambda'$ in $\ZZ^{k-1}(\XX) = \mG'/\Lambda'$, so that $\Lambda'$ is the $\mG'$-stabilizer of $0$, we see that the $\mG$-stabilizer $\Lambda$ of $0$ is given by
$$ \Lambda \coloneqq \{ (\gamma,q) \in \mG: \gamma \in \Lambda', q(0)=0 \}.$$
This is clearly a closed subgroup of $\mG$.  By Effros' theorem\footnote{Effros' theorem \cite{Effros} states that if $\mG$ is a Polish group acting transitively on a compact metric space $X$, then for any $x\in X$ the stabilizer $\Lambda = \{g\in \mG : gx=x\}$ is a closed subgroup of $\mG$ and $X$ is homeomorphic to $\mG/\Lambda$.}, $\mG/\Lambda$ is homeomorphic to $\XX$; in particular $\mG/\Lambda$ is a compact Polish space, and the probability measure on $\XX$ induces a probability measure on $\mG/\Lambda$ which is $\mG$-invariant since the action of $\mG$ is measure-preserving on $\XX$.  The action of $\Gamma$ on $\XX = \mG/\Lambda$ is a translational using the homomorphism
\begin{equation}\label{phig}
\phi \colon \gamma \mapsto (\phi'(\gamma), \rho_\gamma)
\end{equation}
from $\Gamma$ to $\mG$, here $\phi' \colon \Gamma \to \mG'$ is the corresponding homomorphism for the translational system $\mG'/\Lambda'$.

To close the induction for $(i)$, we need to give $\mG$ a filtration obeying the required properties.  This is given by the formulae
$$ \mG_i \coloneqq \{ (s,q) \in \mG: s \in \mG'_i, q \in \Poly_{\leq k-i}(\ZZ^{k-1}(\XX),U) \}$$
for $i \geq 1$, with $\mG_0 \coloneqq \mG$.  This is readily seen (using Lemma \ref{degrees}) to be a degree $k$ filtration of Polish groups.  Since $\mG'_j$ fixes $\ZZ^{j-1}(\mG'/\Lambda') = \ZZ^{\min(j-1,k-1)}(\XX)$ for all $j \geq 1$, it is not difficult to see that $\mG_j$ fixes $\ZZ^{j-1}(\XX)$ for all $j \geq 1$.  

Now we close the induction for $(ii)$.  The case $k=1$ is again straightforward: by Pontryagin duality we see that quotients of a rotational system $Z$ arise by quotienting out by a compact subgroup $K$ of $Z$, which is equipped with Haar probability measure $\mu$, and it is easy to verify the required claims.  Now suppose that $k \geq 2$ and that $(ii)$ has already been established for $k-1$.

Let $\YY$ be a factor of $\XX$.  Write $\XX' \coloneqq \ZZ^{k-1}(\XX)$ and $\YY' \coloneqq \ZZ^{k-1}(\YY)$.  By part $(i)$ we can write 
$$\XX = \XX' \rtimes_\rho U = \mG/\Lambda$$
and
$$ \XX' = \mG'/\Lambda'$$
where $\mG, \Lambda$ were constructed as above.   By Proposition \ref{abelext} and Proposition \ref{prop-invlim}, we may write
$$ \YY = \YY' \rtimes_\sigma W$$
for some compact abelian metrizable $W$ and $W$-valued cocycle $\sigma$ on $\ZZ^{k-1}(\YY)$, and there exists a surjective homomorphism $\varphi \colon U \to W$ and a function $F \in \M(\XX', W)$ such that one has the identity
\begin{equation}\label{compat}
\varphi \circ \rho + dF = \sigma \circ \pi'
\end{equation}
where $\pi' \colon \XX' \to \YY'$ is the factor map; see Figures \ref{fig:extensions}, \ref{fig:schem}.  Furthermore, the factor map $\pi \colon \XX \to \YY$ is given by the formula
\begin{equation}\label{pi-def}
 \pi( x', u) = (\pi'(x'), \varphi(u) + F(x'))
 \end{equation}
for $x' \in \ZZ^{k-1}(\XX)$ and $u \in U$; see Figure \ref{fig:expanded}.

\begin{figure}
\centering
    \[\begin{tikzcd}
	\XX = \XX' \rtimes_\rho U = \mG/\Lambda && \XX' = \mG'/\Lambda' \\
 \\
    \YY = \YY' \rtimes_\sigma W = K \backslash \mG / \Lambda && \YY' = K' \backslash \mG'/\Lambda'
    \arrow[from=1-1, to=1-3]	
    \arrow["\pi", from=1-1, to=3-1]
	\arrow[from=3-1, to=3-3]
    \arrow["\pi'", from=1-3, to=3-3]
 \end{tikzcd}\]
    \caption{The four main $\Gamma$-systems appearing in the proof of  Proposition \ref{inductive}(ii).  Establishing the identification $\YY = K \backslash \mG/\Lambda$ is the main task needed to close the induction.}
  \label{fig:extensions}
\end{figure} 

\begin{figure}
\centering
    \[\begin{tikzcd}
	\XX' && U \\
 \\
    \YY' && W 
    \arrow["\rho_\gamma", from=1-1, to=1-3]	
    \arrow["\pi'", from=1-1, to=3-1]
	\arrow["\sigma_\gamma", from=3-1, to=3-3]
    \arrow["\varphi", from=1-3, to=3-3]
   \arrow["\partial_\gamma F", from=1-1, to=3-3]    
 \end{tikzcd}\]
    \caption{A diagrammatic depiction of the relation \eqref{compat}.  The outer square commutes up to the coboundary correction term $\partial_\gamma F$ on the diagonal.}
  \label{fig:schem}
\end{figure} 

\begin{figure}
\centering
    \[\begin{tikzcd}
	\XX = \XX' \rtimes_\rho U && X' \rtimes_{\varphi \circ \rho} W && X' \rtimes_{\sigma \circ \pi'} W && \XX' \\
 \\
    && && \YY = \YY' \rtimes_\sigma W && \YY'
    \arrow["{(x',u) \mapsto (x',\varphi(u))}",from=1-1, to=1-3]	
    \arrow["{(x',w) \mapsto (x',w+F(x'))}",from=1-3, to=1-5,shift left=1]	
    \arrow["{(x',w) \mapsto (x',w-F(x'))}",from=1-5, to=1-3,shift left=1]	
    \arrow["{(x',w) \mapsto x'}",from=1-5, to=1-7]	
    \arrow["\pi", from=1-1, to=3-5]
    \arrow["\pi'", from=1-7, to=3-7]
    \arrow["{(x',w) \mapsto (\pi'(x'),w)}", from=1-5, to=3-5]
	\arrow["{(y',w) \mapsto y'}",from=3-5, to=3-7]
 \end{tikzcd}\]
    \caption{The relations \eqref{compat}, \eqref{pi-def} expressed as a commuting diagram of $\Gamma$-systems.}
  \label{fig:expanded}
\end{figure} 
Roughly speaking, our goal is to identify a subgroup $K$ of $\mathcal{G}$ that fixes the factor $\YY$. First, observe that by the induction hypothesis, we can write
$$ \YY' = K' \backslash \mG' / \Lambda'$$
for some closed subgroup $K'$ of $\mathcal{G'}$ normalized by $\phi'(\Gamma)$, with $\pi'(g' \Lambda') = K' g' \Lambda'$ for $g' \in \mG'$, with $K' \backslash \mG'/\Lambda'$ a Polish space. We also have a compactly supported Radon probability measure $\mu'$ on $K'$ such that the pushforwards $\mu'_{x'}$ to the compact orbits $K' x'$ are $K'$-invariant for every $x' \in \mG'/\Lambda'$, and $\mu'_{x'} = \mu'_{k'x'}$ for all $k' \in K'$.

We let $K \subset \mG$ denote the set
\begin{equation}\label{K-def}
 K \coloneqq \{ (s,q) \in \mG: s \in K'; \varphi(q) + \partial_s F = 0 \}.
 \end{equation}
To motivate this definition, consider the intermediate factor $\ker \varphi \backslash \XX = \XX' \rtimes_{\varphi \circ \rho} W$, which lies between $\XX$ and $\YY$ (see Figure \ref{fig:expanded}). Any $(s,q) \in \mathcal{G}$ then gives rise to an element $(s, \varphi(q))$ acting on $\ker \varphi \backslash \XX$. We are specifically interested in those transformations that induce a trivial action on the factor $\YY$. Hypothetically, if $\varphi \circ \rho = \sigma \circ \pi'$, this would mean that we could take $s \in K'$ and $\varphi(q) = 0$. However, in general, we have the weaker equation \eqref{compat}. Thus, we must consider the isomorphism $\XX' \rtimes_{\varphi \circ \rho} W \cong \XX' \rtimes_{\sigma \circ \pi'} W$ given by $(x,w) \mapsto (x, w + F(x))$, which induces the condition $\varphi(q) + \partial_s F = 0$ seen in \eqref{K-def}.

Clearly $K$ is closed; from the cocycle equation we see that $K$ is a subgroup of $\mG$.  If $(s,q) \in K$ and $\gamma \in \Gamma$, then from \eqref{phig} we compute
$$ \phi(\gamma) (s,q) \phi(\gamma)^{-1} = ( s_\gamma, (q + \partial_{s_\gamma} \rho_\gamma) \circ \phi'(\gamma)^{-1})$$
where $s_\gamma \coloneqq \phi'(\gamma) s \phi'(\gamma)^{-1}$.  As $K'$ is normalized by $\phi'(\Gamma)$, $s_\gamma$ lies in $K'$, and
\begin{align*}
\varphi((q + \partial_{s_\gamma} \rho_\gamma) \circ \phi'(\gamma)^{-1}) + \partial_{s_\gamma} F 
&= [(\varphi(q) + \partial_s F) + \partial_s(\varphi \circ \rho_\gamma + \partial_\gamma F) ] \circ \phi'(\gamma)^{-1} \\
&= [0 + \partial_s (\sigma_\gamma \circ \pi')] \circ \phi'(\gamma)^{-1} \\
&= 0
\end{align*}
thanks to \eqref{K-def}, \eqref{compat}.  Thus $K$ is normalized by $\phi(\gamma)$.

We now make the key claim that the homomorphism $(s,q) \mapsto s$ from $K$ to $K'$ is surjective; that is to say, for every $s \in K$ there exists $q \in \Poly_{\leq k-1}(\XX',U)$ such that $\varphi(q) + \partial_s F = 0$.  By Lemma \ref{polyrangelift}(i) it suffices to show that
$\partial_s F$ is a polynomial of degree at most $k-1$.  We claim more generally that for any $d \geq 1$ and $s \in K' \cap \mG'_d$, that $\partial_s F$ is a polynomial of degree at most $k-d$.  This is trivial for $d \geq k$ since $\mG'_d$ is trivial in this case.  Now suppose inductively that $1 \leq d < k$ and the claim has already been proven for $d+1$.  If $s \in K' \cap \mG'_d$ and $\gamma \in \Gamma$, we observe that
\begin{equation}\label{sgamma} \partial_\gamma \partial_s F = \partial_s \partial_\gamma F + \partial_{[s,\phi'(\gamma)]} F \circ s \circ \phi'(\gamma).
\end{equation}
Note that $[s,\phi'(\gamma)]$ lies in $K' \cap \mG'_{d+1}$ and hence the term $\partial_{[s,\phi'(\gamma)]} F \circ s \circ \phi'(\gamma)$ is of degree at most $k-d-1$ by induction hypothesis (and Lemma \ref{degrees}).  Meanwhile, by applying $\partial_s$ to \eqref{compat} (noting that $\pi'$ is $s$-invariant) we see that
$$ \partial_s \partial_\gamma F = - \varphi(\partial_s \rho)$$
and hence $\partial_s \partial_\gamma F$ is also of degree at most $k-d-1$, thanks to Lemma \ref{degrees} and the fact that $\rho$ is of degree $k-1$.  Thus $\partial_\gamma \partial_s F$ is of degree at most $k-d-1$ for all $\gamma$, and so $\partial_s F$ is of degree $k-d$, closing the induction.  This concludes the proof that $K$ surjects onto $K'$.

From Lemma \ref{polyrangelift}(ii) (and the continuity of the map $s \mapsto \partial_s F$ from $K'$ to $\Poly_{\leq k-1}(\XX',W)$) we conclude the following stronger property: for every compact subset $K'_0$ of $K'$, there exists a compact subset $K_0$ of $K$ and a measurable map $\Psi \colon s \mapsto (s,q_s)$ from $K'_0$ to $K_0$.  These measurable local sections of the projection $(s,q) \mapsto s$ will be useful later in the argument.

The kernel of the homomorphism $(s,q) \mapsto s$ from $K$ to $K'$ is the set $\Poly_{\leq k-1}(\XX', \mathrm{ker}(\varphi))$ of polynomials of degree at most $k-1$ from $\XX'$ to the kernel $\mathrm{ker}(\varphi)$ of $\varphi \colon U \to W$.  Thus we have a short exact sequence
$$ 0 \to \Poly_{\leq k-1}(\XX', \mathrm{ker}(\varphi)) \to K \to K' \to 0.$$ 
Unfortunately, the group $\Poly_{\leq k-1}(\XX', \mathrm{ker}(\varphi))$ will not be compact in general, but it contains $\mathrm{ker}(\varphi)$ as a compact subgroup. Since these groups correspond to the same vertical translations of $U$, we will be able to use $\mathrm{ker}(\varphi)$ as a suitable proxy for $\Poly_{\leq k-1}(\XX', \mathrm{ker}(\varphi))$ at several steps in the arguments below. 

If $(s,q) \in K$, then for almost every $(x',u) \in \XX$ we have
\begin{align*}
\pi( (s,q)(x',u) ) &= \pi( sx', u + q(x') )\\
&= (\pi'(sy), \varphi(u + q(x')) + F(sx') )\\
&= (\pi'(x'), \varphi(u) + F(x') + \varphi(q(x')) + \partial_s F(x') ) \\
&= (\pi'(x'), \varphi(u)) \\
&= \pi(x',u)
\end{align*}
thanks to \eqref{pi-def}, \eqref{K-def}.  Thus $\pi$ is $K$-invariant, and so every element of $L^\infty(\YY)$ is a $K$-invariant element of $L^\infty(\XX)$.  We now claim conversely (basically by chasing the diagram in Figure \ref{fig:expanded}) that every $K$-invariant element of $L^\infty(\XX)$ is an element of $L^\infty(\YY)$.  Indeed, if $f(x',u)$ is a $K$-invariant function of $L^\infty(\XX)$, then by noting that $K$ contains $\{ (0,u): u \in \mathrm{ker}(\varphi)\}$ we see that $f(x',u)$ is invariant with respect to shifts of $u$ by $\mathrm{ker}(\varphi)$ and so we may write $f(x',u) = g(x',\varphi(u)+F(x'))$ almost everywhere for some measurable function $g$ on $\ZZ^{k-1}(\XX) \times W$.  If $(s,q) \in K$, then by repeating the previous calculations we see that
$$ f((s,q)(x',u)) = f(x',u)$$
implies that
$$ g(sx', \varphi(u)+F(x')) = g(x', \varphi(u)+F(x'))$$
for almost all $x' \in \XX'$ and $u \in U$, thus (since $K$ surjects onto $K'$ and $U$ surjects onto $W$) $g(x',w)$ is invariant with respect to the action of $K'$ in the first variable $y$.  Thus we can write $g(x',w) = h(\pi'(x'),w)$ for some measurable function $x'$ on $\YY = \YY' \rtimes_\sigma W$.  From \eqref{pi-def} we have $f = h \circ \pi$ and so $f$ is $\YY$-measurable as required.

In view of the above description of $L^\infty(\YY)$, it is now plausible that $\YY$ can be identified with the topological quotient space $K \backslash \mG / \Lambda$, but we need to take some care with this identification as $K$ is not compact.  The first step is to build the measure $\mu$.  By induction hypothesis, $\mu$ is supported in some compact subset $K'_0$ of $K'$.  Using Lemma \ref{polyrangelift}(ii) as discussed previously, there exists a compact subset $K_0$ of $K$ and a measurable local section map $\Psi \colon s \mapsto (s,q_s)$ from $K'_0$ to $K_0$.  The group $\ker \varphi$ is a compact central subgroup of $K$, and hence $K_1 \coloneqq (\ker \varphi) K_0$ is a compact subset of $K$.  We define $\mu$ via the Riesz representation theorem as the unique Radon measure on $K_1$ for which
$$ \int_{K_1} f(s,q)\ d\mu(s,q) = \int_{K'_0} \int_{\ker \varphi} f(s,q_s+u)\ du d\mu'(s)$$
where $du$ denotes the probability Haar measure on $\ker \varphi$.  This is easily seen to be a Radon probability measure on $K_1$.  To show that $\mu_x$ is $K$-invariant for every $x \in \XX$, it suffices to show that
$$ \int_{K_1} f( (s,q) x )\ d\mu(s,q) = \int_{K_1} f( (s_0,q_0) (s,q) x )\ d\mu(s,q)$$
for all bounded measurable $f \colon \XX \to \R$ and $(s_0,q_0) \in K$.  If $f$ has mean zero with respect to the central compact $\ker \varphi$ action, then both integrals vanish thanks to the invariance of Haar measure $du$, so we may assume instead that $f$ is $\ker \varphi$-invariant.  We may thus write $f(x',u) = \tilde f(x', \varphi(u) + F(x'))$ for some bounded measurable $\tilde f \colon \XX' \rtimes_{\sigma \circ \pi'} W \to \R$ (cf., Figure \ref{fig:expanded}).  From \eqref{K-def} and the section property we see that
$$ \int_{K_1} f( (s,q) x )\ d\mu(s,q) = \int_{K'_0} \tilde f( sx', \varphi(u) + F(x'))\ d\mu'(s)$$
for $x = (x',u) \in \XX$, and similarly
$$ \int_{K_1} f( (s_0,q_0) (s,q) x )\ d\mu(s,q) = \int_{K'_0} \tilde f( s_0 sx', \varphi(u) + F(x'))\ d\mu'(s).$$
The claim now follows from the induction hypothesis.  Thus $\mu_x$ is $K$-invariant.  A similar argument gives
$$ \int_{K_1} f( (s,q) x )\ d\mu(s,q) = \int_{K_1} f( (s,q) (s_0,q_0) x )\ d\mu(s,q)$$
for all bounded measurable $f \colon \XX \to \R$ and $(s_0,q_0) \in K$, so that $\mu_x = \mu_{(s_0,q_0) x}$.  Clearly the support of $\mu_x$ is contained in the compact set $K_1 x$ and is $K$-invariant; since $K_1$ is a subset of the group $K$, we conclude that $K_1 x = Kx$.  In particular, $Kx$ is compact.  This verifies all the properties required on $\mu$.

Next we verify that the quotient space $K \backslash \mG / \Lambda$ is Polish.  From \cite[Proposition 2.2.6]{kechrisbecker} it suffices to show that $KA = \{k.a : k\in K, a\in A\}$ is a compact subset of $\mG/\Lambda$, whenever $A$ is closed in $\mG/\Lambda$.  But since $K_1 x = Kx$ for all $x \in \mG/\Lambda$, we have $KA = K_1 A$; since $K_1, A$ are both compact and the group action is continuous, $KA$ is compact as required.

Finally, we show that $\YY$ can be identified with $K \backslash \mG / \Lambda$.  For any bounded measurable function $f \colon \mG/\Lambda \to \R$, define the average $\E_K f \colon \mG/\Lambda \to \R$ by the formula
$$ \E_K f(x) \coloneqq \int_{\mG/\Lambda} f\ d\mu_x = \int_{K_1} f((s,q) x)\ d\mu(s,q).$$
From Fubini's theorem and the triangle inequality, $\E_K f$ is a contraction on $L^\infty(\mG/\Lambda)$ and also on $L^2(\mG/\Lambda)$; from the compactness of $K_1$ we also see that this operator preserves continuity.  Since the $\mu_x$ are supported on $Kx$, $\E_K f = f$ whenever $f$ is $K$-invariant; since $\mu_x = \mu_{kx}$ for all $k \in K$, we see that $\E_K f$ is $K$-invariant for any bounded measurable $f$. In fact, by the Birkhoff--Alaoglu ergodic theorem \cite{ab40abstract}, $\E_K f$ must equal the orthogonal projection from $L^2(\mG/\Lambda)$ to the subspace of $K$-invariant functions, since that projection is the unique $K$-invariant function in the closed convex hull of the $K$-orbit of $f$.

Using this average we can now perform the identification.  Certainly any element of $L^\infty(K \backslash \mG/\Lambda)$ induces a $K$-invariant element of $L^\infty(\mG/\Lambda)$, which is then identifiable with an element of $L^\infty(\YY)$ by the previous arguments.  Conversely, any element of $L^\infty(\YY)$ can be viewed as a $K$-invariant element $f$ of $L^\infty(\mG/\Lambda)$.  This function need not be continuous, but by Lusin's theorem it is the $L^2$ limit of a sequence $f_n$ of uniformly bounded continuous functions.  Applying $\E_K$ to the $f_n$, we may assume without loss of generality that the $f_n$ are also $K$-invariant in addition to being continuous on $\mG/\Lambda$, and are thus identifiable with continuous functions on $K \backslash \mG / \Lambda$.  In particular, they lie in $L^\infty(K \backslash \mG/\Lambda)$, and hence $f$ lies in this space also.  This provides an identification of $L^\infty(\YY)$ and $L^\infty(K \backslash \mG/\Lambda)$ which is easily seen to be an isomorphism of abelian von Neumann algebras that preserves the $\Gamma$ action, hence we have identified $\YY$ and $K \backslash \mG/\Lambda$ as $\Gamma$-systems.

 This concludes the proof of Proposition \ref{inductive}(ii), and hence also Theorem \ref{structure-thms}.

 \begin{remark}\label{rem-bundle}
 One can ask whether the ``poor man's Haar measure'' $\mu$ is necessary in the above arguments.
The closed group $K$ acts on $X=\mG/\Lambda$ by homeomorphisms, which makes the tuple $(X,K)$ a (right) topological transformation space with orbit space $K\backslash X$. 
Thus we can equip $K\times X$ with a topological groupoid structure with unit space $X$, see \cite[Chapter 1, \S 1]{renault-book} for notation, and in particular the first example there. 
Also, since the quotient map from $\mG/\Lambda$ to $K \backslash \mG/\Lambda$ is open, the bundle map from $K\times X$ to the unit space $X$ is open.  
Now using a result of Steinlage \cite[Lemma 4.3]{Stein} (and an induction argument), we can find a $K$-invariant Radon probability measure $\mu_x$ on the compact homogeneous spaces $Kx$ for each $x$.
If we knew that the $\mu_x$ were unique and that $K\cdot A$ is compact in $X$ for every compact set $A$ in $X$ (cf. \cite[Proposition 2.2.6]{kechrisbecker}), then the Haar bundle $\mu_{x}$ would be continuous in $x$ (in the vague topology) by modifying the argument in \cite[Lemma 1.3]{renault} using the openness of the bundle map. In this case, the conditional expectation operator $\E_K$ would map continuous functions to continuous functions, and thus the topological and measurable identifications $\YY=K\backslash \mG/ \Lambda$ coincide similarly as concluded in our proof of Proposition \ref{inductive}(ii). By the results in \cite{Stein}, uniqueness amounts to showing that $K$ acts on each fibre $Kx$ equicontinuously which does not seem to be satisfied even for Kronecker systems. On the other hand, the classical disintegration theorem gives a representation of the conditional expectation operator $\E_K\colon L^2(\mG/\Lambda)\to L^2(\mG/\Lambda)^K$ (since $L^2(\mG/\Lambda)$ is separable) which by construction and the Banach--Alaoglu theorem is also $K$-invariant, however it is only unique almost surely, and it is not clear to us how to get uniqueness everywhere. 

From our proof, all these properties are satisfied \emph{a posteriori} by employing  the convenient device of the ``poor man's Haar measure'' $\mu$ on $K$. However, the existence of $\mu$ is guaranteed by our very particular situation (indeed, the Zorn argument in the proof of Lemma \ref{polyrangelift}(ii) requires divisibility), and it remains an open question if a more general approach exists which does not use the particular properties of the extension $\XX$. 
\end{remark}

For applications to inverse theorems for the Gowers uniformity norms, we record some additional structural facts about the translational system $\mG/\Lambda$ and the group $K$ constructed by Proposition \ref{inductive}. 

\begin{proposition}\label{more-inductive}  Let the hypotheses be as in Proposition \ref{inductive}, and let $U_1,\dots,U_k$ be the structure groups of $\XX$.
\begin{itemize}
\item[(i)]  $\XX = \mG/\Lambda$ can be identified with the filtered abelian group 
$${\mathcal U} = {\mathcal D}^1(U_1) \times \dots \times {\mathcal D}^k(U_k)$$
in such a way that the compact $k$-step nilspace structure of ${\mathcal U}$ is compatible with the $\Gamma$-system structure of $\XX$ in the sense of Appendix \ref{nilspace-app}\footnote{See the discussion under Theorem \ref{haar-thm}.}.
Furthermore ${\mathcal U}$ can be identified with a filtered subgroup of $\mG$.
\item[(ii)]  The double quotient $K \backslash \mG/\Lambda$ has the structure of a $k$-step nilspace which is compatible with the $\Gamma$-system structure of $\YY =  K \backslash \mG/\Lambda$.
\item[(iii)]  If $\XX$ is $k$-divisible, and $\YY = K \backslash \mG/\Lambda$ is a factor whose structure groups $W_i$ are all $m$-torsion for some $m \geq 1$, then we have
$$ m^{\binom{k+1}{2}} \mG \leq K.$$
\end{itemize}
\end{proposition}
\begin{remark}
The exponent $\binom{k+1}{2}$ here can be improved significantly, but we do not attempt to optimize it here.
\end{remark}
\begin{proof}  We prove $(i)$ by induction on $k$.  The case $k=1$ is straightforward, so suppose $k \geq 2$ and the claim has already been proven for $k-1$.  Writing $\XX = \ZZ^{k-1}(\XX) \rtimes_\rho U_k$ with $\rho$ polynomial of degree $k-1$, we have from induction that $\ZZ^{k-1}(\XX)$ is a translational system $\mG'/\Lambda'$, and from the proof of Proposition \ref{inductive} we have that $\XX = \mG/\Lambda$ with
$$\mG \coloneqq \mG' \ltimes \Poly_{\leq k-1}(\ZZ^{k-1}(\XX),U_k).$$
Note that $U_k$ can be viewed as a subgroup of $\Poly_{\leq k-1}(\ZZ^{k-1}(\XX),U_k)$ that is invariant under the action of $\mG'$.
As $\mG'$ contains a copy of ${\mathcal D}^1(U_1) \times \dots \times {\mathcal D}^{k-1}(U_{k-1})$ as a filtered abelian subgroup, we conclude that $\mG$ contains ${\mathcal U}$ as a filtered abelian subgroup.  By induction we also see that ${\mathcal U}$ acts on $\XX$ transitively and freely.
 Therefore, $\mG=\mathcal{U}\cdot \Lambda$ as groups, and by induction one can show that the identification $\mG/\Lambda = \mathcal{U}$ is a homeomorphism. 
 By \cite[Lemma 6.1]{jt21-1}, we have that $\mathrm{HK}^k(\mG)=\mathrm{HK}^k(\mathcal{U})\cdot \mathrm{HK}^k(\Lambda)$ for all $k$. Since $\mathcal{U}\cap \Lambda$ is trivial, this gives the nilspace isomorphism $\mG/\Lambda \cong \mathcal{U}$.   

 It remains to show that the compact nilspace structure on $\mathcal{U}$ is compatible with the $\Gamma$-system structure of $\XX$, that is say that the Host--Kra measure $\mu^{[n]}$ on $\XX$ agree with Haar probability measure on $\HK^n(\mathcal{U})$ under the above identifications for each $n \geq 0$.  From \cite[Lemma A.23(iv)]{btz} we see that each Host--Kra measure $\mu^{[n]}$ is preserved by the action of $\HK^n(\mG)$, and hence by $\HK^n(\mathcal{U})$ (which acts by translation on $\mathcal{U}^{\{0,1\}^n}$), so it suffices by uniqueness of Haar measure to show that
 $\mu^{[n]}$ is supported on  $\HK^n(\mathcal{U})$.  A routine induction using the explicit description of the $\Gamma$-action of a Weyl system shows that the space $\HK^n(\mathcal{U})$ is preserved by the action of $T^\gamma$ on any $n-1$-face of $\{0,1\}^n$ for any $\gamma \in \Gamma$.  Since this space also contains diagonal elements $(x)_{\omega \in \{0,1\}^n}$ for any $x \in {\mathcal U}$, we conclude that
 $$ (T^{\sum_{i=1}^n \omega_i \gamma_i} x)_{\omega \in \{0,1\}^n} \in \HK^n(\mathcal{U})$$
 for any $x \in {\mathcal U}$ and $\gamma_1,\dots,\gamma_n \in \Gamma$.  In particular, if $f_\omega \colon \mathcal{U} \to \R$ are continuous functions supported on open sets $V_\omega$ with $\prod_{\omega \in \{0,1\}^n} V_\omega$ avoiding $\HK^n(\mathcal{U})$, we have
$$ \int_{\mathcal U} \prod_{\omega \in \{0,1\}^n} f_\omega \circ T^{\sum_{i=1}^n \omega_i \gamma_i} = 0$$
for all $\gamma_1,\dots,\gamma_n \in \Gamma$.  After repeatedly averaging over F{\o}lner sequences and passing to the ergodic limit using 
\cite[(A.6)]{btz}, we conclude that
$$ \int_{{\mathcal U}^{\{0,1\}^n}} \bigotimes_{\omega \in \{0,1\}^n} f_\omega \ d\mu^{[n]} = 0$$
which implies that $\mu^{[n]}$ is supported on $\HK^n(\mathcal{U})$ as claimed.
 This proves $(i)$.

 Now we prove $(ii)$.  To show that $K \backslash G / \Gamma$ has the structure of a $k$-step compact nilspace, it suffices by Lemma \ref{double-nilspace}
 (and part $(i)$) to verify the groupable axiom \eqref{kgl}.  For $i \leq 1$ the claim is trivial since $G_i=G$, and when $i > k$ the claim is also trivial since $G_i$ is trivial.  Now we verify the $i=k$ case.  It suffices to show that if $(y,u) \in G/\Lambda$ (with $y \in G'/\Lambda'$ and $u \in U_k$) and $(s,q) \in K$ is such
that $(s,q)(y,u) = (y,v)$ for some $v \in U_k$, then there exists $(0,q') \in K \cap G_k$ such that $(0,q')(y,u)=(y,v)$.  But from \eqref{sqy} one has $sy=y$ and $u+q(y) = v$, which by \eqref{K-def} implies that $\phi(q(y)) = \phi(q(y)) + \partial_s F(y) = 0$.  In particular, $(0, q(y)) \in K \cap G_k$ and $(0, q(y)) (y,u) = (y,v)$, giving the claim.

Now suppose by downward induction that $1 < i < k$ and the claim has already been established for $i+1$.  It suffices to show that if $x \in G/\Lambda$ and $t \in K$ is such that $tx \in G_i x$, then there exists $t' \in K \cap G_i$ such that $tx = t'x$.  By quotienting $G,K,\Gamma$ by the normal group $G_{i+1}$ and then applying the previous argument (with $k$ replaced by $i$), we can find $t_1 \in (K \cap G_i) G_{i+1}$ such that $G_{i+1} tx = G_{i+1} t_1 x$, thus $t_1^{-1} t x \in G_{i+1} x$.  By induction hypothesis we can find $t_2 \in K \cap G_{i+1}$ such that $t_1^{-1} t x = t_2 x$, thus $tx = t_1 t_2 x$.  Since $t_1 t_2 \in K \cap G_i$, this closes the induction. 

 To complete the proof of $(ii)$ we need to show that the $k$-step compact nilspace structure is compatible with the $\Gamma$-system structure of $\YY$, that is to say that the Host--Kra measure $\mu^{[n]}_\YY$ on $\YY$ agree with the Haar measure on $C^n(K \backslash G / \Lambda)$ for every $n \geq 0$.  By \cite[Lemma A.22]{btz} (or \cite[Lemma 4.5]{host2005nonconventional}), $\mu^{[n]}_\YY$ is the pushforward of the Host--Kra measure $\mu^{[n]}_\XX$ on $\XX$, which by the proof of $(i)$ is supported in $C^n(G/\Lambda)$.  This implies that $\mu^{[n]}_\YY$ is supported in $C^n(K \backslash G / \Lambda)$.  This measure is also invariant with respect to the action of the central group $G_k$ on each of the vertices, and hence also with respect to the $k^{\mathrm{th}}$ structure group of $K \backslash G / \Lambda$ (which is a quotient of $G_k$).  Applying an induction hypothesis, the pushforward of $\mu^{[n]}_\YY$ to $C^n(K' \backslash G' / \Lambda')$ is Haar measure, and hence $\mu^{[n]}_\YY$ is itself Haar measure, as required.

 Now we prove $(iii)$ by induction on $k$.  The case $k=1$ is again straightforward since $\mG=U_1$ is $m$-divisible by Theorem \ref{bounded-tor}. Now suppose that $k \geq 2$ and the claim has already been proven for $k-1$.  
 
  From the proof of Proposition \ref{inductive}, we can write $\YY = \ZZ^{k-1}(\YY) \rtimes_\sigma W_k$ where $\ZZ^{k-1}(\YY) = K' \backslash \mG' / \Lambda'$ is a double coset system, and one has a surjective homomorphism $\varphi \colon U_k \to W_k$ and a function $F \in \M(\ZZ^{k-1}(\XX), W_k)$ such that one has the identity \eqref{compat}, with the factor map $\pi \colon \XX \to \YY$ given by \eqref{pi-def}, and $K$ is given by \eqref{K-def}. 

Let $s\in \mG'$. By induction hypothesis we see that
$$  s^{m^{\binom{k}{2}}} \in K'.$$ As in Proposition \ref{inductive}, the map $\partial_{s^{m^{\binom{k}{2}}}} F$ is a polynomial of degree $k-1$.
The (commutative) ring $\left\langle \partial_{s^{m^{\binom{k}{2}}}} \right\rangle$ of operators on $\M( \ZZ^{k-1}(\YY), W_k )$ generated by $\partial_{s^{m^{\binom{k}{2}}}}$ is $m$-torsion, since $W_k$ is $m$-torsion.  Hence by Lemma \ref{alg-lemma} (and the crude inequality $p^r \geq 2^r \geq r+1$) we have
$$ \partial_{s^{m^{\binom{k}{2}+r}}} \in \left \langle \partial_{s^{m^{\binom{k}{2}}}} \right\rangle \partial_{s^{m^{\binom{k}{2}}}}^{r+1}$$
for any $r \geq 1$ (indeed we could obtain a much better exponent than $r+1$ here if desired).  In particular, by Proposition \ref{degrees} 
$\partial_{s^{m^{\binom{k}{2}+r}}}F$ is a polynomial of degree $k-r-1$. Setting $r=k$ and using the identity $$\binom{k+1}{2} = \binom{k}{2}+k,$$ we deduce that $\partial_{s^{m^{\binom{k+1}{2}}}}F=0$.

Now let $(s,q)\in \mG$. Then $(s,q)^{m^{\binom{k+1}{2}}} = (s^{m^{\binom{k+1}{2}}}, \sum_{j=0}^{m^{\binom{k+1}{2}}-1}V^{s^j}_* q)$, where $V^{s^j}_* = 1 + \partial_{s^j}$ denotes the operation of shifting by $s^j$. By the induction hypothesis, $s^{m^{\binom{k}{2}}}\in K'$ and therefore so is $s^{m^{\binom{k+1}{2}}}$. By \eqref{K-def} and the computation above, it suffices to show that $\sum_{j=0}^{m^{\binom{k+1}{2}}-1} V^{s^j}_* \varphi(q)$ vanishes. But from Lemma \ref{alg-lemma} again we have
$$ \sum_{j=0}^{m^{\binom{k+1}{2}}-1} V^{s^j}_* \in \langle \partial_s \rangle \partial_s^{\binom{k+1}{2}}$$
(viewing $\partial_s$ as an operator on the $m$-torsion group $\M(\ZZ^{k-1}(\XX), W_k)$), and hence
$\sum_{j=0}^{m^{\binom{k+1}{2}}-1} V^{s^j}_* \varphi(q)$ has degree at most $k-1 - \binom{k+1}{2} < 0$, giving the claim.
\end{proof}

As a consequence we can find good nilspace fibrations for the double coset system from Theorem \ref{inductive}, in the spirit of \cite[Theorem 1.7]{CGSS}.  (See Definition \ref{nilspace-def}(iii) for the definition of a nilspace fibration.)

\begin{corollary}[Existence of good fibration]
\label{fibration}
Let $\Gamma$ be a countable abelian $m$-torsion group for some $m \geq 1$. Let $\XX$ be an ergodic $\Gamma$-system of some order $k \geq 1$. Then there exists a nilspace fibration  $\psi \colon \mathcal{W}\rightarrow \XX$, where 
$\mathcal{W}$ is a filtered abelian group of the form
\begin{equation}\label{Wdef}
\mathcal{W} = \mathcal{D}^1(W_1) \times \dots \times \mathcal{D}^k(W_k)
\end{equation}
for some $m^{\binom{k}{2}}$-torsion abelian groups $W_1,\dots,W_k$.
\end{corollary}

\begin{proof}  By Theorem \ref{inductive} we may write $\XX = K\backslash \mG/\Lambda$ with  $\YY$, $\mG$, $\Lambda$ and $K$ as in that theorem. Thus, we can view $\XX$ as a nilspace (see Definition \ref{def-filtration}(iii)).
As in the previous proposition, $\mG/\Lambda$ is isomorphic to the nilspace $\mathcal{U}$. Furthermore, since $\mathcal{U}$ is a subgroup of $\mG$, Proposition \ref{more-inductive}(iii) implies that $m^{\binom{k}{2}}\mathcal{U}\subseteq K$. Let ${\mathcal W}$ denote the filtered abelian group \eqref{Wdef} with $W_i \coloneqq U_j / m^{\binom{k+1}{2}} U_j$. The (continuous) factor map $\pi:\mG/\Lambda\rightarrow K\backslash \mG/\Lambda$ maps $Z_i(\mG/\Lambda)$ to $Z_i(K\backslash \mG/\Lambda)$ and therefore is a fibration (see \cite[Proposition 4.6]{host2005nonconventional} and \cite[Lemma 3.3.8]{candela}). The quotient map $\mathcal{U}\rightarrow \mathcal{W}$ is also a fibration by definition. We deduce that $\psi: \mathcal{W}\rightarrow K\backslash \mG/\Lambda$ defines a fibration, as required. 
\end{proof}

\section{Application to the inverse theorems for the Gowers uniformity norms}\label{sec-gowers}

In this section we prove Theorem \ref{gowers-main}.  Before we turn to the proof, we record two preliminary tools.

Firstly, we take advantage of Theorem \ref{pSylow} in order to reduce matters to the case where $m$ is a power of a prime.  Let $m\geq 2$ be any integer, and let $\Gamma$ be an $m$-torsion countable abelian group. Let $\XX$ be an ergodic $\Gamma$-system of order $k$. By Theorem \ref{pSylow} we can write $\XX = \biggenprod_{p|m} \XX_p$, where for each $p|m$, $\XX_p$ is an ergodic $\Gamma_p$-system of type $k$. Applying the results in the previous section for each $\XX_p$ separately, we can write $\XX_p = K_p\backslash \mG_p/\Lambda_p$ as a double coset system.

The key to reducing to the prime power case is

\begin{lemma}[Schur--Zassenhaus for nilspaces]\label{SchurZassenhaus}
Let $p,q$ be distinct primes and let $l,k,m \geq 1$. Let $X$ be a $k$-step nilspace whose structure groups are $p^m$-torsion. Then any nilspace morphism $g \colon \mathcal{D}^1(\Z/{q^l}\Z)\rightarrow \XX$ is a constant.
\end{lemma}

\begin{proof}
We prove the claim by induction on the step $k$ of $\XX$. If $\XX$ is of order $1$ then it is isomorphic to a $p^m$-torsion group equipped with the abelian filtration. In that case $g$ is an affine map from $\Z/{q^l}\Z$ to a $p^m$-torsion group and hence constant. Now assume inductively that $k>1$ and the claim holds for all smaller values of $k$. Measurably, we can write $X = \mathcal{F}_{k-1}(X)\times U$ for some $p^m$-torsion group $U$ (see Definition \ref{nilspace-def}(iii) for the definition of $\mathcal{F}_{k-1}(X)$). As in \cite[Definition 3.3.25]{candela}, the cubes on $X$ are determined by a potential $F \colon \mathcal{F}_{k-1}(X)^{[k]}\rightarrow U$. More specifically, $c\in C^k(X)$ is a cube if $c_\omega = (x_\omega,t_\omega)$ and $\sum_{\omega\in \{0,1\}^{k}}(-1)^{|\omega|}t_\omega = F(\bm{x})$ where $\bm{x}=(x_\omega)$, while $c\in C^i(X)$ is a cube for $i\not = k$ if $x\in C^i(\mathcal{F}_{k-1}(X))$ and any $k$-dimensional face is in $C^k(\mathcal{F}_{k-1}(X))$.

Write $g(x) = (g_{k-1}(x),u(x))$, by induction, the nilspace map $g_{k-1}(x) = x_0$ is a constant. We deduce that $u-F(\bm{x}_0):\mathcal{D}^1(\Z/q\Z)\rightarrow \mathcal{D}^k(U)$ is a nilspace map, where $\bm{x}_0 = (x_0)_{\omega\in \{0,1\}^k}$. Therefore, $u:\Z/{q^l}\Z\rightarrow U$ is a polynomial of degree at most $k$. The derivative $\nabla^k u \colon \Z/q^l \Z \times \dots \times \Z/q^l\Z \to U$ is multilinear and hence trivial since $p,q$ are coprime; thus $u$ is in fact of degree $k-1$. Repeating this argument by induction we see that $u$ must be a constant, closing the induction.
\end{proof}
The next tool records some polynomiality properties of a certain ``residue map''.  Given an element $x \in \Z/d_1 \Z \times \dots \times \Z/d_n\Z$ of a product of cyclic groups, let $|x| \in \Z^n$ denote the representative of $x$ in the standard fundamental domain $\{0,\dots,d_1-1\} \times \dots \times \{0,\dots,d_n-1\}$.

\begin{lemma}\label{nilspacedegree}  Let $p$ be a prime, let $\Z/d_1 \Z \times \dots \times \Z/d_n\Z$ be a product of cyclic groups which is $p^r$-torsion for some $r \geq 1$, and let $W$ be a $p^s$-torsion abelian group for some $s \geq 1$.  Let $P \in \Poly_{\leq k}(\Z^n,W)$.  Then the map $x \mapsto P(|x|)$ is a polynomial of degree at most $k s(p^r-1)$.
\end{lemma}

\begin{proof}  Of course we may assume that $d_1,\dots,d_n > 1$. By Taylor expansion, we can express $P(|x|)$ as the sum of finitely many terms of the form
$$ c \binom{|x_1|}{a_1} \dots \binom{|x_n|}{a_n}$$
with $a_1,\dots,a_n > 0$ and $c \in W$ are such that $a_1+\dots+a_n \leq k$ and $a_i < d_i$, where $x = (x_1,\dots,x_n)$.  As $W$ is $p^s$-torsion, it thus suffices to show that each monomial
$$x \mapsto \binom{|x_1|}{a_1} \dots \binom{|x_n|}{a_n} \mod p^s$$
is a polynomial of degree at most $ks(p^r-1)$ from $\Z/d_1 \Z \times \dots \times \Z/d_n\Z$ to $\Z/p^s\Z$. By Lemma \ref{Tensor}, it suffices to show that each term
$$ x_i \mapsto \binom{|x_i|}{a_i} \mod p^s$$
is a polynomial of degree at most $a_i s(p^r-1)$ from $\Z/d_i\Z$ to $\Z/p^s\Z$, or equivalently (by Taylor expansion) that
$$ \partial_1^{a_i s(p^r-1) + 1} \binom{|x_i|}{a_i} = 0 \mod p^s$$
for all $x_i \in \Z/d_i\Z$ (where the derivative is understood to be in the $x_i$ variable).  We may take $a_i \geq 1$, since the claim is trivial for $a_i=0$.  But from the Pascal triangle identity we have
\begin{align*}
\partial_1 \binom{|x_i|}{a} &= \left(\binom{|x_i|+1}{a} - \binom{|x_i|}{a}\right) - \left(\binom{|x_i|+1}{a} - \binom{|x_i+1|}{a}\right)\\ 
&= \binom{|x_i|}{a-1} - \binom{d_i}{a} \binom{|x_i|}{d_i-1}
\end{align*}
for any $1 \leq a < d_i$, while of course
$$ \partial_1 \binom{|x_i|}{0} = 0.$$
Note also that $\binom{d_i}{a}$ is divisible by $p$ whenever $1 \leq a < d_i$.  A routine induction then shows that for any $j \geq 0$, $\partial_1^j \binom{|x_i|}{a}$ can be expressed as a combination of terms $p^b \binom{|x_i|}{c}$ with $c - b(d_i-1) \leq a-j$ and $0 \leq c < d_i$.  In particular, for $j = a + s(d_i-1)$ we must have $b \geq s$.  In particular
$$ \partial_1^{a_i + s(d_i-1)} \binom{|x_i|}{a_i} = 0 \mod p^s.$$
Since 
$$a_i + s(d_i-1) \leq 1 + a_i s(d_i-1) \leq 1 + a_i s(p^r-1),$$
the claim follows.
\end{proof}

\subsection{Proof of the inverse theorem for the Gowers uniformity norms over finite fields}

We can now begin the proof of Theorem \ref{gowers-main}.  We follow the nonstandard analysis approach from \cite{tz-lowchar}, \cite{jt21-1}.  We refer the reader to \cite[Appendix A]{tz-lowchar} or \cite[Appendix A]{jt21-1} for the nonstandard formalism used here.

Assume by contradiction that the claim fails for some $m$, $k\geq 1$ and $\delta>0$. Let $C=C(k,m)$ be a sufficiently large constant to be chosen later. Then for every $n \geq 1$, we can find a finite $m$-torsion group $G^{(n)}$ and a $1$-bounded function $f^{(n)} \colon G^{(n)}\rightarrow \C$ with $\|f^{(n)}\|_{U^{k+1}(G^{(n)})}>\delta$ so that there does not exists any polynomial $P\in \Poly_{C}(G^{(n)})$ such that \begin{equation}\label{contradiction}|\E_{x\in G^{(n)}} f^{(n)}(x) e(-P(x))|\geq \frac{1}{n}.
\end{equation}
Fix a non-principal ultrafilter $\alpha$. Let $G$ be the ultraproduct of the $G^{(n)}$, and let $f \coloneqq \lim_{n\rightarrow \alpha}f^{(n)}$. Then we can endow $G$ with the Loeb
probability space structure (see \cite[Appendix A]{jt21-1}), and $f\in L^\infty(G)$. Then we have $\|f\|_{U^{k+1}(G)} \geq \delta$, where $\|\cdot\|_{U^{k+1}(G)}$ is the natural counterpart of the Gowers norms in the nonstandard case; see \cite[(4.21)]{jt21-1}.
We now invoke a correspondence principle obtained by the first and third author (see also \cite{towsner} for a closely related principle).

\begin{proposition}\label{correspondence}
Let $m$ be a fixed natural number. Let $G = (G,+)$ be a nonstandard finite abelian $m$-torsion group, equipped with the Loeb probability space structure, and let $\mathcal{F}$ be an at most countable subset of $L^\infty(G)$. Then there
exists an ergodic separable $(\Z/m\Z)^\omega$-system $\XX$, which (as a probability space) is a factor of (the inseparable probability space) $(G,\mathcal{L}_G,\mu_G)$ (so in particular we can embed $L^\infty(\XX)$ in $L^\infty(G)$) with the following
properties:
\begin{itemize}
    \item[(i)] (Equivalence of Host--Kra and Gowers inner products) For any $n\geq 0$ and any tuple $(f_\omega)_{\omega\in\{0,1\}^n}$ of bounded functions on $\XX$, one has
    $$ \left\langle (f_\omega)_{\omega\in\{0,1\}^n}\right\rangle_{U^n(\XX)} =\left\langle (f_\omega)_{\omega\in\{0,1\}^n}\right\rangle_{U^n(G)}.$$
\item[(ii)] ($\mathcal{F}$ is modeled) We have $\mathcal{F}\subseteq L^\infty(\XX)$. 
\end{itemize}
\end{proposition}

\begin{proof}
Apply \cite[Proposition 5.1]{jt21-1} to get a $\Z^\omega$-system $\XX$, and observe that any $\Z^\omega$-action on the $m$-torsion group $G$ by translations factors through a $(\Z/m\Z)^\omega$-action.
\end{proof}

Let $\XX$ be as in the above proposition, and factorize  $m = \prod_{p|m} p^{\nu_p(m)}$.  By Theorem \ref{pSylow} and Proposition \ref{inductive} we can write
$$ \XX = \biggenprod_{p|m} \XX_p$$
for some order $k$ ergodic double coset $(\Z/p^{\nu_p(m)}\Z)^\omega$-systems $\XX_p = K_p\backslash \mG_p/\Lambda_p$.  In particular (as in \cite{tz-lowchar} or \cite{jt21-1}) one can find a continuous function
$$ F \colon \biggenprod_{p|m} \XX_p \to \C$$
such that
$$ \int_G f(x) \overline{F(\pi(x))}d\mu_G(x) \neq 0$$
where $\pi \colon G \to \XX$ is the factor map.   From Proposition \ref{more-inductive}(ii), each $\XX_p$ has the structure of a compact nilspace compatible with the $\Gamma$-system structure.  Now we argue similarly to the proof of \cite[Theorem 1.5]{candela-szegedy-inverse}.  By \cite[Theorem 2.7.3]{candela}, $\XX_p$ is the inverse limit of compact finite rank $k$-step nilspaces $\XX_{p,\beta}$, with the factor maps $\Pi_{p,\beta} \colon \XX_p \to \XX_{p,\beta}$ continuous fibrations.  By perturbing $F$ slightly in uniform norm, we may thus assume without loss of generality that $F = F_\beta \circ \Pi_\beta$ for some continuous $F_\beta \colon \biggenprod_{p|m} \XX_{p,\beta} \to \C$ and some $\beta$, where $\Pi_\beta \colon \XX \to \prod_{p|m} \XX_{p,\beta}$ is the map $\Pi_\beta((x_p)_{p|m}) = (\Pi_{p,\beta}(x_p))_{p|m}$, thus
$$ \int_G f(x) \overline{F_\beta(\Pi_\beta \circ \pi(x))}d\mu_G(x) \neq 0$$

At this point we adapt the arguments from \cite[Lemmas 7.2, 7.3]{jt21-1}.  We begin by claiming that the map $\pi$ is \emph{almost polynomial} in the sense that 
$$ (\pi(x_\omega))_{\omega \in \{0,1\}^n} \in C^n(\XX)$$
for all (standard) $n\geq 0$ and $\mu_{\HK^n(G)}$-almost all $(x_\omega)_{\omega \in \{0,1\}^n}$, where $\mu_{\HK^n(G)}$ denotes Loeb measure on $\HK^n(G)$.  Since the space $\HK^n( \XX )$ is second countable, it suffices (as in the proof of \cite[Lemma 7.2]{jt21-1}) to show that
$$ \int_{\HK^n(G)} \prod_{\omega \in \{0,1\}^n} 1_{\pi^{-1}(U_\omega)}(x_\omega) d\mu_{\HK^n(G)}((x_\omega)_{\omega \in \{0,1\}^n}) = 0$$
whenever $U_\omega$ are open subsets of $\XX$ such that $\prod_{\omega \in \{0,1\}^n} U_\omega$ is disjoint from $C^n( \XX )$.  Repeating the proof of \cite[Lemma 7.2]{jt21-1}, the integral here can be re-expressed as a Gowers--Host--Kra inner product
\begin{equation}\label{uwx}
 \langle (1_{U_\omega})_{\omega \in \{0,1\}^n} \rangle_{U^n(\XX)}.
 \end{equation}
However, since $\prod_{\omega \in \{0,1\}^n} U_\omega$ avoids $C^n( \XX ) = \HK^n(\XX)$, we have
$$ \prod_{\omega \in \{0,1\}^n} 1_{U_\omega}(T^{\sum_{i=1}^n \omega_i h_i} x) = 0$$
for all $x \in \XX$ and $h_1,\dots,h_n \in \Gamma$; taking multiple ergodic averages along F{\o}lner sequences we conclude that \eqref{uwx} vanishes as claimed.  Thus $\pi$ is almost polynomial, which implies that $\Pi_\beta \circ \pi$ is also almost polynomial.

Applying \cite[Theorem 4.2]{candela-szegedy-inverse} (restated in the language of nonstandard analysis as in \cite[Lemma 7.3]{jt21-1}) we can find an internal nilspace morphism $g \colon G\rightarrow {}^*X_\beta$ so that $\Pi_\beta \circ \pi = \mathrm{st} g$ a.e., hence
$$ F_\beta(\Pi_\beta \circ \pi(x)) = F_\beta(\mathrm{st} g(x)) = \mathrm{st} F_\beta(g(x)).$$
Writing $g= \lim_{n\rightarrow\alpha} g^{(n)}$ where $g^{(n)} \colon G^{(n)} \rightarrow \XX_\beta$ are nilspace morphisms, we conclude that
$$\mathrm{st}\lim_{n\rightarrow\alpha} \E_{x\in G^{(n)}} f^{(n)}(x) F_\beta(g^{(n)}(x))\neq 0.$$

Since $G^{(n)}$ is a finite abelian $m$-torsion group, we can write $G^{(n)} = \biggenprod_{p|m} G^{(n)}_p$, where $G^{(n)}_p$ is a finite abelian $p^{\nu_p(m)}$-torsion group.  Meanwhile, the $\XX_{p,\beta}$ are nilspace factors of $\XX_p$, whose structure groups are $p^{\nu_p(m)}$-torsion by Theorem \ref{bounded-tor}; hence by \cite[Lemma 3.3.8]{candela0} the structure groups of $\XX_{p,\beta}$, being quotients of the structure groups of $\XX_{p}$, are also $p^{\nu_p(m)}$-torsion.  From Lemma \ref{SchurZassenhaus} we conclude that for each $p|m$, the $\XX_{p,\beta}$ component of $g^{(n)}$ is constant with respect to every factor $G^{(n)}_q$ with $q \neq p$.  Thus there exist nilspace morphisms $g^{(n)}_p \colon G^{(n)}_p \to \XX_{p,\beta}$ such that
$$ g^{(n)}( (x_p)_{p|m} ) = (g^{(n)}_p(x_p))_{p|m}$$
whenever $x_p \in G^{(n)}_p$.  Next, we use Proposition \ref{fibration}, composed with the fibration from $\XX_p$ to $\XX_{p,\beta}$ to obtain a nilspace fibration 
$\psi_p \colon \mathcal{W}_p\rightarrow \XX_{p,\beta}$ for each $p|m$, where $\mathcal{W}_p$ is a filtered abelian metrizable $p^{\binom{k}{2} \nu_p(m)}$-torsion group.  We can factor
$$ G^{(n)}_p = \bigoplus_{j=1}^{N^{(n)}_p} \Z/d^{(n)}_{j,p} \Z$$
for some $d^{(n)}_{j,p} | p^{\nu_p(m)}$, which is a quotient of $\Z^{N^{(n)}_p}$.  The nilspace morphism $g^{(n)}_p$ then induces a nilspace morphism from $\Z^{N^{(n)}_p}$ to $\XX_{p,\beta}$, which by Lemma \ref{fibration-lift} lifts to a nilspace morphism $h^{(n)}_p \colon \Z^{N^{(n)}_p} \to \mathcal{W}_p$, so in particular
$$ g^{(n)}_p(x) = \psi_p( h^{(n)}_p(|x|) )$$
for all $x \in G^{(n)}_p$.  Putting all this together, we conclude that
$$\mathrm{st}\lim_{n\rightarrow\alpha} \E_{x\in G^{(n)}} f^{(n)}(x) F_\beta \circ \psi( (h^{(n)}_p(|x_p|))_{p|m} )\neq 0$$
where $\psi \colon \prod_p \mathcal{W}_p \to \prod_{p|m} \XX_{p,\beta}$ is the direct product of the factor maps $\psi_p$.  Using the Stone--Weierstrass theorem to approximate $F_\beta \circ \psi$ uniformly by linear combinations of characters, we conclude that there exist characters $\xi_p \in \widehat{{\mathcal W}_p}$ for each $p|m$ such that
\begin{equation}\label{thecorrelation}\mathrm{st}\lim_{n\rightarrow\alpha} \E_{x\in G^{(n)}} f^{(n)}(x) e\left(-\sum_{p|m} \xi_p \circ h^{(n)}_p(|x_p|)\right) \neq 0.
\end{equation}
Since $\mathcal{W}_p$ is $p^{\binom{k}{2} \nu_p(m)}$-torsion, $\xi_p$ takes values in $\frac{1}{p^{\binom{k}{2} \nu_p(m)}}\Z/\Z$.  By Lemma \ref{nilspacedegree}, the map $x \mapsto \sum_{p|m} \xi_p \circ h^{(n)}_p(|x_p|)$ will be a polynomial of degree $C$, if $C=C(k,m)$ is sufficiently\footnote{Indeed, our arguments allow us to take $C(k,m) = k \binom{k}{2} \sup_{p|m} p^{\nu_p(m)} (p^{\nu_p(m)} - 1)$, though it is definitely possible to lower this value of $C$ further with more care.} large.  But this now contradicts the inability to satisfy \eqref{contradiction}, for all $n$ in some $\alpha$-large set. The claim follows.

To prove Theorem \ref{gowersextension}, set $C_2:=C$ and observe that we only need to show that the polynomials $h_p^{(n)}\circ |\cdot|$ takes the form $P_n(|\cdot|)$ where $P_n:H\rightarrow \mathbb{T}$ for some bounded extension $q:H\rightarrow G$. Equivalently, we prove that there exists a constant $C_1=C_1(k,m)$ so that $h_p^{(n)}$ is invariant to translation by the coset $C_1\cdot\mathbb{Z}^{N_p^{(n)}}$, we can then set $H_n:=(\mathbb{Z}/C_1\mathbb{Z})^{N_p^{(n)}}$ and $P_n$ the projection of $h_p^{(n)}$ to $H_n$. 

By the structure of $\mathcal{W}_p$, $h_p^{(n)} = ((h_p^{(n)})_1,...,(h_p^{(n)})_k)$ where $(h_p^{(n)})_i:\mathcal{D}^1(\mathbb{Z}^{N_p^{(n)}})\rightarrow \mathcal{D}^i(U_i)$ is a nilspace morphism. By definition, this implies that $(h_p^{(n)})_i:\mathbb{Z}^{N_p^{(n)}}\rightarrow U_i$ is a polynomial of degree $\leq i$. To complete the proof we will apply the following Lemma.

\begin{lemma}
Let $p$ be a prime and let $d,n,m\in \mathbb{N}$. Let $U$ be a $p^m$-torsion group and let $Q:\mathbb{Z}^n\rightarrow U$ be a polynomial of degree $\leq d$. Then $Q$ is invariant to translations by every $t\in p^{dm}\mathbb{Z}^n$.
\end{lemma}
\begin{proof}
    We induct on $d\geq 1$. When $d=1$, $Q$ is an affine map and the claim follows. Let $d>1$ and assume inductively that the claim holds for all smaller values of $d$. Let $s\in \mathbb{Z}^n$ be arbitrary, by the induction hypothesis $\partial_t \partial_s Q = 0$ for all $t\in p^{C_{d-1}}\mathbb{Z}^n$. We deduce that $\partial_t Q =\phi(t)$ is a constant for all $t\in p^{C_{d-1}}\mathbb{Z}^n$. By the cocycle identity, we also have that $\phi:p^{m(d-1)}\mathbb{Z}^n\rightarrow \mathbb{T}$ is a homomorphism. 

    Since $U$ is $p^m$-torsion, we deduce that $\phi(p^mt) = p^m \partial_t Q = 0$. Equivalently, $Q$ is invariant to translations by every $t\in p^{dm}\mathbb{Z}/\mathbb{Z}$, as required.
\end{proof}

Since $(h_p^{(n)})_i$ are of degree at most $k$ and takes values in an at most $p^{m\binom{k}{2}}$-group, we deduce by applying the previous lemma once for every $1\leq i \leq k$, that $h_p^{(n)}$ is invariant to translations by every $t\in p^{mk\binom{k}{2}}\mathbb{Z}^n$. Setting $C_2(k,m):=mk\binom{k}{2}$ the claim follows.
\appendix

\section{Host--Kra structure theory}\label{host-kra-theory}

In this appendix, we collect some needed results in Host--Kra structure theory from the pioneering work of Host and Kra \cite{host2005nonconventional} for $\Z$-actions, and later work by Bergelson, Tao, and Ziegler \cite{btz} and Shalom \cite{shalom3,shalom2} who studied a structure theory for other countable abelian group actions.   
The basic results in Host--Kra theory are well recorded in the textbook \cite{hk-book} for (ergodic) $\Z$-actions. 
Many of these results extend without difficulty to arbitrary countable abelian group actions as was previously observed by several authors, see e.g., \cite[Appendix A]{btz}, \cite[\S 2]{jst}.  
We start by gathering some of these basic facts. 
Throughout this section, we fix a countable discrete abelian group $\Gamma$ unless mentioned otherwise. 

\subsection{Measure-preserving systems}\label{system-sec}

Intuitively, a $\Gamma$-system should be a probability space $(X,\X,\mu)$ equipped with a measure-preserving action of $\Gamma$, and one system $(Y,\Y,\nu)$ should be a factor of another $(X,\X,\mu)$ if there is a factor map $\pi \colon X \to Y$ that intertwines with the action and pushes forward the measure $\mu$ to the measure $\nu$.  However, for technical reasons\footnote{The main reason for this is that our measurable cocycles $\rho$ will only obey identities such as the cocycle equation almost everywhere, rather than everywhere.} it is more correct to work with \emph{near-actions} which are only defined (and are actions) almost everywhere rather than everywhere, and the factor relation has to be defined more abstractly at the probability algebra level.  The precise definitions are as follows:

\begin{definition}[Measure-preserving systems]\label{system-def}  Let $\Gamma = (\Gamma,+)$ be a discrete countable abelian group.
\begin{itemize}
    \item A \emph{$\Gamma$-system} is a quadruplet $\XX = (X,\X,\mu,T)$, where $(X,\X,\mu)$ is a Lebesgue probability space (a standard Borel space equipped with a probability measure $\mu$), and $T\colon \Gamma\times X\to X$ is a \emph{near-action} of $\Gamma$, thus $T^{\gamma_1+\gamma_2}(x)=T^{\gamma_1}\circ T^{\gamma_2}(x)=T^{\gamma_2}\circ T^{\gamma_1}(x)$ for almost every $x\in X$ for all $\gamma_1,\gamma_2\in\Gamma$, and $T^\gamma\colon X\to X$ is measure-preserving for each $\gamma\in\Gamma$.  
    \item If $(X,\X,\mu,T)$ is a $\Gamma$-system, its \emph{abstraction} is the triplet $(\X/\sim, \mu/\sim, T^*/\sim)$, where $\X/\sim$ is the complete Boolean algebra formed by quotienting $\X$ by null sets, $\mu/\sim \colon \X/\sim \to [0,1]$ is $\mu$ similarly quotiented by null sets, and for each $\gamma \in \Gamma$, $(T^\gamma)^*/\sim \colon \X/\sim \to \X/\sim$ is the pullback map $(T^\gamma)^*(E) \coloneqq (T^\gamma)^{-1}(E)$, again quotiented by null sets.  Note that as $(X,\X,\mu)$ is a Lebesgue probability space, $(\X/\sim,\mu/\sim)$ is a separable probability algebra, and the pullback maps $(T^\gamma)^*/\sim$ are measure-preserving.
    \item We say that one $\Gamma$-system $\YY = (Y,\Y,\nu,S)$ is a \emph{factor} of another $\Gamma$-system $\XX = (X,\X,\mu,T)$ (or equivalently that $\XX$ is an \emph{extension} of $\YY$), and write $\YY \leq \XX$, if there is an ``abstract factor pullback map'' $\pi^*/\sim \colon \Y/\sim \to \X/\sim$ which is measure-preserving and intertwines with the action, thus
    $$ (\mu/\sim)( (\pi^*/\sim)(E) ) = (\nu/\sim)(E)$$
    and
    $$ ((T^\gamma)^*/\sim) (\pi^*/\sim) E = (\pi^*/\sim) ((S^\gamma)^*/\sim) E$$
    for all $E \in \Y/\sim$ and $\gamma \in \Gamma$.  If $\pi^*/\sim$ is invertible, we say that $\XX$ and $\YY$ are \emph{isomorphic}.
    \item A $\Gamma$-system $\XX = (X,\X,\mu,T)$ is ergodic if the invariant algebra $\{ E \in \X/\sim: ((T^\gamma)^*/\sim) E = E \forall \gamma \in \Gamma \}$ is the trivial algebra $\{0,1\}$.
\end{itemize}
\end{definition}

Informally, two $\Gamma$-systems are isomorphic if they are equivalent ``up to null sets'', although the equivalence need not be realizable as a \emph{pointwise} map between the systems.  See \cite[Appendix A]{jst} for further discussion.  Observe that if $\YY$ is a factor of $\XX$, then $L^\infty(\YY)$ can be identified with a $\Gamma$-invariant closed subalgebra of $L^\infty(\XX)$, and if two factors of $\XX$ are identified with the same such subalgebra then they are isomorphic.  Conversely, every $\Gamma$-invariant closed subalgebra of $L^\infty(\XX)$ gives rise to a factor; see \cite[Proposition A.3]{jst}.  

Let us say that an abstract factor pullback map $\pi^*/\sim \colon \Y/\sim \to \X/\sim$ is \emph{representable} by a measurable map $\pi \colon X \to Y$ (which we call a \emph{(concrete) factor map} if $\pi$ intertwines the $\Gamma$-actions almost everywhere (thus $\pi \circ S^\gamma = T^\gamma \circ \pi$ almost everywhere for each $\gamma \in \Gamma$) and $\pi^*/\sim$ is given by the pullback map $E \mapsto \pi^{-1}(E)$ up to almost everywhere equivalence.  In general, an abstract factor pullback map need not have a concrete representation for the original $\Gamma$-systems $\XX, \YY$, but it turns out that such representations exist if one is willing to replace the original $\Gamma$-systems $\XX, \YY$ by isomorphic models $\hat \XX, \hat \YY$ (for instance, the ``Cantor models'' of these systems); see \cite[Proposition A.3(ii)]{jst}. Also, concrete representations always exist (and are unique up to almost everywhere equivalence) if the base space $\YY$ is Polish; see \cite[Proposition 3.2]{jt19}.  Because the iterated skew-products \eqref{structuregroups} are Polish, this means that for systems of some finite order $k$ we will be able to work with concrete factor maps without difficulty.

Observe that if $\XX$ is a $\Gamma$-system and $\Gamma'$ is an extension of $\Gamma$ in the sense that there is a surjective homomorphism $\phi \colon \Gamma' \to \Gamma$ of countable abelian groups, then $\XX$ can also be viewed as a $\Gamma'$ system, by composing the $\Gamma$-near action with $\phi$ to create a $\Gamma'$-near action.  We say that a $\Gamma'$-system $\YY$ is a \emph{generalized extension} of $\XX$ (or that $\XX$ is a \emph{generalized factor} of $\YY$) if it is an extension of $\XX$ when viewed as a $\Gamma'$-system rather than a $\Gamma$-system.

We define some special types of $\Gamma$-systems.

\begin{definition}[Translational, rotational, and double coset systems]\label{translational-def}  Let $\Gamma$ be a countable abelian group.
\begin{itemize}
    \item A \emph{translational $\Gamma$-system} is a system of the form $G/\Lambda$, where $G$ is a Polish group, $\Lambda$ is a closed cocompact subgroup of $G$, the compact quotient space $G/\Lambda$ is also Polish and endowed with a $G$-invariant probability measure, and the action $T$ of $\Gamma$ is given by $T^\gamma g \Lambda = \phi(\gamma) g \Lambda$ for some homomorphism $\phi \colon \Gamma \to G$ and all $\gamma \in \Gamma$, $g \in G$.  If $G$ is abelian (so that $Z = G/\Lambda$ is a compact abelian group), we refer to $Z$ as a \emph{rotational $\Gamma$-system}.
    \item A \emph{double coset $\Gamma$-system} is a system of the form $K\backslash G/\Lambda$, where $G/\Lambda$ is a translational $\Gamma$-system, $K$ is a closed subgroup of $G$ normalized by $\phi(\Gamma)$, the compact space $K \backslash G / \Lambda$ is also Polish and is given the probability measure inherited from $G/\Lambda$ by projection, and the action $T$ of $\Gamma$ is given by $T^\gamma K g \Lambda = K \phi(\gamma) g \Lambda$.  Note that this double coset system $K \backslash G/\Lambda$ is a factor of the translational system $G/\Lambda$ with factor map $g \Lambda \mapsto K g \Lambda$.
    \item When the group $G$ is nilpotent of class $k$ we say that the translational system or the double coset system associated with $G$ is of degree $k$.
\end{itemize}
\end{definition}

We caution that the group $G$ in Definition \ref{translational-def} is \emph{not}\footnote{For instance, $G$ might contain a copy of $\Poly_{\leq 1}(\T, \T^\N) = \Poly_{\leq 1}(\T)^\N = (\Z \oplus \T)^\N$, which is Polish but not locally compact.} assumed to be locally compact.  In particular, the theory of Haar measure is not necessarily available, and the existence of a $G$-invariant probability measure on $G/\Lambda$ is part of the definition of a translational system and not an automatic consequence from the other hypotheses.  The Polish nature\footnote{On the other hand, $G/\Lambda$ is automatically Polish when $\Lambda$ is a closed subgroup of the Polish group $G$; see \cite[Proposition 1.2.3]{kechrisbecker}.} of $K \backslash G/\Lambda$ is similarly not a consequence of the other hypotheses, and must be established separately.  On the other hand, the translational and double coset systems we actually use in this paper satisfy additional properties (involving a filtration structure on $G$) that are not captured by the above definition (see \cite[Theorem 1.7]{jst} for the $k=2$ case of this phenomenon, as well as the ``groupable'' axiom from \cite[Definition 1.8]{cgss-2023}).  It may therefore be that one should impose additional axioms on these systems in order to get the most suitable category of systems for applications, but we do not pursue this matter here. 

\subsection{Host--Kra factors}

We recall the construction of the Host--Kra factors from \cite[\S 3]{host2005nonconventional}, \cite[Chapter 9.1]{hk-book}, or \cite[Appendix A]{btz}.  Given a $\Gamma$-system $\XX = (X, \mu, T)$, we can recursively define the Host--Kra parallelepiped $\Gamma$-systems $\XX^{[k]} = (X^{[k]}, \mu^{[k]}, T^{[k]})$ for $k \geq 0$ by setting $\XX^{[0]} \coloneqq \XX$ and
$$ \XX^{[k+1]} \coloneqq \XX^{[k]} \times_{\ZZ^0(\XX^{[k]})} \XX^{[k]}$$
where the right-hand side is the relatively independent product of $\XX^{[k]}$ with itself over the invariant factor $\ZZ^0(\XX^{[k]})$; see \cite[Chapter 5]{furstenberg2014recurrence} for the construction of relatively independent product for Lebesgue probability spaces.
For $k\geq 1$, one can show that there exists a unique (up to isomorphism) factor $\ZZ^k(\XX)$ of $\XX$ with the property that for $f\in L^\infty(\XX)$, $$\int_{X^{[k]}} \bigotimes_{\omega\in\{0,1\}^k} {\mathcal C}^{|\omega|} f d\mu^{[k]} =0 \iff \E(f\,| \, \ZZ^{k-1}(\XX))=0,$$ where $ \bigotimes_{\omega\in\{0,1\}^k} f$ is the tensor product of $2^k$ copies of $f$, ${\mathcal C} \colon z \mapsto \overline{z}$ denotes complex conjugation and $\E(\cdot| \ZZ^k(\XX))$ denotes conditional expectation onto the factor $\ZZ^k(\XX)$, cf. \cite[\S 4]{host2005nonconventional}, \cite[Theorem 7, Chapter 9]{hk-book}, or \cite[Appendix A]{btz}.  The quantity 
$$\|f\|_{U^k(X)} \coloneqq \left(\int_{X^{[k]}} \bigotimes_{\omega\in\{0,1\}^k} {\mathcal C}^{|\omega|} f d\mu^{[k]}\right)^{\frac{1}{2^k}}$$ is known as the \emph{Host--Kra--Gowers seminorm} of $f$.  

An ergodic $\Gamma$-system $\XX$ is said to be \emph{of order $k$} if $\XX=\ZZ^k(\XX)$.   We record some basic functoriality properties of $\ZZ^k$ and of systems of order $k$:

\begin{proposition}[Functoriality properties of $\ZZ^k$]\label{prop-functoriality}
(cf. \cite[\S 4]{host2005nonconventional}, \cite[Propositions 11, 17, 21, and Theorem 20; Chapter 9]{hk-book}, \cite[Lemma A.22]{btz}) 
Let $k\geq 1$. 
\begin{itemize}
\item[(i)]  A factor  of an ergodic $\Gamma$-system  of order $k$ is of order $k$.
\item[(ii)]  Let $\pi \colon \YY \to \XX$ be a (concrete) factor map between two ergodic $\Gamma$-systems. Let $\pi_k^\YY:\YY\rightarrow \ZZ^k(\YY)$, $\pi_k^\XX:\XX\rightarrow \ZZ^k(\XX)$ be the factor maps onto the $k$-th Host--Kra factors respectively. Then (after passing to a suitable concrete model if necessary) there exists a concrete factor map $\pi_{k}:\ZZ^k(\YY)\rightarrow \ZZ^k(\XX)$ such that $\pi^\XX_k\circ \pi = \pi_k \circ \pi^\YY_k$. 
\item[(iii)] An inverse limit of ergodic $\Gamma$-systems of order $k$ is an ergodic $\Gamma$-system of order $k$. 
\item[(iv)] If $\XX$ is an inverse limit of ergodic $\Gamma$-systems $\XX_i,i\in I$, then $\ZZ^k(\XX)$ is an inverse limit of $\ZZ^k(\XX_i),i\in I$.  
\end{itemize}
\end{proposition}

\subsection{Cocycles and extensions}

A useful representation of the Host--Kra factors $\ZZ^k(X)$ is in terms of a chain of abelian group skew-product extensions given by cocycles of a certain type as mentioned in the introduction; see \eqref{structuregroups}. We will introduce these notions in the next two definitions, and the representation \eqref{structuregroups} then follows from the ensuing proposition by induction.   

\begin{definition}[Cocycles, abelian group extension, and cohomology]\label{coh:def}
Let $\XX=(X,\mu,T)$ be a $\Gamma$-system and $U=(U,+)$ be a compact metrizable abelian group.
\begin{itemize}
    \item[(i)] A \emph{cocycle} on $\XX$ with values in $U$ is a map $\rho \colon \Gamma \to \M(X,U)$ (which we denote as $\gamma \mapsto \rho_\gamma$ satisfying the cocycle identity 
    \begin{equation}\label{cocycle-rho}
    \rho_{\gamma+\gamma'} = \rho_\gamma + \rho_{\gamma'} \circ T^\gamma
    \end{equation}
    for all $\gamma,\gamma'\in \Gamma$.  
    \item[(ii)] If $\rho$ is a cocycle in $\XX$ with values in $U$, the \emph{abelian group skew-product extension} $\XX \rtimes_\rho U$ of $\XX$ by $\rho$ is the $\Gamma$-system defined on the product probability space $X\times U$ (where $U$ is equipped with Haar probability measure) with the $\Gamma$-near action given by 
    $$T^\gamma(x,u) = (T^\gamma x , u+ \tilde \rho_{\gamma}(x))$$
    for all $x \in \XX, u \in U, \gamma \in \Gamma$, where $\tilde \rho_\gamma \colon X \to U$ is an arbitrarily chosen\footnote{It is because of this arbitrary nature of the representative that we work (initially, at least) with near-actions rather than genuine actions.} representative of $\rho_\gamma \in \M(X,U)$.  It is easy to see that $\XX \rtimes_\rho U$ is a well-defined (up to isomorphism) as a $\Gamma$-system.
    \item[(iii)] If $F \in \M(X,U)$, we let $dF$ be the cocycle defined by $(dF)_\gamma \coloneqq \partial_\gamma F = F \circ T^\gamma - F$.  A cocycle $\rho$ is said to be a \emph{coboundary} if it is of the form $\rho = dF$ for some $F \in \M(X,U)$. Two cocycles $\rho,\rho'$ are said to be \emph{cohomologous} if their difference $\rho-\rho'$ is a coboundary, thus $\rho-\rho'=dF$ for some $F \in \M(X,U)$.
    \item[(iv)] Let $k\geq -1$. A function $\rho \colon \Gamma \to \M(X,U)$ is said to be a \emph{quasi-cocycle of degree $k$} if one has
$$ \rho_{\gamma+\gamma'} - \rho_\gamma - \rho_{\gamma'} \circ T^\gamma \in \Poly_{\leq k}(\XX,U)$$
for every $\gamma,\gamma'\in \Gamma$.  We say that $\rho$ is a \emph{quasi-coboundary of degree $k$} if there exists $F \in \M(X,U)$ such that 
$$ \rho_\gamma - (dF)_\gamma \in \Poly_{\leq k}(\XX,U)$$
for all $\gamma \in \Gamma$.  Note that every quasi-coboundary of degree $k$ is a quasi-cocycle of degree $k$, and when $k=-1$ the notions of quasi-cocycle and quasi-coboundary collapse to that of cocycle and coboundary respectively.
\end{itemize}
\end{definition}

\begin{remark}\label{cohiso}
Cohomologous cocycles define isomorphic abelian extensions with the isomorphism being given pointwise almost everywhere by the map $(x,u)\mapsto (x,u-F(x))$ where $\rho-\rho'=dF$. 
\end{remark}

\begin{remark}[Cocycles and short exact sequences]\label{cocyc-skew}
One can also interpret cocycles using the language of short exact sequences of groups.  With the notation of Definition \ref{coh:def}, we can form the semi-direct product $\Gamma \ltimes \M(X,U)$ to be the group of pairs $(\gamma, F)$ with $\gamma \in \Gamma$ and $F \in \M(X,U)$ with group law
$$ (\gamma', F') (\gamma,F) = (\gamma + \gamma', F + F' \circ T^\gamma).$$
One easily verifies that $\Gamma \ltimes \M(X,U)$ is a group (acting on the left on $X \times U$ by the formula $(\gamma,F)(x,u) \coloneqq (T^\gamma x, u+F(x))$), with the short exact sequence
\begin{equation}\label{basic-short-exact} 
0 \to \M(X,U) \to \Gamma \ltimes \M(X,U) \to \Gamma \to 0.
\end{equation}
A cocycle $\rho$ is then essentially the same thing as a splitting $\gamma \mapsto (\gamma,\rho_\gamma)$ of this sequence, and two cocycles differ by a coboundary if and only if their associated splittings are related by a conjugation by an element of $\M(X,U)$.  In a similar vein, a quasi-cocycle $\rho$ of degree $k$ gives rise to a short exact sequence
\begin{equation}\label{quasi-short-exact}
0 \to \Poly_{\leq k}(\XX,U) \to \{ (\gamma, \rho_\gamma+P): \gamma \in \Gamma, P \in \Poly_{\leq k}(\XX,U) \} \to \Gamma \to 0,
\end{equation}
which splits when $\rho$ is a quasi-coboundary (or more generally if $\rho$ is quasi-cohomologous to a cocycle).  We will frequently study variants of these short exact sequences \eqref{basic-short-exact}, \eqref{quasi-short-exact} in this paper, with particular attention drawn to the question of when such sequences split.
\end{remark}

\begin{definition}[Functions of type $k$] \label{type:def}
(cf. \cite[Definition 4.1]{btz})   Let $\XX=(X,\mu,T)$ be a $\Gamma$-system, let $U=(U,+)$ be a compact metrizable abelian group, let $k\geq 0$, and let $\XX^{[k]}$ be the Host--Kra parallelepiped $\Gamma$-system associated with $\XX$.
	\begin{itemize}
		\item[(i)] For a measurable $f \colon X\rightarrow U$, we define $\Delta^{[k]}f \colon X^{[k]}\rightarrow U$ by $$\Delta^{[k]}f((x_\omega)_{\omega\in \{0,1\}^k}):=\sum_{\omega\in \{0,1\}^k}(-1)^{\mathrm{sgn}(\omega)}f(x_\omega)$$ where $\text{sgn}(\omega) \coloneqq \sum_{i=1}^k \omega_i$.  Note that if two functions $f,f'$ agree $\mu$-almost everywhere, then $\Delta^{[k]} f$ and $\Delta^{[k]} f'$ agree $\mu^{[k]}$-almost everywhere, so we may view $\Delta^{[k]}$ as a homomorphism from $\M(X,U)$ to $\M(X^{[k]}, U)$.  This homomorphism is $\Gamma$-equivariant, so in particular if $\rho \colon \Gamma \to \M(X,U)$ is a $U$-valued cocycle on $X$, then $\Delta^{[k]} \rho \colon \Gamma \to \M(X^{[k]},U)$ is a $U$-valued cocycle on $X^{[k]}$.
		\item[(ii)] A function $\rho \colon \Gamma \to \M(X,U)$ is said to be a \emph{function of type $k$} if $\Delta^{[k]}\rho \colon \Gamma \to \M(X^{[k]},U)$ is a coboundary on the  Host--Kra parallelepiped $\Gamma$-system $\XX^{[k]}$, that is, there exists  $F \in \M(X^{[k]},U)$ such that $\Delta^{[k]}\rho_\gamma = F \circ (T^\gamma)^{[k]} - F$ for all $\gamma\in \Gamma$.
	\end{itemize}
\end{definition}

\begin{proposition}\label{abelext}
Let $\XX=(X,\mu,T)$ be an ergodic $\Gamma$-system. Then for every $k\geq 1$, the Host--Kra factor $\ZZ^k(\XX)$ of order $k$ is (isomorphic to) an abelian group skew-product extension $\ZZ^{k-1}(\XX)\rtimes_\rho U$ of the Host--Kra factor $\ZZ^{k-1}(\XX)$ of order $k-1$ by a compact metrizable abelian group $U$ and a cocycle $\rho$ of type $k$.
\end{proposition}

\begin{proof} See \cite[Proposition 6.3]{host2005nonconventional}, \cite[Proposition 3, Chapter 18]{hk-book}, or \cite[Proposition 3.4]{btz}.  The arguments in \cite{host2005nonconventional}, \cite{hk-book} are formulated for $\Z$-systems, but (as observed in \cite{btz}) extend without difficulty to more general $\Gamma$-systems.
\end{proof}

Iterating this proposition, we see that every system of order $k$ has a canonical representation of the form \eqref{structuregroups} (with $U_1$ given the structure of a translational $\Gamma$-system.  In particular, such systems have a canonical Polish space structure, and so we can represent abstract factor maps into these spaces by concrete factor maps without difficulty.

The following results analyse the representation \eqref{structuregroups} resulting from Proposition \ref{abelext} with respect to taking factors and inverse limits.

\begin{proposition}\label{prop-invlim}\ 
    \begin{itemize}
    \item[(i)]  Let $\YY=U_1\rtimes_{\rho_1} U_2\rtimes \ldots\rtimes_{\rho_{k-1}}U_k$ be an ergodic $\Gamma$-system of order $k$ written in the form \eqref{structuregroups}. Let $\XX=U_1'\rtimes_{\rho_1'} U_2'\rtimes \ldots\rtimes_{\rho_{k-1}'} U_k'$ be a factor of $\YY$ (which must also be of order $k$ thanks to Proposition \ref{prop-functoriality}(i)). Then for every $1\leq i\leq k$ there exists a surjective group homomorphism $\varphi_i \colon U_i\rightarrow U_i'$ and $F_i \in \M(\ZZ^i(\YY),U'_{i+1})$ such that $\varphi_{i+1}\circ \rho_i = \rho_i'\circ \pi_i + dF_i$ and $\pi_{i+1}(y,u) = (\pi_i(y), \varphi_{i+1}(u)+F_i(y))$ for almost every $(y,u) \in \ZZ^{i+1}(\YY)$ where $\pi_i \colon \ZZ^{i}(\YY)\rightarrow \ZZ^{i}(\XX)$ is a concrete factor map induced by $\pi$ by Proposition \ref{prop-functoriality}(ii) (and the fact that $\ZZ^i(\XX)$ is Polish).  
    \item[(ii)] Let $\YY=U_1\rtimes_{\rho_1}U_2\rtimes \ldots\rtimes_{\rho_{k-1}}U_k$ be an ergodic $\Gamma$-system of order $k$ written in the form \eqref{structuregroups}. Suppose that $\YY$ is the inverse limit of $\YY_n = U_{1,n}\rtimes_{\rho_{1,n}}U_{2,n}\rtimes \ldots\rtimes_{\rho_{k-1,n}} U_n$ in the category of (abstract) measure-preserving $\Gamma$-systems, where $\YY_n$ are also written in the form \eqref{structuregroups}. Then for every $1\leq i\leq k$, there exists a surjective group homomorphism $\varphi_{i,n}:U_i\rightarrow U_{i,n}$ such that $U_i$ is the inverse limit of the sequence $U_{i,n}$ in the category of compact abelian groups.
\end{itemize}    
\end{proposition}

\begin{proof} See \cite[Lemma 6.11]{shalom1}.
\end{proof}

We collect some useful results about type of cocycles.

\begin{proposition}\label{type}
Let $\XX=(X,\mu,T)$ be an ergodic $\Gamma$-system, let $U$ be a metrizable compact abelian group, let $\rho=(\rho_\gamma)_{\gamma\in \Gamma}$ be cocycle on $\XX$ with values in $U$, and let $k\geq 1$.
\begin{itemize}
    \item[(i)] (Moore--Schmidt theorem) $\rho$ is a coboundary if and only if $\xi\circ \rho$ is a coboundary as a cocycle on $\XX$ with values in $\T$ for all Fourier characters $\xi$ in the Pontryagin dual $\hat U$ of $U$. 
    \item[(ii)] $\rho$ is of type $k$ if and only if for every $\xi\in\hat U$, $\xi\circ\rho$ is of type $k$.
    \item[(iii)] If $\XX$ is of order $k$, then $\XX\rtimes_\rho U$ is of order $k$ if and only if $\rho$ is of type $k$.
    \item[(iv)] If $\rho$ is a function (not necessarily a cocycle) of type $m\geq 0$ and $S\in \Aut(\XX)$ is an automorphism\footnote{An automorphism of a $\Gamma$-system $\XX$ is a measure-preserving isomorphism  $S$ of $(X,\mu)$ commuting (up to almost everywhere equivalence) with the $T$-action.} of the $\Gamma$-system $\XX$ that fixes the $\sigma$-algebra of the Host--Kra factor $\ZZ^k(\XX)$, then $\partial_S \rho$ is a function of type $\max\{m-k-1,0\}$.
    \item[(v)] Suppose $\YY=(Y,\mu,T)$ is an ergodic $\Gamma$-extension of $\XX$ with factor map $\pi:Y\rightarrow X$. If $\rho$ is of type $k$, then $\rho\circ\pi$ is of type $k$ as well. 
    \item[(vi)] Suppose that $U=\T$ is the torus and $\rho$ is a cocycle of type $1$. Then $\rho$ is cohomologous to polynomial cocycle of degree $0$, that is a homomorphism $\Gamma\to \T$. 
    \item[(vii)] If $\XX$ is not ergodic, then $\rho$ is a coboundary if and only if $\rho$ is a coboundary on every ergodic component of $\XX$. 
\end{itemize}
\end{proposition}

\begin{proof}
For $(i)$, see \cite{moore1980coboundaries}. The claim $(ii)$ follows immediately from $(i)$. The proof of $(iii),(v)$, and $(vii)$ can be found in \cite[Corollary 7.7]{host2005nonconventional}, \cite[Corollary 7.8]{host2005nonconventional}, and  \cite[Lemma 9.1]{host2005nonconventional} for $\Gamma=\Z$ respectively (see also \cite[Propositions 5, 8 and Corollary 9, Chapter 18]{hk-book} and \cite[Lemma 11, 
Chapter 5]{hk-book}) but the arguments extend without difficulty to arbitrary discrete countable abelian groups $\Gamma$. The proof of $(iv)$ was established in \cite[Lemma 5.3]{btz} for automorphisms of a specific type, but the same proof holds for arbitrary automorphisms. 

We establish $(vi)$. Results of this type have appeared in the literature before \cite{moore1980coboundaries}, \cite[Lemma 10.3]{fw}, \cite[Chapter 5, Lemma 13]{hk-book}, \cite[Proposition 2.4(vi)]{jt21-1}, but the result here is slightly more general in that we do not require $\Gamma$ to be torsion-free. 

Since $\rho$ is of type $1$ and $\mu^{[1]} = \mu \times \mu$, there exists $F \in M(X \times X, \T)$ such that
\begin{equation}\label{rho-eq}
 \rho_\gamma(x_0) - \rho_\gamma(x_1) = F(T^\gamma x_0,T^\gamma x_1) - F(x_0,x_1)
 \end{equation}
for all $\gamma \in \Gamma$ and $\mu \times \mu$-almost all $(x_0,x_1)$.  As $\rho$ is a $(\Gamma,X,\T)$-cocycle, the set
$$ G \coloneqq \{ (\gamma, \rho_\gamma + c): \gamma \in \Gamma, c \in \T \}$$
is an abelian group that near-acts on $X \times \T$.  By \eqref{rho-eq}, the integral operator $H \colon L^2(X \times \T) \to L^2(X \times \T)$ defined by
$$ H f(x_1,u_1) \coloneqq \int_{X \times \T} e( F(x_0,x_1) + u_1 - u_0 ) f(x_0,u_0)\ d\mu(x_0) du_0$$
commutes with this near-action; then $HH^*$ also commutes with this near-action.  As this operator is compact, self-adjoint, and non-trivial, there thus exists a non-trivial joint eigenfunction $\beta \in L^2(X \times \T)$ of $H H^*$ as well as the near-action of $G$, with the eigenvalue of $H$ being non-zero.  As $\beta$ is a scalar multiple of $HH^*\beta$, it must take the form $\beta(x,u) = \beta_0(x) e(u)$ almost everywhere for some non-trivial $\beta_0 \in L^2(X)$; the function $|\beta_0|$ is then $\Gamma$-invariant and thus constant by ergodicity.  We can then normalize $|\beta_0|=1$, thus we may write
$$ \beta(x,u) = e(F(x)+u)$$
almost everywhere for some $F \in M(X,\T)$.  As $\beta$ is an eigenfunction of the near-action of $(\gamma,\rho_\gamma)$, a calculation then shows that $(\rho-dF)_\gamma$ is almost everywhere constant, thus in $\Poly_{\leq 0}(X)$, for each $\gamma \in \Gamma$.
\end{proof}

The following result from \cite{shalom2} is a higher-order version of the Moore--Schmidt theorem which requires higher-order divisibility of the underlying system (and the acting group to be torsion-free), see the example after \cite[Proposition 3.8]{shalom2} for a counterexample when the divisibility hypothesis is dropped.  

\begin{proposition}[Higher order Moore-Schmidt]\label{typecircle:prop}
\cite[Proposition 3.8]{shalom2}
	Let $k\geq 1$, let $\Gamma$ be a torsion-free countable abelian group, and let $\XX$ be a $k$-divisible ergodic $\Gamma$-system. Let $U$ be a compact abelian metrizable group, let $\rho \colon \Gamma \to \M(X,U)$ be a $U$-valued cocycle on $X$, and suppose that for every $\xi\in \hat U$, $\xi \circ \rho$ is a $\T$-valued quasi-coboundary on $\XX$ of order $k-1$.  Then $\rho$ is a $U$-valued quasi-coboundary on $\XX$ of order $k-1$.
\end{proposition}

We record some relevant properties about polynomials. 
	\begin{proposition}\label{ppfacts} Let $\XX=(X,\mu,T)$ be an ergodic $\Gamma$-system.
	\begin{itemize}
	\item[(i)] Let $k\geq 1$, and suppose that $\phi,\psi\in \Poly_{\leq k}(\XX)$ and $\phi-\psi$ is non-constant. Then $$\|e(\phi)-e(\psi)\|_{L^2(\XX)} \geq \sqrt{2}/2^{k-2}$$ where $e(y)=e^{2\pi i y}$.  In particular, there are only countably many elements of $\Poly_{\leq k}(\XX)$ up to constants.
    \item[(ii)]  For any $m \geq 0$, a polynomial in $\Poly_{\leq m}(\XX)$ is measurable in $\ZZ^m(\XX)$.
	\item[(iii)] Let $m\geq 0, k\geq 1$ and let $P\in \Poly_{\leq m}(X)$. If $t\in \Aut(\XX)$ fixes the $\sigma$-algebra $\ZZ^k(\XX)$, then $\partial_{t} P$ is a polynomial of degree $\leq m-k-1$. 
	\item[(iv)] Let $m\geq 0$, let $P\in \Poly_{\leq m}(X)$, and let $K\leq \Aut(\XX)$ be a compact subgroup. Then there is an open subgroup $V\leq K$ such that $\partial_{u} P$ is a constant for every $u\in V$. 
	\item[(v)] Let $m\geq 0$ and let $f \in \M(X,\T)$. Then $f$ is a polynomial of degree at most $m-1$ if and only if $\Delta^{[m]}f(x)\equiv 0$ for $\mu^{[m]}$-almost every $x\in \XX^{[m]}$. 
    \item[(vi)]  If $\XX$ is an inverse limit of $(\XX_\alpha)_{\alpha \in A}$ is a directed set of ergodic $\Gamma$-systems (with compatible factor maps) and $k \geq 1$, then $\Poly_{\leq k}(\XX)$ is the union of the $\Poly_{\leq k}(\XX_\alpha)$ (where we embed the latter groups in the former in the obvious fashion).
	\end{itemize}
	\end{proposition}
	\begin{proof}
	The proof of $(i)$, $(ii)$, and $(iii)$ can be found in \cite[Lemma C.1]{btz}, \cite[Lemma A.35]{btz}, and \cite[Proof of Lemma 8.8]{btz} respectively. The proof of $(iv)$ can be found in \cite[Corollary B.3]{shalom3}. The proof of $(v)$ is given in \cite[Lemma 4.3(iii)]{btz}.  For $(vi)$, let $P \in \Poly_{\leq k}(\XX)$, then the Gowers--Host--Kra seminorm $\|e(P)\|_{U^{k+1}(\XX)}$ of $e(P)$ is equal to $1$.  As $\XX$ is the inverse limit of the $\XX_\alpha$, and the $U^{k+1}$ norm is controlled by the $L^{2^{k+1}}$ norm, we see that for any $\eps>0$ there exists $Q \in \M(\XX_\alpha,\T)$ for some $\alpha$ such that 
 $$\|e(P)-e(Q)\|_{U^{k+1}(\XX)}, \|e(P)-e(Q)\|_{L^2(\XX)} \leq \eps.$$
    For $\eps$ small enough, we can invoke \cite[Theorem 1.3]{eisner-tao} (shrinking $\eps$ as necessary) and also ensure that there exists $R \in \Poly_{\leq k}(\XX_\alpha)$ such that
    $$ \|e(Q)-e(R)\|_{L^2(\XX_\alpha)} \leq \eps.$$
    In particular $\|e(P)-e(R)\|_{L^2(\XX)} \leq 2\eps$.  Applying part $(i)$ and taking $\eps$ small enough, we conclude that $P$ and $R$ differ by a constant, hence $P \in \Poly_{\leq k}(\XX_\alpha)$, giving the claim.
	\end{proof}

We need some further technical lemmas. Recall that $\mathcal{M}(X,\T)$ denotes the space of measurable functions from $X$ to $\T$ where two functions are identified if they agree almost surely. We equip $\mathcal{M}(X,\T)$ with the topology of convergence of probability (or equivalently,  $L^2$-topology) and corresponding Borel $\sigma$-algebra. 

\begin{proposition}\label{C.8}
\cite[Lemma C.8]{host2005nonconventional}
Let $U$ be a compact abelian group acting freely on a probability space $\XX=(X,\mu)$. Namely, $\XX=Y\times U$ as a measure space, for some probability space $Y$, and the action of $U$ on $\XX$ is given by $V_u(y,v) = (y,v+u)$ for all $y\in Y$, $u,v\in U$. Let $u\mapsto \phi_u$ be a measurable map from $U$ to $\mathcal{M}(X,\T)$ such that $$\phi_{u+v} = \phi_u + \phi_v\circ V_u$$ for all $u,v\in U$. Then there exists $\Phi\in \mathcal{M}(X,\T)$ such that $\phi_u = \Phi\circ V_u - \Phi$.
\end{proposition}

 We have a polynomial version of Proposition \ref{C.8}:  
 
	\begin{proposition} \label{PI1}
 \cite[Proposition 8.9]{btz} 
    Let $\Gamma$ be a countable abelian group, let $j,l\geq 0$, let $U$ be a compact abelian group, and let $\XX=\YY\rtimes_\rho U$ be an ergodic $\Gamma$-system such that $\ZZ^{j-1}(\XX)$ is a factor of $\YY$ and $\rho$ is a cocycle polynomial of degree at most $j$. For any $t\in U$, let $p_t\in \Poly_{\leq l}(\XX)$ depend measurably\footnote{This measurability hypothesis, which is clearly necessary, was omitted by mistake in \cite{btz}.} on $t$ and suppose that for any $t,s\in U$ $$p_{t+s} - p_t\circ V_s - p_s =0.$$
    Then there exists $Q\in \Poly_{\leq l+j}(\XX)$ such that $\partial_t Q = p_t$ for all $t\in U$. Furthermore, we can take $Q(y,u+u_0) \coloneqq p_u(y,u_0)$ for a generic $u_0\in U$.
	\end{proposition}

We need the following construction in the next lemma. Let $\XX=(X,\mu,T)$ be an ergodic $\Gamma$-system for some countable abelian group $\Gamma$ and let $K$ be a group of automorphisms of $\XX$ such that $K$ acts freely (on the left) on $\XX$ as described in the previous proposition. Then we can form the quotient $\Gamma$-system $K \backslash \XX$ where the underlying set is the space of $K$-orbits $[x]$ (which is obtained from $X$ by an equivalence relation), equipped with the quotient $\sigma$-algebra, the pushfoward of the probability measure $\mu$ under the quotient map, and the action $T_\gamma[x]=[T_\gamma x]$. One directly verifies that $K \backslash \XX$ is a $\Gamma$-system and a factor of $\XX$.  

\begin{lemma}\label{lem-goingup}
Let $\Gamma$ be a countable abelian group, let $\XX$ be an ergodic $\Gamma$-system, let $U$ be a compact abelian group, let $\rho\colon \Gamma\times X\to U$ be an ergodic cocycle, and let $K$ be a compact connected abelian group of automorphisms of $\XX$ such that $\partial_v \rho$ is a coboundary for all $v\in K$. 
Then $K$ lifts to a compact connected abelian group $\tilde K$ of automorphisms of the group skew-product extension $\XX\rtimes_\rho U$. 

Moreover, if $K$ acts freely on $\XX$, then its lift $\tilde{K}$ acts freely on $\YY\coloneqq \XX\rtimes_\rho U$ and we have that $\tilde K \backslash \YY$ is a factor of $\YY$ which group skew-extends the factor $K \backslash \XX$ of $\XX$ by a quotient of $U$.  
\end{lemma}

\begin{proof}
See \cite[Lemmas 3.10 and 3.11]{shalom3}; the proofs are given for a special countable abelian group $\Gamma$, but the same proof works for arbitrary countable abelian groups. 
\end{proof}

\section{Nilspace theory}\label{nilspace-app}

We recall some of the theory of compact nilspaces and related structures (the various relationships are summarized in Figure \ref{fig:structures}).  We briefly recall the basic definitions:

\begin{figure}
\centering
    \[\begin{tikzcd}
    {\substack{{\mathcal D}^k(U)}} & {\substack{\text{compact } {\mathcal D}^k(U)}} \\
    {\substack{\text{filtered}\\ \text{group } G}} & {\substack{\text{filtered Polish} \\ \text{group } G}} \\
	{\substack{G/\Lambda}} & {\substack{\text{compact } G/\Lambda}} & {\substack{\text{translational} \\ \Gamma-\text{system } G/\Lambda }} \\
	{\substack{K \backslash G/\Lambda} }& {\substack{\text{compact } K \backslash G/\Lambda}} & {\substack{\text{double coset} \\  \Gamma-\text{system } K \backslash G/\Lambda }} \\
    {\substack{\text{nilspace } X}} & {\substack{\text{compact nilspace } X }} & {\substack{\Gamma-\text{system } \XX}} \\
    &{\substack{\text{cubic coupling } \\ (\mu^{[n]})_{n \geq 0}}}
    \arrow[from=1-1, to=2-1]
    \arrow[from=1-2, to=2-2]
    \arrow[from=1-2, to=1-1]
    \arrow[from=2-1, to=3-1]
    \arrow[from=2-2, to=3-2]
    \arrow[from=2-2, to=2-1]
    \arrow[from=3-2, to=3-1]	
    \arrow[from=3-3, to=3-2]	
    \arrow[from=3-1, to=4-1]	
    \arrow[from=3-2, to=4-2]	
    \arrow[from=3-3, to=4-3]
    \arrow[from=4-2, to=4-1]	
    \arrow[from=4-3, to=4-2]	
    \arrow[dashed,from=4-1, to=5-1]	
    \arrow[dashed,from=4-2, to=5-2]	
    \arrow[from=5-2, to=5-1]	
    \arrow[from=4-3, to=5-3]	
    \arrow[from=5-2, to=6-2, shift left=1]	
    \arrow[from=6-2, to=5-2, shift left=1]	
    \arrow[from=5-3, to=6-2]	
    \arrow[from=3-1, to=5-1, bend right=40]
    \arrow[from=3-2, to=5-2, bend right=40]
\end{tikzcd}\]
    \caption{Some of the different structures considered in this paper (all implicitly assumed to have degree, step, or order $k$).  Arrows from one structure to another indicate a way to generate the latter from the former (possibly after quotienting out by a group action, or applying a forgetful functor).  Dashed arrows require that $K, G, \Gamma$ verify a ``groupable'' axiom. Some minor hypotheses (e.g., metrizability) have been suppressed for brevity.  While this diagram ``morally'' commutes, verifying commutativity in practice requires some non-trivial effort.}
  \label{fig:structures}
\end{figure} 

\begin{definition}[Compact nilspace]\label{nilspace-def}\   
\begin{itemize}
\item[(i)] \cite[Definition 1.2.2]{candela} A \emph{nilspace} is a space $X$, equipped with subsets $C^n(X)$ of $X^{\{0,1\}^n}$ for every $n \geq 0$ obeying the ergodicity axiom $C^0(X)=1$, $C^1(X) = X^{\{0,1\}}$, the composition axiom that $c \circ \phi \in C^m(X)$ whenever $c \in C^n(X)$ and $\phi \colon \{0,1\}^m \to \{0,1\}^n$ is a cube morphism (as defined in \cite[Definition 1.1.1]{candela0}, and that every $n$-corner $c_*$ has at least one completion to an $n$-cube $c \in C^n(X)$, where an $n$-corner is defined to be a tuple $(c_\omega)_{\omega \in \{0,1\}^n \backslash \{1^n}$ in $X^{\{0,1\}^n \backslash \{1^n\}}$ such that the restriction to every $n-1$-face lies in $C^{n-1}(X)$.  Elements of $C^n(X)$ will be known as \emph{$n$-cubes} in $X$.  A nilspace is \emph{$k$-step} for some $k \geq 0$ if every $k+1$-corner has a unique completion to a $k+1$-cube.  If $X$ is a compact metrizable space and the $C^n(X)$ are closed subsets of $X^{\{0,1\}^n}$, we say that the nilspace $X$ is \emph{compact}.
\item[(ii)] \cite[Definition 2.2.11]{candela} A \emph{morphism} between nilspaces $X, Y$ is a map $g \colon X \to Y$ such that $g \circ c \in C^n(Y)$ whenever $c \in C^n(X)$ for some $n \geq 0$.  A morphism between compact nilspaces is a morphism of nilspaces that is also continuous.  An \emph{isomorphism} of (compact) nilspaces is a morphism that has an inverse that is also a morphism.  If the map $c \mapsto g \circ c$ from $C^n(X)$ to $C^n(Y)$ is surjective for every $n \geq 0$, we say that $Y$  is a \emph{factor} of $X$, and that $g$ is a \emph{factor map}.
\item[(iii)] If $X$ is a  nilspace and $n \geq 0$, we define the equivalence relation $\sim_n$ on $X$ by setting $x\sim_n x'$ if there exist two cubes $c,c'\in C^{n+1}(X)$ such that $c_\omega = c'_\omega$ for all $\omega\in \{0,1\}^{n+1}\backslash \{0^{n+1}\}$ and $c_{0^{n+1}} = x$, $c'_{0^{n+1}} =x'$, where $0^{n+1}=(0,0,\dots,0)$. (The nilspace axioms will guarantee that $\sim_n$ is indeed an equivalence relation.)   In \cite[Lemma 3.2.10]{candela0} it is shown that for any $k \geq 0$, ${\mathcal F}_k(X) \coloneqq X/\sim_k$ has the structure of a $k$-step nilspace, with the quotient map from $X$ to $X/\sim_k$ being a factor map.  A factor map $g \colon X \to Y$ between nilspaces is said to be a \emph{fibration}\footnote{This notion was originally referred to as a \emph{fiber surjective morphism} in \cite[Definition 3.3.7]{candela0}, but renamed as ``fibration'' in \cite[Definition 7.1]{gmv}.} if for every $n \geq 0$, $g$ maps equivalence classes of $\sim_n$ on $X$ to equivalence classes of $\sim_n$ on $Y$.
\end{itemize}
\end{definition}

Observe that if $X$ is a (compact) nilspace, then each $C^n(X)$ is also a (compact) nilspace, using the identification $(C^n(X))^m \coloneqq C^{n+m}(X)$ (the precise conventions for this identification are not important, as the composition axiom of a nilspace ensures that the spaces of cubes are preserved by permutations of the coordinates).

Key example of $k$-step nilspaces include degree $k$ nilspaces, and more generally abelian bundles; we recall the definitions.

\begin{definition}[Degree $k$ nilspaces and abelian bundles]  Let $U$ be an abelian group and $k \geq 1$.
\begin{itemize}
    \item[(i)]  \cite[\S 2.2.4]{candela0} The \emph{degree $k$ nilspace} structure ${\mathcal D}^k(U)$ on $U$ is the nilspace $U$ endowed with the cubes
$$ C^n({\mathcal D}^k(U)) \coloneqq \{ (u_\omega)_{\omega \in \{0,1\}^n}: \sum_{\omega \in F} (-1)^{|\omega|} u_\omega = 0 \forall k-\hbox{faces } F \subset \{0,1\}^n \}.$$
One can check that ${\mathcal D}^k(U)$ is a $k$-step nilspace; see \cite[\S 2.2.4]{candela0}.  If $U$ is compact metrizable, then ${\mathcal D}^k(U)$ is a compact nilspace.
    \item[(ii)] \cite[Definitions 3.2.17, 3.2.18]{candela0} More generally, a $k$-step nilspace $X$ is said to be a \emph{degree $k$ abelian bundle} over another $k-1$-step nilspace $Y$ with structure group $U$ if there is a free action $u \colon x \mapsto x+u$ of $U$ on $X$, with $Y$ isomorphic as a nilspace factor to the quotient $X/U$ (thus there is an identification of $Y$ with $X/U$ such that for every $k \geq 0$, the quotient map from $X^{\{0,1\}^k}$ to $(X/U)^{\{0,1\}^k}$ maps $C^k(X)$ surjectively onto $C^k(Y)$), and such that for any $n \geq 0$ and $c \in C^n(X)$, one has
    $$ \{ c_2 \in C^n(X): \pi^{\{0,1\}^n}(c_2) = \pi^{\{0,1\}^n}(c) \} = \{ c + c_3: c_3 \in C^n( {\mathcal D}^k(U) ) \}$$
    where $\pi^{\{0,1\}^n}: C^n(X) \to C^n(Y)$ is the projection map and $C^n( {\mathcal D}^k(U) )$ acts on $X^{\{0,1\}^n}$ in the obvious fashion.  If $X,Y$ are compact nilspaces and the free action of $U$ on $X$ is continuous, we say that the bundle is continuous.
\end{itemize}
\end{definition}

A basic structural theorem in the subject (cf., Proposition \ref{abelext}) is

\begin{theorem}[$k$-step nilspaces are towers of abelian bundles]\label{tower}\cite[Theorem 3.2.19]{candela0}, \cite[\S 2.1.1]{candela} Let $X$ be a $k$-step nilspace for some $k \geq 1$.   Then there exists abelian groups $U_1,\dots,U_k$ (known as the \emph{structure groups}) and a sequence $X_0,\dots,X_k=X$ of nilspaces (known as the \emph{factor subspaces}), with $X_0$ a point, and for each $i=1,\dots,k$, $X_i$ is an $i$-step nilspace that is a degree $i$ abelian bundle over $X_{i-1}$ with structure group $U_i$, unique up to isomorphism (of nilspaces and bundles); indeed, $X_i$ is isomorphic to ${\mathcal F}_i(X)$.  Furthermore if $X$ is a compact nilspace, the $U_i$ are compact metrizable, the bundles are continuous, and all these structures are unique up to isomorphism (of compact nilspaces and continuous bundles).
\end{theorem}

We also have a canonical Haar measure:

\begin{theorem}[Existence and uniqueness of Haar measure]\label{haar-thm}\cite[Proposition 2.2.5, Corollary 2.2.7, Proposition 2.2.11]{candela}  Let $X$ be a $k$-step compact nilspace for some $k \geq 1$, and let $X_0,\dots,X_k$ and $U_1,\dots,U_k$ be as in Theorem \ref{tower}.  Then there exists a unique Radon probability measure $\mu_X$ on $X$ with the property that for each $1 \leq i \leq k$, the pushforward $\mu_{X_i}$ of $X$ to $X_i$ is $U_i$-invariant.  The support of $\mu_X$ is all of $X$.  If $X$ is a continuous fibration of another $k$-step compact nilspace $Y$, then the Haar measure of $X$ pushes forward to the Haar measure of $Y$.
\end{theorem}

Since $C^n(X)$ is also a compact nilspace, we also have Haar measures $\mu^{[n]}_X \coloneqq \mu_{C^n(X)}$ on the cube spaces $C^n(X)$ for all $n \geq 0$.  In the language of \cite{candela-szegedy-cubic}, these measures $\mu^{[n]}_X$ have the structure of a cubic coupling; see \cite[Proposition 3.6]{candela-szegedy-cubic}.   If a space $X$ has both the structure of a $\Gamma$-system and a compact nilspace, we say that the $\Gamma$-system structure and the compact nilspace structure are \emph{compatible} if the $\sigma$-algebra of the $\Gamma$-system is the Borel $\sigma$-algebra of the compact nilspace, and the Host-Kra measures $\mu^{[n]}_X$ of the $\Gamma$-system agree\footnote{In the language of \cite{candela-szegedy-cubic}, this is asserting that the $\Gamma$-system and the compact nilspace generate the same cubic coupling; see also Figure \ref{fig:structures}.} with the Haar measures $\mu^{[n]}_X$ of the compact nilspace.  In \cite[Theorem 5.11]{candela-szegedy-cubic} it was shown that every $\Gamma$-system of order $k$ (and more generally any cubic coupling, see \cite[\S 3.6]{candela-szegedy-cubic}) is isomorphic to a model equipped with a compatible $k$-step compact nilspace structure, though in the current paper we will not directly use this fact as we will construct the compatible $k$-step nilspace structures directly for the systems under consideration.

Following \cite[Definition 2.5.1]{candela}, we call a compact $k$-step nilspace \emph{compact finite rank}, or \emph{CFR} for short, if the structure groups $Z_1,\dots,Z_k$ all have finitely generated Pontryagin duals. The nilspaces we will deal with here will usually not be CFR, but on the other hand every compact $k$-step nilspace is expressible as the inverse limit of CFR $k$-step nilspaces; see \cite[Theorem 2.7.3]{candela}.

Besides $\Gamma$-systems, another key source of nilspaces comes from filtered groups and their quotients; we now review the key definitions.

\begin{definition}[Filtered groups and their quotients]\label{def-filtration}
\text{}
\begin{itemize}
\item[(i)]
 A \emph{filtration} on a group $G = (G,\cdot)$ is a collection of subgroups $G_\bullet = (G_i)_{i=0}^\infty$ of $G$ such that $G_0 = G_1 \geq G_2\geq \dots$ and $[G_i,G_j] \subset G_{i+j}$ for all $i,j \geq 0$ where $[G_i,G_j]$ denotes the group generated by the commutators $[g,h]$ with $g\in G_i,h\in G_j$.  We refer to the pair $(G,G_\bullet)$ as a \emph{filtered group}.  A \emph{filtered subgroup} $H$ of $G$ is a subgroup $H$ with a filtration $(H_i)_{i=0}^\infty$ with $H_i \leq G_i$ for all $i$, and similarly a \emph{filtered homomorphism} $\phi \colon G \to H$ between two filtered groups $G,H$ is a homomorphism such that $\phi(G_i) \leq H_i$ for all $i$. The product $G \times H$ of two  filtered groups $G, H$ is another filtered group with filtration $(G \times H)_i = G_i \times H_i$. 
 \item[(ii)] If $G = (G,G_\bullet)$ is a filtered group, and $k \geq 0$, we define the \emph{Host--Kra group} $\HK^k(G) = \HK^k(G,G_\bullet)$ to be the subgroup of $G^{\{0,1\}^k}$ generated by the elements
$$
[g_{\omega_0}]_{\omega_0} \coloneqq \left(g_{\omega_0}^{1_{\omega \geq \omega_0}}\right)_{\omega \in \{0,1\}^k}
$$
for $\omega_0$ in $\{0,1\}^k$ and $g_{\omega_0} \in G_{|\omega_0|}$, where $g_{\omega_0}^{1_{\omega \geq \omega_0}}$ is defined to equal $g_{\omega_0}$ when $\omega \geq \omega_0$ (in the product order on $\{0,1\}^k$) and the identity $1$ otherwise.
\item[(iii)]  If $G = (G,G_\bullet)$ is a filtered group, $k \geq 0$, and $\Lambda$ is a subgroup of $G$, we define the \emph{Host--Kra space} $\HK^k(G/\Lambda)$ to be the set
$$ \HK^k(G/\Lambda) \coloneqq \pi_\Lambda^{\{0,1\}^k}( \HK^k(G) ) \subset (G/\Lambda)^{\{0,1\}^k},$$
where $\pi_\Lambda \colon G \to G/\Lambda$ is the quotient map and $\pi^{\{0,1\}^k} \colon G^{\{0,1\}^k} \to (G/\Lambda)^{\{0,1\}^k}$ is the map defined by pointwise evaluation of $\pi$, thus
$$ \pi^{\{0,1\}^k}( (h_\omega)_{\omega \in \{0,1\}^k} ) \coloneqq (\pi(h_\omega))_{\omega \in \{0,1\}^k}.$$
Note that we have a canonical identification
$$ \HK^k(G/\Lambda) \equiv \HK^k(G) / \HK^k(\Lambda)$$
and also
$$ \HK^k(\Lambda) = \HK^k(G) \cap \Lambda^{\{0,1\}^k}.$$ If $K\leq G$ is another subgroup of $G$, we define $\HK^k(K\backslash G/\Lambda)$ to be the set $\pi_{K,\Lambda}^{\{0,1\}^k}(\HK^k(G))\subseteq (K\backslash G/\Lambda)^{\{0,1\}^k}$, where $\pi_{K,\Lambda}\colon G\to K\backslash G/\Lambda$ is the (double) quotient map. 
 \end{itemize}
 \end{definition}

Every filtered group $G$ has the structure of a nilspace with $C^n(G) = \HK^n(G)$ for all $n \geq 0$, and if $G$ has degree $k$ then it is a $k$-step nilspace; see \cite[Proposition 2.2.8]{candela0}.  Furthermore, an inspection of the proof shows that for $1 \leq i \leq k$, the $i^{\mathrm{th}}$ structure group $U_i$ is given by $U_i = G_i / G_{i+1}$, and the $i^{\mathrm{th}}$ factor space is given by ${\mathcal F}_i(G) = G / G_{i+1}$ (with the obvious filtered group structures, free action, and factor maps); that is to say, the equivalence classes of $\sim_i$ are the orbits of $G_{i+1}$.

We note that if $U = (U,+)$ is an abelian group and $k \geq 1$, and we define the \emph{degree-$k$ filtration} on $U$ by setting $U_0 = U_1 =\ldots = U_k\coloneqq U$ and $U_i = \{0\}$ for $i>k$, then the associated nilspace is precisely ${\mathcal D}^k(U)$.  Because of this, we shall also abuse notation and use ${\mathcal D}^k(U)$ to refer to this filtered abelian group.  If $U$ is compact metrizable, it is not difficult to see that all the Host--Kra groups $\HK^n({\mathcal D}^k(U))$ are compact metrizable, and that $\mu^{[n]}_{{\mathcal D}^k(U)}$ is the Haar probability measure on $\HK^n({\mathcal D}^k(U))$.  Similarly, if $U_1,\dots,U_k$ are compact metrizable abelian groups, then 
$$ {\mathcal U} \coloneqq {\mathcal D}^1(U_1) \times \dots \times {\mathcal D}^k(U_k)$$
is a compact $k$-step nilspace with Haar measures $\mu^{[n]}_{\mathcal U}$ being the Haar probability measures on
$$\HK^n( {\mathcal D}^1(U_1) ) \times \dots \times \HK^n( {\mathcal D}^k(U_k) ).$$

Every group quotient $G/\Lambda$ acquires the structure of a nilspace with $C^n(G/\Lambda) = \HK^n(G/\Lambda)$; see \cite[Proposition 2.3.1]{candela0} or \cite[Proposition 2.6]{gmv}. It will be a $k$-step nilspace if $G$ is of degree $k$.  In particular, $G/\Lambda$ will be a compact nilspace if $G/\Lambda$ is compact and $\HK^n(G/\Lambda)$ is closed in $(G/\Lambda)^{\{0,1\}^n}$ for all $n$.  An inspection of the proofs show that for $1 \leq i \leq k$, the $i^{\mathrm{th}}$ structure group $U_i$ is given by $U_i = G_i / \Lambda_i G_{i+1}$ (where $\Lambda_i \coloneqq \Lambda \cap G_i$), and the $i^{\mathrm{th}}$ factor space is given by ${\mathcal F}_i(G) = G / \Lambda G_{i+1}$, 
with the obvious filtered group structures, free action, and factor maps.  (Note that since $G_{i+1}$ is normal, $\Lambda G_{i+1}$ is a subgroup of $G$.)

In this paper we will need to extend this result to double coset spaces $K \backslash G / \Lambda$.  Here an additional algebraic hypothesis on the subgroups $K, \Lambda$ is required.  We will use the following recent result from \cite{cgss-2023}.

\begin{lemma}[Criterion for double cosets to be nilspaces]\label{double-nilspace}  Let $G$ be a degree $k$ filtered group, and $K,\Lambda$ be subgroups of $G$.  Assume the ``groupable axiom''
\begin{equation}\label{kgl}
K g \Lambda \cap G_i g \Lambda = (K \cap G_i) g \Lambda
\end{equation}
for all $i \geq 0$ and $g \in G$.  Then $K \backslash G / \Lambda$ has the structure of a $k$-step nilspace if we define $C^n(K \backslash G / \Lambda) = \HK^n(K \backslash G / \Lambda)$ for all $n \geq 0$, and the quotient map from $G/\Lambda$ to $K \backslash G/\Lambda$ is a fibration.   Finally, we have ${\mathcal F}_i(K \backslash G / \Lambda) = K \backslash G / G_{i+1} \Lambda$ for all $i \geq 0$.  
\end{lemma}

\begin{proof}  See \cite[Lemma 5.6]{cgss-2023}.  The fact that $K \backslash G/\Lambda$ is $k$-step follows from the fact that there is a fibration from the $k$-step nilspace $G/\Lambda$ to $K \backslash G/\Lambda$.
\end{proof}

Finally, we record a lifting property of nilspace fibrations.

\begin{lemma}[Lifting through a fibration]\label{fibration-lift} Let $X,Y$ be nilspaces and let $\psi \colon X \to Y$ be a nilspace fibration. Let $g$ be a morphism from $\mathcal{D}^1(\Z^m)$ to $Y$ for some $m \geq 1$.  Then there exists a nilspace morphism $g'$ from $\mathcal{D}^1(\Z^m)$ to $X$ such that $\psi \circ g' = g$.
\end{lemma}

\begin{proof} This follows from \cite[Corollary A.7]{CGSS} (which proves a more general and slightly stronger statement).
\end{proof}

\section{Multilinear maps, polynomials, and divisibility}\label{appendix-div}

In general, measure-preserving dynamical systems are not divisible, not even $1$-divisible, e.g., an irrational rotation.  
Therefore we construct extensions satisfying divisibility, a property which is crucial in deriving the main structure theorems of this paper. 
In these constructions, we exploit certain properties of multilinear maps and polynomials and their connections through divisibility.  
We record these technical aspects in this appendix.  

\begin{definition}[Multilinear maps]\label{def-multilinear}
Let $G,H$ be locally compact (Hausdorff) abelian groups and let $m\geq 1$. A map $\lambda:G^m\rightarrow H$ is said to be \emph{multilinear} if it is a (group) homomorphism in each coordinate. We denote by $\ML_{m}(G,H)$ the group of continuous multilinear maps $G^m\rightarrow H$ (where we equip $G^m$ with the product topology). A multilinear map $\lambda\in\ML_{m}(G,H)$ is said to be \textit{totally symmetric} if $$\lambda(g_1,\ldots,g_m) = \lambda(g_{\sigma(1)},\ldots,g_{\sigma(m)})$$ for every permutation $\sigma$ of $\{1,\ldots,m\}$. We denote by $\SML_{m}(G,H)$ the subgroup of $\ML_{m}(G,H)$ consisting of continuous totally symmetric multilinear maps.  
\end{definition}

Note that $\SML_{m}(G,H)$, $\ML_{m}(G,H)$ are abelian groups. 
For discrete $G$, the groups $\ML_m(G,\T)$ and $\SML_{m}(G,\T)$ are isomorphic to the Pontryagin dual of the tensor product and the symmetric tensor product of $m$ copies of $G$ respectively whose definitions are recalled next. 
	\begin{definition}[Tensor products]
	Let $G$ be a discrete abelian group and let $m\geq 1$. The \emph{$m$-fold tensor product} of $G$ is an abelian group $G^{\otimes m}$ satisfying the following universal property: There exists a multilinear map $\imath:G^m\rightarrow G^{\otimes m}$ such that for every abelian group $H$ and every multilinear map $\lambda\in \text{ML}_m(G,H)$ there exists a homomorphism $\varphi : G^{\otimes m}\rightarrow H$ such that $\lambda=\varphi \circ\imath$. We will usually write $\imath(g_1,\ldots,g_m)$ as $g_1\otimes\ldots\otimes g_m$.
 Similarly, one can define a symmetric tensor product $G^{\otimes_{sym} m}$ by replacing multilinear maps by symmetric multilinear maps in the universal property.  One can also define the tensor product $G_1 \otimes G_2$ of two distinct groups $G_1,G_2$ using bilinear maps from $G_1 \times G_2$ to $H$.
	\end{definition}
        \begin{remark}
        The tensor product and the symmetric tensor product always exist and are unique up to isomorphism (as quotients of the free abelian group $\Z^{G^m}$). Note that the symmetric tensor product is a quotient of the tensor product. 
        \end{remark}

We now state a version of the familiar fact that the product of a polynomial of degree $m$ and a polynomial of degree $n$ is a polynomial of degree $m+n$.

\begin{lemma}[Tensor products of polynomials]\label{Tensor}  Let $\XX$ be a $\Gamma$-system, let $U, V$ be abelian groups, and let $f \in \Poly_{\leq m}(\XX, U)$ and $g \in \Poly_{\leq n}(\XX,V)$ for some integers $m,n$.  Then $f \otimes g \in \Poly_{\leq m+n}(\XX, U \otimes V)$.
\end{lemma}

\begin{proof}  The claim is trivial if $m < 0$ or $n < 0$, so suppose that $m,n \geq 0$ and the claim has already been proven for smaller values of $m+n$.  For any $\gamma \in \Gamma$, we have the Leibniz identity
$$ \partial_\gamma (f \otimes g) = (\partial_\gamma f) \otimes g + f \otimes (\partial_\gamma g) + (\partial_\gamma f) \otimes (\partial_\gamma g).$$
By the induction hypothesis, all terms on the right-hand side are polynomials of degree $\leq m+n-1$, and hence $f \otimes g$ is a polynomial of degree $\leq m+n$ as claimed.
\end{proof}

An abelian group $d$ is \emph{divisible} if for every $d \in D$ and $n \geq 1$ there exists $d' \in G$ such that $nd'=d$.  The following fact is classical:

\begin{lemma}[Divisible abelian groups are injective]\label{divis-inject}  Let $D$ be a divisible group.
\begin{itemize}
\item[(i)] If $H \leq G$ are nested subgroups, then every homomorphism $\phi \colon H \to D$ extends to a homomorphism from $G$ to $D$.
\item[(ii)] Every short exact sequence $0 \to D \to G \to K \to 0$ of abelian groups that starts with $D$ splits; in particular, $G$ is isomorphic to $D \times K$.
\end{itemize}
\end{lemma}

\begin{proof}  To prove $(i)$, it suffices by Zorn's lemma to establish the case when $G$ is generated by $H$ and some element $e$ not in $H$.  If $ne \not\in H$ for any $n \geq 1$, one can extend $\phi$ by selecting $\phi(e) \in D$ arbitrarily and then using the homomorphism property to specify $\phi$ on the rest of $G$; if instead $ne \in H$ for some minimal $n \geq 1$, we use the divisibility of $D$ select $\phi(e) \in D$ so that $n \phi(e) = \phi(ne)$ and then use the homomorphism property to extend $\phi$ to the rest of $G$.

To prove $(ii)$, we can identify $D$ with a subgroup of $G$.  By $(i)$, the identity homomorphism on $D$ extends to a projection homomorphism from $G$ to $D$, which provides the desired splitting.
\end{proof}

	It is a basic fact of abstract harmonic analysis that the Pontryagin dual of a discrete torsion-free abelian group is a compact divisible group, see e.g., \cite[Corollary 8.5]{hofmann}. 
        This fact generalizes to the symmetric tensor product of discrete abelian groups.

	\begin{proposition} \label{symdivdiscrete}
		Let $G$ be a countable discrete torsion-free abelian group. Then for every $m\geq 1$, the group $\SML_m(G,\T)$ is divisible. 
	\end{proposition}
	\begin{proof}
	See e.g., \cite[Proposition 3.12]{shalom2}.  
	\end{proof}
 
	We now prove a similar result for compact groups.

\begin{proposition}\label{symdiv}
Let $Z$ be a compact abelian group with divisible dual and let $A$ be a divisible group which is either countable or contains the torus $\T$ as a countable index open subgroup. Then the following hold true. 
\begin{itemize}
    \item[(i)] $\ML_k(Z,A)$ equipped with the compact-open topology is a discrete, countable, and divisible abelian group for every $k \geq 1$.
    \item[(ii)] Let $k\geq 1$. For every multilinear map $b\in \ML_k(Z,A)$ there exists an open subgroup $V\leq Z$ such that $b(s_1,\ldots,s_k)=0$ whenever $s_i\in V$ for at least one $1\leq i \leq k$. 
    \item[(iii)] $\SML_k(Z,A)$ is divisible for any $k\geq 1$. 
\end{itemize}
\end{proposition}
\begin{proof}  We begin with $(i)$.
If $A$ contains the torus $\T$ as a countable index open subgroup, then by Lemma \ref{divis-inject} $A$ is isomorphic to $\T \times A'$ for a countable discrete group $A'$.  From this we see that to prove $(i)$, it suffices to do so in the special cases when $A$ is either countable, or equal to $\T$.

We prove $(i)$ by induction on $k$. If $k=1$ and $A=\T$, then $\ML_1(Z,A)$ is the dual group, and thus a countable, discrete, abelian group which is divisible by assumption. If $A$ is discrete, then the kernel of every continuous homomorphism $\varphi:Z\rightarrow A$ is an open subgroup. Since $Z$ is a compact metrizable abelian group, there are at most countably many open subgroups (this can be seen from the fact that its dual is countable for example). Since also $Z/\ker(\varphi)$ is finite (as $Z$ is compact), we deduce that $\ML_1(Z,A)$ is countable. Now since $A$ is discrete, the compact-open topology on $\ML_1(Z,A)$ is metrizable by the metric $\tilde{d}(\varphi_1,\varphi_2)\coloneqq \sup\{d(\varphi_1(z),\varphi_2(z))\colon z\in Z\}$, where $d$ is the discrete metric on $A$, and thus is also discrete. It is left to show that $\ML_1(Z,A)$ is divisible. Let $\varphi\in \ML_1(Z,A)$ and let $n\geq 2$. Then $V\coloneqq \ker(\phi)$ is an open subgroup, and $\varphi$ descends to a homomorphism $Z/V\rightarrow A$. Since $Z/V$ is a finite abelian group, it is self-dual $\widehat{Z/V}\cong Z/V$. Let $e_1,\ldots,e_N$ be the generators of $Z/V$ and let $\hat e_1,\ldots,\hat e_N$ be the image of $e_1,\ldots,e_N$ under the isomorphism $\widehat{Z/V}\cong Z/V$. By Pontryagin duality, the surjective homomorphism $Z\rightarrow Z/V$ gives rise to an injection $\widehat{Z/V}\rightarrow \hat Z$. We can therefore identify $\widehat{Z/V}$ as a subgroup of $\hat Z$. Since we assume that $\hat Z$ is divisible, we can find $\hat w_i\in\hat Z$ such that $n\cdot \hat w_i =\hat e_i$ for each $1\leq i \leq N$. Let $W$ be the Pontryagin dual of the group $\hat W$ generated by $\hat w_1,\ldots,\hat w_N$, and let $w_1,\ldots,w_N\in W$ be the elements which correspond to $\hat w_1,\ldots,\hat w_N$ under the identification $\hat W \cong W$ (note that $\hat W$ is finite in $\hat Z$ as all its elements have bounded torsion). We have that $n\cdot w_i = e_i$ for all $1\leq i \leq N$. Moreover, by construction we have that $W$ is a quotient of $Z$ and $Z/V$ is a quotient of $W$. Now, let $\tilde{\varphi}:W\rightarrow A$ be the unique homomorphism which satisfies $\tilde{\varphi}(w_i) = \varphi(e_i)/n$ where $\varphi(e_i)/n\in A$ is any element satisfying $n\cdot \varphi(e_i)/n=\varphi(z_i)$ for all $i=1,\ldots,N$ where $z_i\in Z$ is such that $z_i +V=e_i$. Lifting $\tilde{\varphi}$ to $Z$ we see that $n\cdot\tilde{\varphi} = \varphi$. This completes the proof of $(i)$ in the case where $k=1$. Let $k\geq 2$ and assume inductively that $(i)$ holds for all smaller values of $k$. Note that $\ML_k(Z,A) = \Hom(Z, \ML_{k-1}(Z,A))$. Thus setting $\tilde{A}:=\ML_{k-1}(Z,A)$, we have that $\tilde{A}$ is discrete and divisible by induction hypothesis, and $(i)$ for $k$ follows from the case $k=1$.

Now let $b \colon Z^k\rightarrow A$ be as in $(ii)$. We already proved $(ii)$ in the case where $k=1$. Let $k\geq 2$, for each coordinate we get a map $b_i \colon Z\rightarrow \ML_{k-1}(Z,A)$. By $(i)$, $\ML_{k-1}(Z,A)$ is discrete. Therefore, the kernel $V_i$ of $b_i$ is an open subgroup. The intersection  $V \coloneqq \bigcap_{i=1}^k V_i$ satisfies the property in the claim.

To prove $(iii)$, let $n\geq 2$ and let $b:Z^k\rightarrow A$ be a totally symmetric multilinear map. We need to show that there exists a totally symmetric multilinear map $b' \colon Z^k\rightarrow A$ such that $n\cdot b'=b$. We repeat the argument from the case $k=1$ in $(i)$ with some modifications. By $(ii)$, we can find an open subgroup $V$ such that $b$ descends to a totally symmetric multilinear map on $Z/V$. Construct $W$ just as before, and let $w_1,\ldots,w_N$ be the corresponding basis. Now set $b'(w_{i_1},\ldots,w_{i_k}) = \frac{b(z_{i_1},\ldots,z_{i_k})}{n}$ and extend $b'$ (uniquely) to a multilinear map. Since $b$ is symmetric so is $b'$, and by lifting $b$ to $Z$ we have that $n\cdot b = b'$ as required.
\end{proof}

	Our next goal is to study the relationship between polynomials and multilinear maps. Let $Z$ be an ergodic $\Gamma$-rotational system for some countable abelian group $\Gamma$, let $k\geq 1$, and let $P \in \Poly_k(Z)$. Then $b(s_1,\ldots,s_k) \coloneqq \partial_{s_1}\ldots\partial_{s_k}P$ is a totally symmetric multilinear map. We can reverse this implication when $Z$ has divisible Pontryagin dual:
 
\begin{lemma}[Symmetric multilinear maps are polynomial] \label{SML=Polydiv} Let $Z$ be an ergodic $\Gamma$-rotational system for some countable abelian group $\Gamma$ such that the dual group $\hat{Z}$ is divisible. Let $k\geq 1$ and let $b\in \SML_k(Z,\T)$. Then there exists a polynomial $Q\in \mathrm{Poly}_{\leq k}(Z,\T)$ such that
$$b(s_1,\ldots,s_k) = \partial_{s_1}\ldots\partial_{s_k} Q$$
\end{lemma}

\begin{proof}
We give a precise formula for $Q$. By Proposition \ref{symdiv}(iii), we can find a totally symmetric multilinear map $\mu\in \SML_k(Z,\T)$ such that $k! \cdot \mu = b$. Set \begin{equation}
    Q(s) := \mu(s,s,\ldots,s).
\end{equation}
We claim by induction on $k$ that $Q$ is a polynomial of degree $k$ and $\partial_{s_1}\ldots\partial_{s_k} Q = k! \mu(s_1,\ldots,s_k)$. If $k=1$, then $Q(s)=\mu(s)$ is a homomorphism. Homomorphisms $Z\rightarrow\T$ are polynomials of degree $1$ and $\partial_{s_1} Q(s) =  \mu(s+s_1)-\mu(s) = \mu(s_1)$ as required. Let $k\geq 2$ and 
observe that
$$\partial_{s_k} Q(s) = \mu(s+s_k,s+s_k,\ldots,s+s_k)-\mu(s,s,\ldots,s) = \sum_{i=0}^{k-1} \binom{k}{i} \mu_i(s,s_k)$$ where $\mu_i(s,s_k) = \mu(s,s,\ldots,s,s_k,s_k,\ldots,s_k)$ where $s$ appears $i$ times and $s_k$ appears $(k-i)$-times. By the induction hypothesis, $\mu_i$ is a polynomial of degree $i$ and so
$$ \partial_{s_1}\ldots\partial_{s_k} Q(s) =  \partial_{s_1}\ldots\partial_{s_{k-1}} \partial_{s_k} Q(s) = k\cdot \partial_{s_1}\ldots\partial_{s_{k-1}} \mu_{k-1}(s,s_k).$$
The claim follows since by the induction hypothesis we have 
$$\partial_{s_1}\ldots\partial_{s_{k-1}} \mu_{k-1}(s,s_k)=(k-1)! \mu(s_1,\ldots,s_k).$$
\end{proof}

\begin{theorem}\label{divisible}
 Let $Z$ be an ergodic $\Gamma$-rotational system for some countable abelian group $\Gamma$ such that the dual group $\hat{Z}$ is divisible. Then for  every $k\geq 1$, the group  $\mathrm{Poly}_{\leq k}(Z,\T)$ is divisible.
\end{theorem}

\begin{proof}
We prove this theorem by induction on $k$. By ergodicity, $\mathrm{Poly}_{\leq 0}(Z,\T)\cong \T$ and $\T$ is divisible. For the case $k=1$, we have 
$$\mathrm{Poly}_{\leq 1}(Z,\T)\cong \T\oplus \Poly_{\leq 1}(Z,\T)/\mathrm{Poly}_{\leq 0}(Z,\T)\cong \T\oplus \hat Z,$$
and the claim follows from the assumption that $\hat Z$ is divisible. 
Now let $k\geq 2$ and let $Q:Z\rightarrow\T$ be a polynomial of degree at most $k$. Then $b(s_1,\ldots,s_k):=\partial_{s_1}\ldots\partial_{s_k}Q$ is a totally symmetric multilinear map. By Proposition \ref{symdiv}(iii), we can find a totally symmetric multilinear map $b':Z^k\rightarrow \T$ such that $n\cdot b'=b$ for any $n\geq 1$. By Lemma \ref{SML=Polydiv}, we can find a polynomial $Q':Z\rightarrow\T$ of degree at most $k$ such that $\partial_{s_1}\ldots\partial_{s_k} Q' = b'(s_1,\ldots,s_k)$. We deduce that $Q-nQ'$ is a polynomial of degree at most $k-1$. By induction hypothesis, we can find a polynomial $Q''$ of degree at most $k-1$ such that $n\cdot Q'' = Q-nQ'$. Letting $P=Q'+Q''$, we see that $n\cdot P = Q$.
\end{proof}
\section{An Abramov system that is not Weyl}\label{appendix-d}

In this appendix we prove Proposition \ref{second-counter}.  The system $\YY$ we will construct will be a skew-product
$$ \YY \coloneqq \ZZ \rtimes_\rho \Z/8\Z$$
where
$\ZZ$ is the rotational $\F_2^\omega$-system $\ZZ = \prod_{n=1}^\infty \Z/2\Z$ with the action given by
$$(T^\gamma z)_i = z_i+\gamma_i \mod 2
$$
where $z=(z_1,z_2, \ldots)$, and $\rho \colon \Gamma \to \M( \ZZ, \Z/8\Z)$ is the cocycle
$$
\rho_\gamma(z) \coloneqq \sum_{i=1}^\infty (|z_i+\gamma_i| - |z_i|) \mod 8
$$
where $|\cdot| \colon \F_2 \to \{0,1\}$ is the map defined by $|0|=0$ and $|1|=1$.  We note that the infinite sum is well defined as all but finitely many of the $\gamma_i$'s are zero.  It is easy to see that $\rho$ is a cocycle. 

The $\F_2^\omega$-system $\YY$ has a factor
$$ \XX \coloneqq \ZZ \rtimes_{\rho \mod 4} \Z/4\Z$$
with factor map $(z,t) \mapsto (z, t \mod 4)$.  One easily verifies that
$$
\rho_\gamma(z) = \sum_{i=1}^\infty |\gamma_i|(1 + 2|z_i|) \mod 4
$$
and so $\XX$ is precisely the system constructed in \cite[Appendix E]{tz-lowchar}.  In particular, $\XX$ is an ergodic Weyl $\F_2^\omega$-system of order $2$, and the function $\iota \in \M(\XX,\T)$ defined by $\iota(z,t) \coloneqq \frac{t}{4}$ is a polynomial of degree $2$ that does not admit a square root in $\XX$ of degree $3$, that is to say there does not exist $Q \in \Poly_{\leq 3}(\XX)$ with $2Q = \iota$.

We can view $\YY$ as an extension
$$ \YY = \XX \rtimes_\sigma \Z/2\Z$$
of $\XX$, where $\sigma \colon \Gamma \to \M(\XX\rightarrow \Z/2\Z)$ is defined by the formula
\begin{equation}\label{sig}
\sigma_\gamma(z,t) = \frac{\rho_\gamma(z) - \partial_\gamma F(z,t)}{4}
\end{equation}
and $F(z,i \mod 4) \coloneqq i \mod 8$ for $i=0,1,2,3$ (so that in particular $F(z,t) \mod 4 = t$ and $\partial_\gamma F \mod 4 = \rho_\gamma$).  One easily verifies that $\sigma$ is a $\Z/2\Z$-valued cocycle.

We observe that $\frac{\sigma}{2}$ is not a $\T$-valued coboundary on $\XX$.  For if it were, there must exist $G \in \M(\XX,\T)$ such that
$$ \frac{\rho_\gamma(z)}{8} = \partial_\gamma G(z,t)$$
for all $\gamma \in \F_2^\omega$ and almost all $(z,t) \in \XX$.  If we set $\gamma = e_n$ to be a generator of $\F_2^\omega$ and send $n \to \infty$, then the right-hand side converges strongly to zero but the left-hand side does not.  Thus $\frac{\sigma}{2}$ is not a $\T$-valued coboundary.
By \cite[Theorem 2.4(ii)]{jst} this implies that $\YY = \XX \rtimes_\sigma \Z/2\Z$ is ergodic.

Observe that $\rho_\gamma$ is a polynomial of degree $2$ on $\ZZ$ for every $\gamma$, thus the coordinate map $(z,t) \mapsto t$ is a polynomial of degree three.  Also the coordinate maps $(z,t) \mapsto z_i$ are polynomials of degree $1$.  From Fourier expansion we conclude that the phase polynoimals of degree at most $3$ span a dense subspace of $L^2(\YY)$, so $\YY$ is Abramov of order $3$.

Next we claim $\YY$ is not of order $2$.  If we let $\phi \in \M(\YY,\T)$ be the function $\phi(z,t) \coloneqq \frac{t}{8}$, then direct computation using $\partial_\gamma |z_i| = - 2 |\gamma_i| |z_i|$  shows that
$$ \partial_{\gamma} \partial_{\gamma'} \partial_{\gamma''} \phi(z,t) = \sum_i \frac{|\gamma_i| |\gamma'_i| |\gamma''_i|}{2} \mod 1.$$
From this one easily verifies that $e(\phi)$ has vanishing $U^3$ Gowers--Host--Kra seminorm and so must be orthogonal to $\ZZ^2(\YY)$, and so $\YY$ cannot be of order $2$.

The Conze--Lesigne factor $\ZZ^2(\YY)$ is now a strict factor of $\YY$ that contains the order $2$ factor $\XX$.  By Mackey theory, we must therefore have $\ZZ^2(\YY) = \XX$.  If $\YY$ was Weyl of order $3$, then by Proposition \ref{prop-invlim}(ii), $\sigma$ would have to be cohomologous (as a $\Z/2\Z$-valued cocycle) to a polynomial of degree $2$, thus there exists $G \in \M(\XX,\Z/2\Z)$ and $q \colon \F_2^\omega \to \Poly_{\leq 2}(\XX, \Z/2\Z)$ such that
$$ \frac{\rho_\gamma(z) - \partial_\gamma F(z,t)}{4} = q_\gamma(z,t) + \partial_\gamma G(z,t).$$
In particular, 
$$\partial_\gamma \left( \frac{F(z,t)}{8} + \frac{G(z,t)}{2} \right) = \frac{\rho_\gamma(z)}{8} - \frac{q_\gamma(z,t)}{2}$$
is a polynomial of degree $2$ for every $\gamma$, thus $Q \coloneqq \frac{F}{8} + \frac{G}{2}$ is a polynomial of degree $3$.  But 
$$ 2Q = \frac{F}{4} = \iota,$$
contradicting the results of \cite[Appendix E]{tz-lowchar}.  Thus $\YY$ is not Weyl of order $3$, and the proof of Proposition \ref{second-counter} is complete.

\begin{remark}  One can view $\YY$ as the ergodic system arising from applying the Furstenberg correspondence principle to the finite vector spaces $\F_2^n$ and to the functions $f \colon \F_2^n \to \T$ defined by
$$ f(x_1,\dots,x_n) \coloneqq \frac{\sum_{i=1}^n |x_i|}{8} \mod 1;$$
these are cubic functions that cannot be expressed efficiently as the sum of a classical cubic polynomials and function of bounded quadratic complexity, which is the basic reason why they do not generate a Weyl system.  This function also arises as counterexample to the ``classical'' formulation of the inverse conjecture for the Gowers norms on finite vector spaces: see \cite{green-tao-finite}, \cite{lms}.
\end{remark}

\bibliographystyle{abbrv}
\bibliography{bibliography}

\end{document}